\documentclass[a4paper]{amsart}
\usepackage{amsmath,amsthm,amssymb,latexsym,epic,bbm,comment,yfonts,mathrsfs}
\usepackage{graphicx,enumerate,stmaryrd}
\usepackage[all]{xy}
\xyoption{2cell}

\allowdisplaybreaks

\usepackage[active]{srcltx}
\usepackage[parfill]{parskip}

\newtheorem{theorem}{Theorem}
\newtheorem{lemma}[theorem]{Lemma}
\newtheorem{remark}[theorem]{Remark}
\newtheorem{corollary}[theorem]{Corollary}
\newtheorem{proposition}[theorem]{Proposition}

\newtheorem{definition}[theorem]{Definition}

\newcommand{\tto}{\twoheadrightarrow}

\newcommand{\LL}{\mathbf{L}}
\newcommand{\RR}{\mathbf{R}}

\newcommand{\Hom}{{\mathrm{Hom}}}
\newcommand{\Ext}{{\mathrm{Ext}}}
\newcommand{\ext}{{\mathrm{ext}}}

\newcommand{\add}{{\mathrm{add}}}
\newcommand{\pd}{{\mathrm{pd}}}
\newcommand{\gd}{{\mathrm{gld}\,}}
\newcommand{\fd}{{\mathrm{fnd}\,}}
\newcommand{\gl}{{\mathrm{gl}}}
\newcommand{\eins}{\leavevmode\hbox{\small1\kern-3.8pt\normalsize1}}
\newcommand{\op}{{\mathrm{op}}}

\newcommand{\Res}{{\rm Res}^{\fg}_{\fg_{\oa}} }

\newcommand{\Ind}{{\rm Ind}^{\fg}_{\fg_{\oa}} }

\newcommand{\aaa}{\mathtt{a}}
\newcommand{\len}{\mathtt{l}}

\newcommand{\llll}{\mathtt{s}_\lambda}
\newcommand{\dddd}{\mathtt{d}_\lambda}

\newcommand{\intdom}{\Lambda^+_{\mathrm{int}}}

\newcommand{\sO}{\mbox{{\it \textbf{O}}}}

\newcommand{\db}{\dot{b}}
\newcommand{\dv}{\dot{v}}

\newcommand{\mN}{\mathbb{N}}

\newcommand{\mZ}{\mathbb{Z}}

\newcommand{\dd}{{\mathbf{d}}}

\newcommand{\card}{{\rm card}}

\newcommand{\fg}{{\mathfrak g}}
\newcommand{\fh}{{\mathfrak h}}

\newcommand{\fl}{{\mathfrak l}}
\newcommand{\fn}{{\mathfrak n}}

\newcommand{\fq}{{\mathfrak q}}

\newcommand{\fu}{{\mathfrak u}}

\font\sc=rsfs10
\newcommand{\fancyS}{\sc\mbox{S}\hspace{1.0pt}}

\newcommand{\cD}{\mathcal{D}}
\newcommand{\cM}{\mathcal{M}}

\newcommand{\cC}{\mathcal{C}}
\newcommand{\cR}{\mathcal{R}}
\newcommand{\cK}{\mathcal{K}}

\newcommand{\cO}{\mathcal{O}}

\newcommand{\cN}{\mathcal{N}}

\newcommand{\cA}{\mathcal{A}}

\newcommand{\cL}{\mathcal{L}}

\newcommand{\cZ}{\mathcal{Z}}

\newcommand{\cX}{\mathcal{X}}
\newcommand{\cI}{\mathcal{I}}
\newcommand{\cT}{\mathcal{T}}

\newcommand{\oa}{\bar{0}}
\newcommand{\ob}{\bar{1}}

\newcommand{\End}{{\rm End}}
\newcommand{\Id}{{\rm Id}}

\newcommand{\g}{{\mathfrak{g}}}

\begin{document}
%%%%%%%%%%%%%%%%%%%%%%%%%%%%%%%%%%%%%%%%%%%%%%%%%%%%%%%%

\title[Homological properties of $\cO$. IV]{Some homological properties\\ of category~$\cO$. IV}

\author{Kevin Coulembier and Volodymyr Mazorchuk}
\date{}

\begin{abstract}
We study projective dimension and graded length of structural modules in 
parabolic-singular blocks of the BGG category $\mathcal{O}$. Some of these are calculated
explicitly, others are expressed in terms of two functions. We also obtain several partial results and
estimates for these two functions and relate them to monotonicity properties for quasi-hereditary algebras. The results are then
applied to study blocks of $\mathcal{O}$ in the context of Guichardet categories, in 
particular, we show that blocks of $\mathcal{O}$ are not always weakly Guichardet.
\end{abstract}

\maketitle


\vspace{5mm}

\section{Introduction}\label{section1}

Let $\mathfrak{g}$ be a semi-simple complex finite dimensional Lie algebra with a fixed
triangular decomposition $\mathfrak{n}^-\oplus \mathfrak{h}\oplus\mathfrak{n}^+$.
The corresponding BGG category $\mathcal{O}$ from \cite{BGG} and its parabolic generalisations from \cite{RC}
are fundamental objects of study in modern representation theory with
numerous applications to, among others, algebra, topology and combinatorics. These categories have many
nice properties and symmetries. In particular, they form the original motivating example for the general 
definition of a {\em highest weight category} in \cite{CPS}. As a highest weight category, (parabolic)
category  $\mathcal{O}$ has various classes of {\em structural} objects, {\it viz.}: simple, injective, 
projective, standard, costandard and tilting (=cotilting) objects. A general natural question,
for arbitrary highest weight categories, is what the projective dimensions
of these objects are. In the preliminaries we give some overview of the literature on this subject. 
The first two papers \cite{SHPO1,SHPO2} in the
present series initiated the study of the projective dimension of these structural objects for $\cO$, by 
determining them for the principal block~$\cO_0$ of the original ({\it i.e.} non-parabolic) category $\mathcal{O}$.

Structural modules in $\cO_0$ are naturally indexed by elements in 
the Weyl group~$W$ of $\mathfrak{g}$. In most of the cases, the projective dimension is given in terms of 
the usual length function~$\len$ for $W$ (and some of these answers go back to the original
paper \cite{BGG}). However, for injective and tilting module the answer turns out to be 
significantly more complicated and requires the full power of Kazhdan-Lusztig (KL) combinatorics. 
For these structural modules, the answer is given in terms of Lusztig's $\aaa$-function on~$W$,
defined in  \cite{Lu1, Lu2}. A summary of the main results from \cite{SHPO1,SHPO2} is given in 
the left column of the following table, where $w_0$ denotes, as usual, the longest element in~$W$.

\begin{center}
\begin{tabular}{ | l | l | }
\multicolumn{2}{c }{The principal block $\cO_0$.}\\
\hline
projective dimensions & graded lengths \\ \hline
$\pd\, L(x)=2\len(w_0)-\len(x)$ & $\gl\,  L(x)=0$   \\ \hline
$\pd\,  \Delta(x)=\len(x)$ & $ \gl\,  \Delta(x)=\len(w_0)-\len(x)$ \\ \hline
$\pd\,  \nabla(x)=2\len(w_0)-\len(x)$ & $\gl\,  \nabla(x)=\len(w_0)-\len(x)$ \\
\hline
$\pd\,  P(x)=0$ & $\gl\,  P(x)=\len(w_0)+\len(x)$  \\
\hline
$\pd\,  I(x)=2\aaa(w_0x)$ & $\gl\,  I(x)=\len(w_0)+\len(x) $  \\  
\hline
$\pd\,  T(x)=\aaa(x)$ & $\gl\,  T(x)=2\len(w_0)-2\len(x)$  \\
\hline
\end{tabular}
\end{center}
Consequently, the global dimension of $\cO_0$ is $2\len(w_0)$, see also \cite{BGG}. 

The principal block $\cO_0$ is Koszul and hence all structural modules in this block
are gradable with respect to the Koszul $\mathbb{Z}$-grading. This raises the natural 
question of determining the corresponding graded length for these modules. 
For $\cO_0$, this is a standard exercise (which also can be derived from the results
of \cite{Irving,Irving2}) and the answer is recorded in the right column of the 
above table. Note that we use the convention that the graded length of a module 
concentrated in a single degree is zero. Some other papers, for example \cite{SHPO2}, 
use the convention that the graded length of a module concentrated in a single degree is one.

The main aim of the present paper is to study both the projective dimension and the graded
length for all structural modules in all (in particular, singular) blocks of the parabolic
category $\mathcal{O}$. An important motivation for this study stems from the third paper
\cite{SHPO3} of this series where the question of projective dimension for simple objects in singular blocks
of $\mathcal{O}$ naturally appeared during the study of blocks of $\mathcal{O}$ in the context
of Guichardet categories in the sense of  \cite{Fu,Ga}. Another concrete motivation comes from 
the open question of classification of blocks of category $\cO$ for Lie superalgebras, see \cite{Brundan},
and the approach to that question via {\it i.a.} projective dimensions in \cite{CS}. We already apply
our results in this paper to these problems.

To be able to present our results, we need some notation. For two integral dominant weights $\lambda$ and 
$\mu$, we consider the parabolic-singular block $\cO^\mu_\lambda$ where the singularity of the
block is determined by $\lambda$, while the parabolicity is determined by $\mu$ in the usual way,
see for example \cite{Back}. For $\mu=0$, we recover the usual category~$\mathcal{O}$.
Let $X_{\lambda}$ denote the set of longest representatives in $W$
for cosets $W/W_{\lambda}$, where $W_{\lambda}$ is the stabiliser of~$\lambda$ with respect to the
dot-action of $W$. Then elements in~$X_{\lambda}$ naturally index isomorphism classes of simple
object in the corresponding (singular) block~$\cO_\lambda$ of the usual category~$\mathcal{O}$.

We define maps $\llll$ and $\dddd$ from  $X_\lambda$ to $\{0,1,2,\dots\}$ by 
\begin{displaymath}
\llll(x)=\pd_{\cO_\lambda}L(x\cdot\lambda)\quad\text{ and }\quad\dddd(x)=\pd_{\cO_\lambda}\Delta(x\cdot\lambda). 
\end{displaymath}
Note that, by the above table, we have  $\mathtt{s}_0(x)=2\len(w_0)-\len(x)$ and $\mathtt{d}_0(x)=\len(x)$. 
Our first collection of main results expresses all projective dimensions and 
graded lengths of structural modules in $\cO^\mu_\lambda$ in terms of $\llll$, $\dddd$, $\len$ and $\aaa$ as follows
(here $x\in X_{\lambda}$ is such that it survives in $\cO_\lambda^\mu$ and 
$\pd_{\cO_\lambda^\mu}$ is abbreviated by $\pd$, not to confuse with 
$\pd=\pd_{\cO}$ in the previous table):
\begin{center}
\begin{tabular}{ | l | l | }
\multicolumn{2}{c }{The general block $\cO^\mu_\lambda$.}\\
\hline
projective dimensions & graded lengths \\ \hline
$\pd\, L(x\cdot\lambda)=\llll(x)-2\len(w_0^\mu)$ & $\gl\, L(x\cdot\lambda)=0$   \\ \hline
$\pd\, \Delta^\mu(x\cdot\lambda)=\dddd(w_0^\mu x)-\len(w_0^\mu)$ & 
$ \gl\, \Delta^\mu(x\cdot\lambda)=\mathtt{d}_\mu(w_0w_0^\lambda x^{-1})-\len(w_0^\lambda)$ \\ \hline
$\pd\, \nabla^\mu(x\cdot\lambda)$ & 
$\gl\, \nabla^\mu(x\cdot\lambda)=\mathtt{d}_\mu(w_0w_0^\lambda x^{-1})-\len(w_0^\lambda)$ \\
$=\mathtt{d}_{\lambda}(w_0xw_0^\lambda)+\aaa(w_0w_0^\lambda)-2\aaa(w_0^\mu)$ &  \\
\hline
$\pd\, P^\mu(x\cdot\lambda)=0$ & $\gl\, P^\mu(x\cdot\lambda)=\mathtt{s}_\mu(w_0 x^{-1})-2\len(w_0^\lambda)$  \\
\hline
$\pd\, I^\mu(x\cdot\lambda)=2\aaa(w_0 x )-2\aaa(w_0^\mu)$ & 
$\gl\, I^\mu(x\cdot\lambda)=\mathtt{s}_\mu(w_0 x^{-1})-2\len(w_0^\lambda)$  \\  
\hline
$\pd\, T^\mu(x\cdot\lambda)=\aaa(w_0^\mu x w_0^\lambda)-\aaa(w_0^\mu)$ & $\gl\, T^\mu(x\cdot\lambda)$  \\
& $=2\left(\mathtt{s}_\mu(w_0^\lambda x^{-1}w_0^\mu)-\aaa(w_0^\lambda)-\aaa(w_0 w_0^\mu)\right)$  \\
\hline
\end{tabular}
\end{center}
Consequently, the global dimension of  $\cO_\lambda^\mu$ equals
$2\aaa(w_0w_0^\lambda)-2\aaa(w_0^\mu)$. In particular, the above table
determines all projective dimensions either explicitly, or implicitly in terms of the  
KLV polynomials, see \cite{Vogan, Humphreys, Ir4, CPS2}, as these polynomials determine $\llll$ and $\dddd$. 
The connection to  KLV polynomials is justified by the validity of the KL conjecture, see \cite{BB, KL}.
However, these polynomials can only be computed using a recursive algorithm, in general. 
Note that, {\it a priori}, the projective  dimensions of costandard, injective and tilting 
modules are not even implicitly determined in terms of the KLV polynomials. 
Another consequence of the above table is that all projective dimensions in regular blocks of 
parabolic category $\cO$ and all graded lengths in arbitrary blocks of non-parabolic 
category are explicitly determined.

We also obtain several partial results and estimates concerning 
$\mathtt{s}_\lambda$ and $\mathtt{d}_\lambda$, see Propositions~\ref{maxpdL}, \ref{ineq} and~\ref{maxpdD}, 
and apply these results to calculate the functions $\mathtt{s}_\lambda$ and $\mathtt{d}_\lambda$ 
for large classes of cases. In particular, we obtain many examples by connecting the results in \cite{CIS} 
with the $\aaa$-function. To illustrate the difficulty in determining the functions $\mathtt{s}_\lambda$ 
and $\mathtt{d}_\lambda$ in full generality, we briefly review some of our results and examples. 
For arbitrary $\fg$, $\lambda$ and all $x\in X_\lambda$, we have
\begin{displaymath}
\mathtt{s}_\lambda(x)\le \len(w_0 x)+\aaa(w_0w_0^\lambda)\quad\mbox{and}\quad 
\mathtt{d}_\lambda(x)\le \len( xw_0^\lambda),  
\end{displaymath}
where these estimates become equalities when $\lambda=0$. In general, these bounds are far 
from being strict. An extremal  case is when $\lambda$ is such that the algebra, generated by the simple 
roots for which $\lambda$ is singular, forms a  hermitian symmetric pair with $\fg$. In this case we prove that
\begin{displaymath}
\mathtt{s}_\lambda(x)= \aaa(w_0 x)+\aaa(w_0w_0^\lambda)\quad\mbox{and}\quad 
\mathtt{d}_\lambda(x)= \aaa( xw_0^\lambda). 
\end{displaymath}
Moreover, for $\fg=\mathfrak{sl}(n)$, we find that, for arbitrary $\lambda$, the values 
of $\mathtt{s}_\lambda$ vary between the estimate and the above case:
\begin{displaymath}
\aaa(w_0 x)+\aaa(w_0w_0^\lambda)\;\le\;\mathtt{s}_\lambda(x)\;\le\; \len(w_0 x)+\aaa(w_0w_0^\lambda). 
\end{displaymath}
By the above discussion, the lower bound is an equality when the singular 
Weyl group is a maximal Coxeter subgroup of $W$, while the upper bound is an equality when
the singular Weyl group is trivial. In Section \ref{secFam} we use our general results to 
calculate $\mathtt{s}_\lambda$ and $\mathtt{d}_\lambda$, for $\fg=\mathfrak{sl}(n)$ and a 
weight $\lambda$ for which the singular Weyl group is isomorphic to $S_{n-2}$. This sheds some light 
on the general intricate principle which determines the projective dimensions in between maximal 
and trivial Coxeter subgroups and leads to a modest ansatz to what a general description of $\llll$ might be.
 
The second collection of main results is concerned with certain {\em monotonicity properties} of
the functions $\llll$ and $\dddd$ in the context of quasi-hereditary algebras and their relation to Guichardet categories. 
Whereas the projective dimensions of simple, standard and costandard modules in~$\cO_0$ vary 
strictly monotonically along the Bruhat order, it turns out that the corresponding property 
can fail dramatically for singular blocks, as we illustrate by examples in this paper. 
Motivated by this observation, we define several monotonicity properties for various invariants
of quasi-hereditary algebras and obtain strong connections between them. These connections 
are even stronger in the specific example of parabolic category~$\cO$. Consequently, we can 
return to the question of our interest (the functions $\mathtt{s}_\lambda$ and $\mathtt{d}_\lambda$) 
and define, for any block in category~$\cO$, a unique concept of
monotonicity, based on the projective dimension of standard modules. We have increasingly 
strong conditions on a block which we call {\em almost monotone}, {\em weakly monotone} 
and {\em strictly monotone}. Regular blocks are always strictly monotone. When a block 
$\cO_\lambda$ is almost monotone, we prove that the corresponding functions $\dddd$ and $\llll$ satisfy
\begin{equation}
\label{eqintro}\llll(x)=\dddd(w_0 x w_0^\lambda)+\aaa(w_0w_0^\lambda),
\qquad\quad\text{ for all } \quad x\in X_\lambda.
\end{equation}
In particular, we prove that, in the case of a hermitian symmetric pair, the block $\cO_\lambda$ 
is always weakly monotone.  Equation~\eqref{eqintro} was then used to determine 
$\mathtt{s}_\lambda$ from $\mathtt{d}_\lambda$, immediately demonstrating its usefulness.
We also prove that a weakly monotone block is weakly Guichardet and a strictly monotone block 
is strongly Guichardet. 

As mentioned above, we show that blocks are not always almost monotone. Moreover, we prove 
that equation~\eqref{eqintro} is not true for some $\lambda$. We also prove that blocks 
in category~$\cO$ are not always weakly Guichardet, disproving \cite[Conjecture~2.3]{Fu}. 
In \cite[Section 6.2]{SHPO3}, we already proved that blocks in category~$\cO$ are not 
always strongly Guichardet.

The significant breaking of monotonicity does not occur for low-rank cases. In particular, 
all blocks of category~$\cO$ for $\mathfrak{sl}(n)$ are strictly monotone for $n=2$, 
weakly monotone for $n\le 3$ and almost monotone for $n\le 4$.

The paper is organised as follows. In Section~\ref{secprel} we collect all necessary
preliminaries. In Section~\ref{section3} we discuss the notions of projective dimension 
and graded length in derived categories. Section~\ref{section4} is devoted to the study
of projective dimensions in parabolic category~$\cO$. We also show that our results do
not extend to the generalisations $\cO^{\widehat{\mathbf{R}}}$ of parabolic category $\cO$, 
introduced in \cite{cell}, as we prove that these can have infinite global dimension. 
 In Section~\ref{section5} we determine
the global dimension of all blocks of parabolic category~$\cO$. Section~\ref{secconnec} studies
several connections between the projective dimensions and graded lengths. Section~\ref{section7}
contains our main results on projective dimensions of structural modules in parabolic-singular
category $\cO$. 
In Section~\ref{section9} we investigate various monotonicity properties for 
invariants of quasi-hereditary algebras. In Section~\ref{secHS} we deal with the case
of a hermitian symmetric pair. In Section \ref{examsec2} we fully determine projective
dimensions in a specific block for $\mathfrak{sl}(n)$ where the singularity is almost maximal 
and add some discussion towards a full solution for the function $\llll$. 
The projective dimensions of all structural modules for all blocks in category~$\cO$ for 
$\mathfrak{sl}(4)$, as well as the KLV polynomials, are obtained in Section~\ref{blocks4}, which provides in particular
an example which is not weakly Guichardet. We work out some
application of some of our results to Lie {\em superalgebras} in Section~\ref{sectionsuper}. 
In Section~\ref{secOpen} we conclude the paper with some open questions which naturally arose
in the paper, besides the obvious main questions of full description of $\llll$ and $\dddd$.

%%%%%%%%%%%%%%%%%%%%%%%%%%%%%%%%%%%%%%%%%%%%%%%%%%%%%%%%%%%%%%%%%%%%%%%%%%%%%%%%%%%%%%%%

\section{Preliminaries}\label{secprel}

We set $\mN=\{0,1,2,\cdots\}$. We work over $\mathbb{C}$. Unless explicitly stated otherwise, 
any algebra is assumed to be finite dimensional. We also use the convention that $\min\varnothing =0$, 
where $\varnothing$ is the empty set. By a module we mean a left module.

\subsection{Quasi-hereditary algebras}\label{secprel.1}

For a general introduction to the theory of quasi-hereditary algebras we refer to the work 
of Cline-Parshall-Scott and Dlab-Ringel, see e.g. \cite{CPS, DR, PS}. Consider a 
finite-dimensional algebra $A$ with a partial order $\le$ on the indexing set $\Lambda_A$ 
of non-isomorphic simple $A$-modules. The algebra $(A,\le)$ is quasi-hereditary if 
and only if its category of finite dimensional modules $\cC_A:=A$-mod is a highest weight 
category with respect to this order, see \cite[Theorem~3.6]{CPS}. 

Concretely, denote the simple $A$-modules by $L^A(\lambda)$, for all $\lambda\in \Lambda_A$. 
The indecomposable projective cover, respectively injective hull, of $L^A(\lambda)$ is 
denoted by $P^A(\lambda)$, respectively $I^A(\lambda)$. The standard module $\Delta^A(\lambda)$ 
is defined as the maximal quotient of $P^A(\lambda)$ with all simple subquotients 
of the form $L^A(\mu)$ with $\mu\le \lambda$. The costandard module $\nabla^A(\lambda)$ is defined 
as the maximal submodule of $I^A(\lambda)$ with the same condition on its simple subquotients. 
We say that the pair $(A,\le)$ is a {\em quasi-hereditary algebra} if $[\Delta^A(\mu):L(\mu)]=1$ 
and, moreover, all projective modules have a filtration with standard subquotients
(the so-called {\em standard filtration}). This condition is equivalent to 
the corresponding dual condition for costandard modules.

For each $\lambda\in \Lambda_A$, there is a unique, up to isomorphism, 
indecomposable module $T^A(\lambda)$ which has both a standard filtration and a costandard 
filtration and for which there is an injection $\Delta^A(\lambda)\hookrightarrow T^A(\lambda)$
such that the resulting quotient has a standard filtration. This module is called a {\em tilting module}, 
see \cite{Ringel}. We refer to the collection of all the introduced modules 
as the {\em structural modules} of the quasi-hereditary algebra $A$. When there is no confusion 
possible, we leave out the reference to $A$ in the indexing poset, structural modules 
and the module category.

For a quasi-hereditary algebra $A$, its Ringel dual algebra, see \cite{Ringel,prinjective}, 
is defined as
\begin{displaymath}
 R(A):=\End_{A}(T)^{\op}\qquad\mbox{with}\quad T:=\bigoplus_{\lambda\in\Lambda}T(\lambda).
\end{displaymath}
Then $R(A)$ inherits a quasi-hereditary structure from $A$ with respect to the order which is opposite
to $\le$. Moreover, assuming that $A$ is basic, we have $R(R(A))\cong A$, see \cite[Section~6]{Ringel}.
The module $T$ is called the {\em characteristic} tilting module.

\subsection{Projective dimensions}\label{secprel.2}

For an abelian category $\cC$, we consider the {\em Yoneda extension functors}
\begin{displaymath}
\Ext^i_{\cC}(-,-)\;:\;\cC^{\op}\times \cC \,\to {\mathrm{Set}}, 
\end{displaymath}
see e.g. \cite[Section~III.3]{Ve} or \cite[Section~2]{SHPO3}. These Yoneda extension 
functors are isomorphic to the derived functors of the $\Hom$ functor in case $\cC$ contains
enough projective or injective objects. For an object $X\in\cC$, we 
denote by 
\begin{displaymath}
\pd_{\cC}X\in \mN\cup \{\infty\}
\end{displaymath}
the {\em projective dimension} of $X$ defined as the supremum of all $i\in\mN$ for which 
$\Ext^i_{\cC}(X,-)$ is not trivial. The {\em global (homological) dimension} of $\cC$, 
denoted by $\gd\cC$, is the supremum of the projective dimensions taken over all objects in $\cC$. 
The {\em finitistic dimension} of $\cC$, denoted by $\fd\cC$, 
is the supremum of the projective dimensions taken over all objects in $\cC$ which have 
finite projective dimension. Note that in general we have both
\begin{displaymath}
\gd{\cC}\in \mN\cup \{\infty\}\quad\mbox{ and }\quad \fd{\cC}\in \mN\cup \{\infty\}.
\end{displaymath}

A natural question for any quasi-hereditary algebra is to determine the projective 
dimensions of its structural modules and its global dimension. This global dimension 
is always finite, as proved by Parshall and Scott in \cite[Theorem~4.3]{PS}. Further 
results were obtained by K\"onig in \cite{Koenig}. In \cite[Section~4]{DR}, Dlab and 
Ringel study the implications of having standard modules with low projective dimensions 
and in \cite{DR1} they prove that every algebra of global dimension two has a 
quasi-hereditary structure. In \cite[Corollary~1]{MO}, the global dimension is 
linked to the projective dimension of the characteristic tilting module.

For the specific case of the principal block $\cO_0$ of category $\cO$ for reductive 
Lie algebras, the questions of projective dimensions were first addressed in the 
original paper \cite{BGG}. The second author completed these results in \cite{SHPO1,SHPO2} 
by determining all projective dimensions of all structural modules. In the current paper 
we will focus on these questions for the quasi-hereditary algebras associated to arbitrary 
blocks of category~$\cO$ and the parabolic generalisations of the latter.

\subsection{Koszul algebras}\label{secprel.3}

Let
\begin{displaymath}
B=\bigoplus_{i\in\mathbb{Z}}B_i 
\end{displaymath}
be a quadratic positively graded algebra. We denote its {\em quadratic dual} 
by~$B^{!}$, as in \cite[Definition~2.8.1]{BGS}. If $B$ is, moreover, Koszul, we 
denote its {\em Koszul dual} by $E(B)=\Ext^{\bullet}_B(B_0,B_0)$. 
By \cite[Theorem~2.10.1]{BGS}, we have $E(B)=(B^{!})^{\op}$ for any Koszul algebra~$B$. 
For a positively graded algebra $B$, we denote by $B$-gmod the category of finite dimensional 
$\mathbb{Z}$-graded $B$-modules. 

For a complex $\cM^\bullet$ of graded modules, we use the convention 
\begin{displaymath}
\left(\cM^\bullet [a]\langle b\rangle\right)^i_j=\cM^{i+a}_{j-b}, 
\end{displaymath}
for the shift $[\cdot]$ in position in the complex and the shift $\langle\cdot\rangle$ in 
degree in the module. This corresponds to the conventions in \cite{BGS} but differs slightly 
from the one in \cite{MOS}. A graded module $M$, regarded as an object in the derived category 
put in position zero without shift in grading, is denoted by~$M^\bullet$.

For any Koszul algebra $B$, \cite[Theorem~2.12.6]{BGS} introduces the 
{\em Koszul duality functor} $\cK_B$, which is a covariant equivalence of triangulated categories
\begin{displaymath}
\cK_B\,:\,\, \cD^b(B\mbox{-gmod})\;\tilde\to\; \cD^b(B^{!}\mbox{-gmod}).
\end{displaymath}
We use the convention where $\cK_B$ bijectively maps isomorphism classes of simple 
modules (respectively indecomposable projective modules) in~$B\mbox{-gmod}$ to 
isomorphism classes of indecomposable injective modules (respectively simple modules) 
in~$B^{!}\mbox{-gmod}$. This agrees with \cite{MOS}, but is dual to the convention in \cite{BGS}.
The Koszul duality functor $\cK_B$ satisfies 
\begin{displaymath}
\cK_B(\cN^\bullet [a]\langle b\rangle )=\cK_B(\cN^\bullet)[a-b]\langle -b\rangle, 
\end{displaymath}
see \cite[Theorem~2.12.5]{BGS}, or \cite[Theorem~22]{MOS}.

In the present paper, we always work in the situation when both $B$ and $B^!$ are finite dimensional.

\subsection{Category~$\cO$ and its parabolic generalisations}\label{prel1}

Consider the BGG category~$\cO$, associated to a triangular decomposition of a 
finite dimensional complex semisimple (or, more generally, reductive) Lie algebra  
$\fg=\fn^-\oplus\fh\oplus\fn^+$, see~\cite{BGG,Humphreys}. For any weight 
$\nu\in\fh^\ast$, we denote the simple highest weight module with highest weight 
$\nu$ by $L(\nu)$. We also introduce an involution on $\fh^\ast$ by setting 
$\widehat\nu=-w_0(\nu)$, where $w_0$ denotes the longest element of the Weyl group 
$W=W(\fg:\fh)$. We denote by $\langle\cdot,\cdot\rangle$ a $W$-invariant inner product 
on $\fh^\ast$.

We denote the set of integral weights by $\Lambda_{{\mathrm{int}}}$ and
the subset of dominant, not necessarily regular, weights by~$\intdom$.
For any $\lambda\in\intdom$, the indecomposable block in category~$\cO$ 
containing $L(\lambda)$ is denoted by~$\cO_\lambda$.

For $B$ the set of simple positive roots and $\mu\in\intdom$, 
set $B_\mu=\{\alpha\in B\,|\, \langle\mu+\rho,\alpha\rangle=0\}$. Let 
$\fu^{-}_\mu$ be the subalgebra of~$\fg$ generated by the root spaces corresponding 
to the roots in~$-B_\mu$. Then we have the parabolic subalgebra $\fq_\mu$ of~$\fg$, 
given by
\begin{displaymath}
\fq_\mu:=\fu_{\mu}^-\oplus\fh\oplus\fn^+.
\end{displaymath}

The full subcategory of~$\cO_\lambda$ with objects given by the modules 
in~$\cO_\lambda$ which are $U(\fq_\mu)$-locally finite is denoted 
by~$\cO_\lambda^\mu$. We will refer to this category as a {\em block}, see 
the discussion in Subsection \ref{blockdecomp}.

The category $\cO_\lambda^\mu$ is a direct summand of the parabolic 
version $\cO^\mu$ of category~$\cO$ as 
introduced in~\cite{RC}. By construction, $\cO^\mu_\lambda$ is a Serre 
subcategory of~$\cO_\lambda$. We denote the corresponding exact full 
embedding of categories by $\mathbf{\imath}^\mu:\cO^\mu\hookrightarrow \cO$. 
The left adjoint of $\mathbf{\imath}^\mu$ is the corresponding {\em Zuckerman functor},
denoted by $Z^{\mu}$. It is given, for a module $M\in\cO$,  by taking the largest quotient 
of $M$ which belongs to~$\cO^\mu$.

We define the set $X_\lambda$ as the set of {\em longest} representatives in~$W$ of
cosets in~$W/W_\lambda$. The non-isomorphic simple objects in the 
category~$\cO_\lambda$ are then indexed as follows: 
\begin{displaymath}
\{L(w\cdot\lambda)\,|\,w\in X_\lambda\}.
\end{displaymath}
Now, for~$x\in X_\lambda$, the module $L(x\cdot\lambda)$ is an object of~$\cO^\mu_\lambda$ 
if and only if $x$ is a {\em shortest} representative in~$W$ of a coset in~$W_\mu\backslash W$. 
The set of such shortest representatives $x\in X_\lambda$ is denoted by~$X_\lambda^\mu$. 
When $\lambda=0$, we simply write $L(x)$ for $L(x\cdot\lambda)$.

Consider the minimal projective generator of~$\cO^\mu_\lambda$ given by
\begin{equation}\label{progen} 
P^\mu_\lambda\;:=\;\bigoplus_{x\in X_\lambda^\mu}P^{\mu}(x\cdot\lambda),
\end{equation}
where $P^{\mu}(x\cdot\lambda)$ is the indecomposable projective cover of $L(x\cdot\lambda)$ in~$\cO^\mu_\lambda$.
Set $A_\lambda^\mu:=\End_{\fg}(P_\lambda^\mu)$. Then we have the usual 
equivalence  of categories
\begin{displaymath}
\cO^\mu_\lambda\,\;\tilde\to\;\,\mbox{mod-}A_\lambda^\mu;\qquad M\mapsto \Hom_{\fg}(P^\mu_\lambda,M). 
\end{displaymath}

We consider the usual Bruhat order $\le$ on~$W$, with the convention that $e$ is the smallest element. It restricts to the Bruhat 
order on~$X_\lambda^\mu$.
The order on the weights is defined by $x\cdot\lambda\le y\cdot\lambda$ if and only if $y\le x$. 
From the BGG Theorem on the structure of Verma modules, see e.g. \cite[Section~5.1]{Humphreys}, 
it follows that the algebra 
$A^\mu_\lambda$ is {\em quasi-hereditary} with respect to the poset of 
weights~$X_\lambda^\mu\cdot\lambda$.

Consider the translation functor $\theta^{on}_\lambda:\cO_0\to\cO_\lambda$ to the $\lambda$-wall. 
This functor has the adjoint $\theta^{out}_\lambda:\cO_\lambda\to \cO_0$, which is translation
out of the $\lambda$-wall, see \cite[Chapter~7]{Humphreys}. For $x\in W$, we denote by $\theta_x$ 
the unique projective functor on~$\cO_0$ satisfying
\begin{equation}
\label{deftheta}\theta_x \Delta(e)\;\cong\;P(x),
\end{equation}
see~\cite{BG}. Note that, in particular, $\theta^{out}_\lambda\circ\theta^{on}_\lambda=\theta_{w_0^\lambda}$. 
The contravariant duality on~$\cO$ which preserves isomorphism classes of simple objects, 
see \cite[Section~3.2]{Humphreys}, is denoted by~$\dd$. 
Existence of this duality functors implies that $(A_\lambda^\mu)^{\op}\cong A_\lambda^\mu$.

For any $x\in X_\lambda^\mu$, consider the following structural modules in~$\cO^\mu_\lambda$:
\begin{itemize}
\item the standard module (or generalised Verma module) $\Delta^\mu(x\cdot\lambda)$ 
with simple top $L(x\cdot\lambda)$,
\item the costandard module $\nabla^\mu(x\cdot\lambda):=\dd \Delta^\mu(x\cdot\lambda)$,
\item the indecomposable injective envelope $I^\mu(x\cdot\lambda)$ of~$L(x\cdot\lambda)$,
\item the indecomposable projective cover $P^\mu(x\cdot\lambda)$ of~$L(x\cdot\lambda)$,
\item the indecomposable tilting module $T^\mu(x\cdot\lambda)$
with highest weight $x\cdot\lambda$.
\end{itemize}

When $\mu$ is regular, meaning that the corresponding parabolic category~$\cO^\mu$ is the usual category 
$\cO$, we  leave out the reference to $\mu$. Similarly, we will leave out $\lambda$, or replace it by $0$, 
whenever it is regular. By application of~\cite[ Theorem~11]{SoergelD}, all 
categories $\cO^\mu_\lambda$ with $\lambda$ arbitrary integral regular dominant 
and $\mu$ fixed are equivalent to $\cO_0^\mu$, justifying this convention.

As proved in \cite{BGS,Back}, $A^\mu_\lambda$ has a {\em Koszul} grading. The algebra $A^\mu_\lambda$ 
is even standard Koszul in the sense of \cite{ADL}.  The corresponding graded module category is 
denoted by ${}^{\mZ}\cO_\lambda^\mu=A_\lambda^\mu$-gmod. We will sometimes replace the notation 
$\Hom_{{}^\mZ\cO}$ by $\hom_{\cO}$. We, furthermore, choose a normalisation of the grading 
of structural modules by demanding that simple modules appear in degree zero, projective 
and standard modules have their top in degree zero, injective and costandard modules have 
their socle in degree zero, while the grading of the (self-dual) tilting modules is 
symmetric around zero. Projective, inclusion and Zuckerman functors all admit graded lifts. 
We denote the corresponding graded lifts by the same symbols as for $\cO$ and use the grading 
convention of~\cite{Stroppel}. This means that
\begin{equation}\label{normon}
\theta_\lambda^{on}L(x)=
\begin{cases}
L(x\cdot\lambda)\langle -\len(w_0^\lambda)\rangle, & x\in X_\lambda;\\
0, &\text{otherwise};
\end{cases}
\end{equation}
for any $x\in W$, see \cite[Theorem~7.9]{Humphreys} for the ungraded statement.
By applying adjunction to  \eqref{normon}, the action of translation 
out of the wall on projective objects is derived as follows:
\begin{equation}\label{projout}
\theta_\lambda^{out} P(x\cdot\lambda)\langle 0\rangle=P(x)\langle 0\rangle,
\end{equation}
for any $x\in X_\lambda$. With our convention we have 
\begin{equation}\label{adjtransg}
\begin{aligned}
\hom_{\cO_0}(\theta_\lambda^{out}M,N)\cong 
\hom_{\cO_\lambda}(M,\theta_\lambda^{on}N\langle \len(w_0^\lambda)\rangle),\\
\hom_{\cO_0}(M,\theta_\lambda^{out}N)\cong
\hom_{\cO_\lambda}(\theta_\lambda^{on}M\langle \len(w_0^\lambda)\rangle,N),
\end{aligned}\end{equation}
see also \cite[Lemma~38]{MOS}.

Throughout the paper we will freely use that, as projective and standard modules have simple top, their graded lengths with respect to the Koszul grading equal their Loewy lengths, see \cite[Proposition 2.4.1]{BGS}. In general, the Loewy length of a gradable module is only bounded from above by its graded length.

We also introduce the notation $M(x,y)=\theta_x L(y)$, for $x,y\in W$. 
By \cite[Proposition 6.9]{CM4}, we have
\begin{equation}\label{tiltingtheta}
T^\mu(x)\;\cong\; \theta_{w_0w_0^\mu x}L(w_0^\mu w_0)\;\cong M(w_0w_0^\mu x,w_0^\mu w_0),
\end{equation}
for any $x\in X^\mu$. The link between regular and singular tilting modules is given by the following:
\begin{equation}\label{tiltingout}
\theta_\lambda^{out} T^{\mu}(y\cdot\lambda)\langle 0\rangle= T^{\mu}(yw_0^\lambda)\langle \len(w_0^\lambda) \rangle\qquad\forall y\in X^\mu_\lambda,
\end{equation}
This follows, for example, from the fact that 
$\theta_\lambda^{out} T^{\mu}(y\cdot\lambda)\langle 0\rangle$ is a tilting module and 
\cite[Theorem~5.4]{CM4}. Equations \eqref{tiltingtheta} and \eqref{tiltingout} prove in particular that all tilting modules in parabolic category $\cO$ are self-dual.

From Kazhdan-Lusztig theory, see \cite[Chapter 8]{Humphreys}, \cite[Section~3]{Brundan} or \cite{KL, Deo, Vogan, Ir4}, it is possible to determine the 
{\em Kazhdan-Lusztig-Vogan (KLV) polynomials} algorithmically. We denote them by
\begin{displaymath}
p^\mu_\lambda(x,y):=\sum_{k\in\mN}(-q)^k \dim \Ext^k_{\cO_\lambda^\mu}(\Delta^\mu(x\cdot\lambda),L(y\cdot\lambda)), 
\end{displaymath}
following the convention of \cite{Brundan}. It is then immediate that
\begin{displaymath}
\pd_{\cO_\lambda^\mu}\Delta^\mu(x\cdot\lambda)=\max_{y\in X_\lambda^\mu}\deg p^\mu_\lambda(x,y).
\end{displaymath}

Moreover, \cite[Corollary~3.9]{CPS2} implies that
\begin{displaymath}
\pd_{\cO_\lambda^\mu}L(x\cdot\lambda)=
\max_{y\in X_\lambda^\mu}\left(\pd_{\cO^\mu_\lambda}\Delta^\mu(y\cdot\lambda)+\deg p^\mu_\lambda(y,x)\right).
\end{displaymath}

These results imply that the projective dimension of simple and standard modules are, in principle, 
directly determined by the KLV polynomials. However, the KLV polynomials are only determined 
algorithmically, so we are interested in finding closed expressions. 

As noted in the introduction, we will prove that all projective dimensions of structural modules can be obtained 
from the functions $\llll$ and $\dddd$ on $X_\lambda$ for $\lambda\in\intdom$. 
These functions, in turn, can hence be determined in terms of KLV polynomials in the following way:
\begin{equation}
\label{pdDKLV}
\dddd(x)=\max_{y\in X_\lambda}\deg p_\lambda(x,y)\quad\mbox{and}\quad
\llll(x)=\max_{y\in X_\lambda}\left(\dddd(y)+\deg p_\lambda(y,x)\right). 
\end{equation}

The following property of KLV polynomials is well-known, see e.g. \cite{KL} for the case $\lambda=0$.

\begin{lemma}\label{vanish}
For any $x,y\in X_\lambda$ we have $p_\lambda(x,y)=0$ unless $x \ge y$ and 
\begin{displaymath}
\deg p_\lambda(x,y)\le \len(x)-\len(y). 
\end{displaymath}
\end{lemma}

\begin{proof}
We have to prove that $\Ext^j_{\cO_\lambda}(\Delta(x\cdot\lambda),L(y\cdot\lambda))=0$ 
unless $x\ge y$ and $j\le \len(x)-\len(y)$. We prove the claim by induction on $j$. 
For $j=0$, the statement is obvious. For $j>0$, assume the claim is true for $j-1$ 
and consider the short exact sequence
\begin{displaymath}
0\to K \to P(x\cdot\lambda)\to \Delta(x\cdot\lambda)\to 0.
\end{displaymath}
Applying the functor $\Hom_{\cO_\lambda}(-,L(y\cdot\lambda))$ yields a long exact sequence containing
\begin{displaymath}
0\to\Ext^{j-1}_{\cO_\lambda}(K,L(y\cdot\lambda))\to\Ext^j_{\cO_\lambda}(\Delta(x\cdot\lambda),L(y\cdot\lambda))\to 0. 
\end{displaymath}
As $K$ has a filtration where all subquotients are of the form $\Delta(z\cdot\lambda)$, where $z\in X_\lambda$ 
and $z < x$, the induction step implies that there must be a $z\in X_\lambda$ such that
\begin{displaymath}
x > z \ge y\qquad \mbox{and} \qquad  j-1\le \len(z)-\len(y), 
\end{displaymath}
in order for the extension group to be non-zero. This yields the claim for $j$ as well, concluding the proof.
\end{proof}

\subsection{Koszul and Koszul-Ringel duality for category~$\cO$}\label{prelKosO}

The Koszul dual algebra of $A_\lambda^\mu$ has been determined in \cite{Back}, see also 
\cite{SoergelD,BGS}. The Ringel dual algebra of $A_\lambda^\mu$ has been determined 
in \cite{CM4}, see also \cite{SoergelT,prinjective}. The results are summarised as
\begin{displaymath}
E(A_\lambda^\mu)\cong A_{\widehat\mu}^\lambda\cong A_\mu^{\widehat\lambda}\qquad\,\mbox{and}\qquad \;\,R(A_\lambda^\mu)\cong A_\lambda^{\widehat\mu}\cong A_{\widehat\lambda}^\mu.  
\end{displaymath}
Hence the algebra $A_\lambda^\mu$ is Ringel self-dual if $\mu=0$ or $\lambda=0$ 
and Koszul self-dual if $\mu=0=\lambda$.

It is sometimes more convenient to work with the composition of the usual Koszul 
duality functor with the duality $\dd$ to obtain a contravariant equivalence of 
triangulated categories $\cK^\mu_\lambda:=\mathbf{d}\cK_{A^\mu_\lambda}$,
\begin{equation}\label{Kdfun}
\cK_\lambda^\mu\,:\,\cD^b({}^{\mZ}\cO^\mu_\lambda)\;\tilde\to\; \cD^b({}^{\mZ}\cO^\lambda_{\widehat\mu})\quad\mbox{with}\quad 
\cK^{\mu}_\lambda(\cM^\bullet[i]\langle j\rangle)=
\cK^{\mu}_\lambda(\cM^\bullet)[j-i]\langle j\rangle,
\end{equation}
where we silently assumed composition with a functor corresponding to 
the isomorphism $E(A^\mu_\lambda)\cong A^{\lambda}_{\widehat\mu}$. 
We also use the Koszul-Ringel duality functor in the convention 
of~\cite[Section 9.3]{CM4}, see also \cite{MaTilt, MOS, Ringel}, 
which yields a contravariant equivalence of triangulated categories
\begin{equation}\label{KRdfun}
\Phi_\lambda^\mu\,:\,\cD^b({}^{\mZ}\cO^\mu_\lambda)\;\tilde\to\; 
\cD^b({}^{\mZ}\cO^\lambda_{\mu})\quad\mbox{with}\quad 
\Phi^{\mu}_\lambda(\cM^\bullet[i]\langle j\rangle)=
\Phi^{\mu}_\lambda(\cM^\bullet)[j-i]\langle j\rangle.
\end{equation}

The Koszul and Koszul-Ringel duality functors possess very useful
properties with respect to the structural modules.

\begin{lemma}\label{lemKPhi}
For any $x\in X^\mu_\lambda$, we have
\begin{eqnarray*}
\cK^\mu_\lambda(L(x\cdot\lambda)^\bullet)\cong 
P^\lambda(x^{-1}w_0\cdot\widehat\mu)^\bullet&&
\Phi^\mu_\lambda(L(x\cdot\lambda)^\bullet)\cong 
T^\lambda(w_0^\lambda x^{-1}w_0^\mu\cdot\mu)^\bullet\\
\cK^\mu_\lambda(\Delta^\mu(x\cdot\lambda)^\bullet)\cong 
\Delta^\lambda(x^{-1}w_0 \cdot\widehat\mu)^\bullet&&
\Phi^\mu_\lambda(\Delta^\mu(x\cdot\lambda)^\bullet)\cong 
\nabla^\lambda(w_0^\lambda x^{-1}w_0^\mu \cdot\mu)^\bullet\\
\cK^\mu_\lambda(P^\mu(x\cdot\lambda)^\bullet)\cong 
L(x^{-1}w_0\cdot\widehat\mu)^\bullet&&
\Phi^\mu_\lambda(T^\mu(x\cdot\lambda)^\bullet)\cong 
L(w_0^\lambda x^{-1}w_0^\mu \cdot\mu)^\bullet.
\end{eqnarray*}

\end{lemma}
\begin{proof}
The properties for $\cK_\lambda^\mu$ are proved in \cite[Proposition 3.11.1]{BGS}, 
the properties for $\Phi^\mu_\lambda$ are proved in \cite[Corollary 9.10]{CM4}.
\end{proof} 
 
Note that, whereas $\Phi^\lambda_\mu\circ\Phi^\mu_\lambda$ is isomorphic to the 
identity on $\cD^b({}^{\mZ}\cO^\mu_\lambda)$, the composition 
$\cK^\lambda_{\widehat\mu}\circ\cK^\mu_\lambda$ corresponds to an extension 
of the equivalence ${}^\mZ\cO^\mu_\lambda\,\tilde\to\,{}^{\mZ}\cO_{\widehat\lambda}^{\widehat\mu}$ 
to the derived category.

The following is proved in \cite[Proposition 5.8]{CM4}, see also \cite{RH,MOS}:
\begin{equation}
\label{eq5.8}
\imath^\lambda\circ \cK^{\mu}_\lambda\;\cong\; 
\cK^{\mu}\circ\theta^{out}_\lambda\;\quad\mbox{ and }\quad\quad 
\cL Z^\lambda \circ \cK^\mu\;\cong\;\cK^\mu_\lambda\circ\theta_\lambda^{on}\,\langle 
\len(w_0^\lambda) \rangle.
\end{equation}

\subsection{Kazhdan-Lusztig orders and projective-injective modules}\label{prelKL}

We use the left, right and two-sided Kazhdan-Lusztig (KL) preorders on the Weyl group, 
see \cite{KL}, and denote them by $\le_L$, $\le_R$ and $\le_J$ respectively. 
We use the convention that $e$ is the smallest element. We write 
$x \sim_L y$ when $x\le_L y$ and $y\le_L x$, for $x,y\in W$, and similarly  
for $\sim_R$ and $\sim_J$. This gives an equivalence relation and the equivalence classes 
are called the left (respectively right) cells. For these we introduce 
the notation
\begin{displaymath}
\mathbf{L}(x)=\{y\in W\,|\, y\sim_L x\}\quad\mbox{and}\quad \mathbf{R}(x)=\{y\in W\,|\, y\sim_R x\} 
\end{displaymath}

The left and right preorder have, for $x,y\in W$, the following properties:
\begin{equation}
\label{LRorder}
x\le_L y\;\;\Leftrightarrow \; \;x^{-1} \le_R y^{-1}\;\;\Leftrightarrow \;\; 
yw_0 \le_L x w_0 \;\;\Leftrightarrow \;\; w_0 y\le_L  w_0 x,
\end{equation}
see e.g. \cite[Proposition~6.2.9]{BB} and \cite{KL}.

These orders may be used  to give an alternative definition of $X_\lambda^\mu$:
\begin{displaymath}
X^\mu=\{x\in W\,| \,\, x\le_R w_0^\mu w_0\}\quad\mbox{and}\quad X_\lambda=\{x\in W\,| \,\, w_0^\lambda\le_L x\},
\end{displaymath}
see e.g. \cite[Lemma~2.8]{Geck}. Using equation~\eqref{LRorder} and the bijection 
$y\mapsto w_0^\mu y w_0$ on $X^\mu$, allows to reformulate this as
\begin{equation}
\label{altXmu}
X^\mu=\{x\in W\,|\,w_0^\mu  \le_R w_0^\mu x\}.
\end{equation}

With our normalisation, Lusztig's $\aaa$-function satisfies $\aaa(x)\le \aaa(y)$ if 
$x\le_{L} y$ or $x\le_{R} y$. We will sometimes write $\aaa(\mathbf{R})$, respectively 
$\aaa(\mathbf{L})$, to denote the value $\aaa(x)$, for $x$ arbitrary in a right cell
$\mathbf{R}$, respectively left cell $\mathbf{L}$.

An important role in (parabolic) category~$\cO$ is played by 
projective-injective modules, see e.g. 
\cite{Irving,prinjective,SoergelD}.
In the following lemma we summarise some properties of such modules.

\begin{lemma}[R.~Irving]\label{lemIrving}
Consider $\cO^\mu_\lambda$ for some $\lambda,\mu\in\intdom$. 
For any $x\in X^\mu_\lambda$, the following properties are equivalent:
\begin{enumerate}[$($a$)$]
\item\label{lemIrving.1} $P^\mu(x\cdot\lambda)$ is injective.
\item\label{lemIrving.2} $P^\mu(x\cdot\lambda)\cong \dd P^\mu(x\cdot\lambda)\cong I^\mu(x\cdot\lambda)$.
\item\label{lemIrving.3} $P^\mu(x\cdot\lambda)$ is a tilting module.
\item\label{lemIrving.4} $L(x\cdot\lambda)$ appears in the socle of 
$\Delta^\mu(y\cdot\lambda)$ for some $y\in X^\mu_\lambda$.
\item\label{lemIrving.5} $x\in\mathbf{R}(w_0^\mu w_0)$.
\item\label{lemIrving.6} $\gl\, P^\mu(x\cdot\lambda)=2\aaa(w_0^\mu w_0)-2\aaa(w_0^\lambda)$.
\end{enumerate}
Furthermore, the graded length of any indecomposable projective module which 
is not injective is strictly smaller than $2\aaa(w_0^\mu w_0)-2\aaa(w_0^\lambda)$.
\end{lemma}

\begin{proof}
The equivalence of claims~\eqref{lemIrving.2}, \eqref{lemIrving.4} and \eqref{lemIrving.5}
is the main result of~\cite{Irving}.

Now we prove the equivalence of claims~\eqref{lemIrving.1}, \eqref{lemIrving.2} and \eqref{lemIrving.3}. 
As every tilting module in~$\cO^\mu_\lambda$ is self-dual, 
claim~\eqref{lemIrving.3} implies claim~\eqref{lemIrving.2}. It is trivial that claim~\eqref{lemIrving.2} implies 
claim~\eqref{lemIrving.1}. If claim~\eqref{lemIrving.1} is true, then $P^\mu(x\cdot\lambda)$
has both a standard and costandard filtration, implying claim~\eqref{lemIrving.3}.

The equivalence of claim~\eqref{lemIrving.2} and claim~\eqref{lemIrving.6} follows from \cite{Irving2}.

That all non-injective  projective modules have strictly lower graded length follows from~\cite{Irving2}.
\end{proof}

\subsection{Guichardet categories}\label{secprel.7}

Consider an abelian category $\cA$ of finite global dimension and let $S_{\cA}$ denote the class
of simple objects in~$\cA$. An {\em initial segment} in~$\cA$ is the Serre subcategory $\cI$  of 
$\cA$ generated by a subset $S_{\cI}\subset S_{\cA}$, for which the following condition is satisfied:
for any $L,L'\in S_{\cA}$ such that $\pd_{\cA}L'= \pd_{\cA}L-1$, $L\in S_{\cI}$
and $\Ext_{\cA}^1(L,L')\not=0$, we have $L'\in S_{\cI}$. An initial segment is 
{\em saturated} if, for all $L,L'\in S_{\cA}$
with $\pd_{\cA}L=\pd_{\cA}L'$, we have $L\in S_{\cI}$ if and only if $L'\in S_{\cI}$.

These constructions have been used in an attempt to obtain extension fullness properties 
inspired by the result in \cite[Theorem~3.9(i)]{CPS}. For definition of an {\em extension full}
subcategory we refer to \cite{SHPO3, CM3} or to \cite{He} where this concept is referred to 
as {\em entirely extension closed} subcategories.

The following two distinct definitions both correspond to what is called a Guichardet 
category in, respectively, \cite{Fu} and \cite{Ga}, we modify the terminology to make this 
distinction. We call an abelian category $\cA$ of finite global dimension a 
{\em weakly Guichardet category} if every saturated initial segment $\cI$ is extension 
full in~$\cA$. If all initial segments are extension full, $\cA$ is called a 
{\em strongly Guichardet category}.

Some small (counter)examples of strongly Guichardet categories are given in \cite[Section~2.4]{SHPO3}.
%%%%%%%%%%%%%%%%%%%%%%%%%%%%%%%%%%%%%%%%%%%%%%%%%%%%%%%%%%%%%%%%%%%%%%%%%%%%%%%%%%%%%%%%%%%%%%%%%%%%%%%%%%%%%%%

\section{Projective dimension and graded length in the derived category}\label{section3}\label{subgenpd}

In order to make full use of the Koszul duality functor further in the paper, we need to generalise the concepts of graded length and projective dimension of an object of an abelian category to objects in the derived category. That is the aim of this section.

\begin{definition}\label{defpdD}
For an abelian category $\cC$ and $\cN^\bullet\in \cD^b(\cC)$, set
\begin{displaymath}
\delta(\cN^\bullet):=\{i\in \mZ\,|\, \mbox{there is } M\in\cC\mbox{ for which } 
\Hom_{\cD^b(\cC)}(\cN^\bullet, M^\bullet [i])\not=0\}. 
\end{displaymath}
The {\em projective dimension} of $\cN^\bullet$ is defined as
\begin{displaymath}
\pd\,\cN^\bullet=\pd_{\cC}\,\cN^\bullet:=
\max\delta(\cN^\bullet)-\min \delta(\cN^\bullet)\;\,\in\,\; \mN\cup\{\infty\}.
\end{displaymath}
\end{definition}

For a $\mZ$-graded algebra $B$, consider $\cC_B=B$-gmod and let $P_B:= {}_BB$, 
the canonical projective generator.

\begin{definition}\label{defglD}
For $\cN^\bullet\in \cD^b(\cC_B)$, set 
\begin{displaymath}
\sigma(\cN^\bullet):=
\{i\in \mZ\,|\, \bigoplus_{j\in\mZ}\hom_{\cD^b(\cC_B)}(P_B^\bullet[-i+j]\langle j\rangle, \cN^\bullet)\not=0\}.
\end{displaymath}
The {\em graded length} of $\cN^\bullet\in\cD^b(\cC_B)$ is defined to be 
\begin{displaymath}
\gl\, \cN^\bullet :=\max\sigma(\cN^\bullet)- \min\sigma(\cN^\bullet)\;\,\in\,\; \mN.
\end{displaymath}
\end{definition}

We start with demonstrating that these notions correspond to the usual notions 
when restricted to the abelian category.

\begin{proposition}\label{restpd}
{\hspace{2mm}}

\begin{enumerate}[$($i$)$]
\item\label{restpd.1} For an abelian category $\cC$ and $N\in\cC$, we have
$\pd\,N^\bullet\,=\, \pd_{\cC}\, N$.
\item\label{restpd.2} For a graded algebra $B$ and $N\in B$-{\rm gmod}, we have
$\gl\, N^\bullet\,=\, \gl\, N$. 
\item\label{restpd.3} For a graded algebra $B$, $\cN^\bullet\in\cD^b(B\mbox{-gmod})$ and $I_B$ 
an injective cogenerator with the socle contained in degree zero, set 
\begin{displaymath}
\sigma'(\cN^\bullet):=
\{i\in \mZ\,|\, \;\bigoplus_{j\in\mZ}\hom_{\cD^b(\cC_B)}( \cN^\bullet,I_B^\bullet[i-j]\langle -j\rangle)\not=0\}. 
\end{displaymath}
Then we have
\begin{displaymath}
\gl\, \cN^\bullet =\max\{\sigma'(\cN^\bullet)\}- \min\{\sigma'(\cN^\bullet)\}. 
\end{displaymath}
\end{enumerate}
\end{proposition}

\begin{proof}
Claim~\eqref{restpd.1} follows immediately from
\begin{displaymath}
\Hom_{\cD^b(\cC)}(N^\bullet, M^\bullet [i])=\Ext^i_{\cC}(N,M), 
\end{displaymath}
which holds for Yoneda extensions by \cite[Section~III.1 and III.3]{Ve}.

Similarly, claim~\eqref{restpd.2} follows from
\begin{displaymath}
\hom_{\cD^b(\cC_B)}(P^\bullet[-i+j]\langle j\rangle, N^\bullet)=
\ext^{i-j}_{\cC_B}(P\langle j\rangle, N)=\hom_{\cC_B}(P\langle i\rangle,N).
\end{displaymath}

To prove claim~\eqref{restpd.3}, it suffices to prove
\begin{displaymath}
\hom_{\cD^b(\cC_B)}(P^\bullet [a]\langle b\rangle, \cN^{\bullet})=
\hom_{\cD^b(\cC_B)}(\cN^{\bullet},I^\bullet [-a]\langle -b\rangle),
\end{displaymath}
for an arbitrary complex $\cN^\bullet$, $a,b\in\mZ$ and $P$ the projective cover 
of a simple module such that $I$ is the injective hull of that simple module. 
The equation is clearly true in case $\cN^\bullet =N^\bullet[k]$ for some module $N$ and $k\in\mZ$. 
As modules generate the derived category as a triangulated category, the general claim follows
by standard arguments considering distinguished triangles and corresponding 
long exact sequences.
\end{proof}

\begin{proposition}\label{exchange}
Consider a Koszul algebra $B$ such that $B^!$ is finite dimensional.
Let $\cK_B$ be the corresponding Koszul duality functor. 
For any $\cN^\bullet\in \cD^b(B\mbox{-\rm gmod})$, we have
\begin{displaymath}
\gl\, \cK_B(\cN^\bullet)=\pd\,\cN^\bullet\quad\mbox{ and }\quad 
\pd\, \cK_B(\cN^\bullet)=\gl\, \cN^\bullet. 
\end{displaymath}
\end{proposition}

\begin{proof}
As all finitely generated $B$- and $B^{!}$-modules have finite length, 
it suffices to consider simple modules $M$ in Definition~\ref{defpdD} of the projective dimension.

For a simple module $L$ and $I^\bullet=\cK_B(L^\bullet)$, we find
\begin{displaymath}
\hom_{\cD^b(B\mbox{-}{\rm gmod})}(\cN^\bullet, L^\bullet [i]\langle j\rangle)=
\hom_{\cD^b(E(B)\mbox{-}{\rm gmod})}(\cK_B(\cN^\bullet), I^\bullet [i-j]\langle -j\rangle), 
\end{displaymath}
by \cite[Theorem~2.12.6]{BGS}. The result then follows from Proposition~\ref{restpd}\eqref{restpd.3}.
\end{proof}

%%%%%%%%%%%%%%%%%%%%%%%%%%%%%%%%%%%%%%%%%%%%%%%%%%%%%%%%%%%%%%%%%%%%%%%%%%%%%%%%%%%%%%%%%%%%%%%%%%%%%%%%%%%%%%%%%%%%%%%%

\section{Projective dimensions in parabolic category~$\cO$}\label{section4}
\subsection{The parabolic dimension shift}

The principal result in this section implies that the problem of determining 
the projective dimension of a module in~$\cO^\mu$ is equivalent to determining 
its projective dimension as an object in category~$\cO$.

\begin{theorem}\label{pdshift} 
{\hspace{2mm}}

\begin{enumerate}[$($i$)$]
\item\label{pdshift.1} For $\lambda\in\Lambda^{+}_{\mathrm{int}}$ and any 
$M\in\cO^\mu_{\lambda}$, we have
\begin{displaymath}
\pd_{\cO^\mu} M=\pd_{\cO}M-2\len(w_0^\mu). 
\end{displaymath}
\item\label{pdshift.2} For any $M\in \cO^\mu_\lambda$ with $p=\pd_{\cO^\mu_\lambda}M$ 
and $x\in X_\lambda$, we have
\begin{displaymath}
\Ext^{p+2\len(w_0^\mu)}_{\cO_\lambda}(M,L(x\cdot\lambda))=
\begin{cases}
\Ext^p_{\cO^\mu_\lambda}(M, L(x\cdot\lambda)),&\mbox{if } x\in X^\mu_\lambda;\\
0,&\mbox{if }x\not\in X^\mu_\lambda.
\end{cases} 
\end{displaymath}
\end{enumerate}
\end{theorem}

The first result is, in fact, a special case of a more general result.

\begin{theorem}\label{DCshift}
Consider $\lambda,\mu\in\intdom$.
\begin{enumerate}[$($i$)$]
\item\label{DCshift.1} For any $\cN^\bullet\in \cD^b(\cO_\lambda^\mu)$, we have
\begin{displaymath}
\gl\, \theta_\lambda^{out}(\cN^\bullet)=\gl\,\cN^\bullet +2\len(w_0^\lambda). 
\end{displaymath}
\item\label{DCshift.2} For any $\cN^\bullet\in \cD^b(\cO_\lambda^\mu)$, we have
\begin{displaymath}
\pd_{\cO_\lambda}\, \imath^\mu(\cN^\bullet)=\pd_{\cO_\lambda^\mu}\,\cN^\bullet +2\len(w_0^\mu). 
\end{displaymath}
\end{enumerate}
\end{theorem}

Before proving these, we note the following consequences.

\begin{corollary}\label{glidentity}
For $M$ in~${}^{\mZ}\cO_\lambda$, we have
\begin{displaymath}
\gl\, \theta^{on}_\lambda\theta^{out}_\lambda M\,=\,\gl \,\theta^{out}_\lambda M\,=\,\gl\, M+2\len(w_0^\lambda). 
\end{displaymath}
\end{corollary}

\begin{proof}
The equality $\gl\, \theta^{on}_\lambda\theta^{out}_\lambda M\,=\,\gl\, M+2\len(w_0^\lambda)$ 
follows immediately from \cite[Proposition~5.1]{CM4}. The equality 
$\gl\, \theta^{out}_\lambda M\,=\gl\, M+2\len(w_0^\lambda)$ is a special case of 
Theorem~\ref{DCshift}\eqref{DCshift.1} by Proposition~\ref{restpd}\eqref{restpd.2}.
\end{proof}

\begin{corollary}\label{cor2808}
Consider $\lambda,\mu\in\intdom$. Let
\begin{displaymath}
0\to N\to Q\to M\to 0 
\end{displaymath}
be a short exact sequence in $\cO_\lambda$. 
If $Q$ is projective in $\cO^\mu$, then 
\begin{displaymath}
\pd_{\cO}M\,=\,\pd_{\cO}N+1. 
\end{displaymath}
\end{corollary}

\begin{proof}
Since $\cO^\mu_\lambda$ is a Serre subcategory of $\cO_\lambda$,
the modules $M$ and $N$ belong to $\cO^\mu_\lambda$. 
Then we immediately have $\pd_{\cO^\mu}M\,=\,\pd_{\cO^\mu}N+1$. 
The assertion of the corollary now follows directly from Theorem~\ref{pdshift}.
\end{proof}

Now we start the proofs of Theorems \ref{pdshift} and~\ref{DCshift}.

\begin{lemma}\label{thetaP}
{\hspace{2mm}}

\begin{enumerate}[$($i$)$]
\item\label{thetaP.1} For any $x\in X_\lambda^\mu$, we have 
\begin{displaymath}
\theta^{on}_\lambda P^\mu(x)\langle 0\rangle =
\bigoplus_{j\in\mN} P^\mu(x\cdot\lambda)^{\oplus c_j}\langle j-\len(w_0^\lambda)\rangle, 
\end{displaymath}
for $c_j\in\mN$ satisfying $c_0=c_{2\len(w_0^\mu)}=1$ and $c_j=0$ if $j> 2\len(w_0^\mu)$.
\item\label{thetaP.2} For any $x\in X^\mu\backslash X^\mu_\lambda$, the module 
$\theta^{on}_\lambda P^\mu(x)\langle 0\rangle$ is the direct sum of 
shifted projective objects in~$\cO_\lambda^\mu$, where all occurring 
degrees are strictly between $-\len(w_0^\lambda)$ and $\len(w_0^\lambda)$. 
\end{enumerate}
\end{lemma}

\begin{proof}
We restrict to $\mu=0$. The proof for the general case does not change 
substantially or, alternatively, the result follows from the non-parabolic 
case by applying the Zuckerman functor. Claim~\eqref{thetaP.1} follows from 
equation~\eqref{projout} and \cite[Proposition~5.1]{CM4}.

To see in which degrees the indecomposable projective summands of $\theta_\lambda^{on}P(x)$
appear, for~$x\not\in X_\lambda$, we consider
\begin{displaymath}
\hom_{\cO_\lambda}(\theta_\lambda^{on}P(x),L(y\cdot\lambda)\langle j\rangle)=
\hom_{\cO_\lambda}(P(x),\theta_{w_0^\lambda}L(y)\langle j-\len(w_0^\lambda)\rangle), 
\end{displaymath}
for any $y\in X_\lambda$. As the extremal degrees $\len(w_0^\lambda)$ and 
$-\len(w_0^\lambda)$ in~$\theta_{w_0^\lambda}L(y)$ must correspond to the 
simple top and socle, which is given by $L(y)$, claim~\eqref{thetaP.2} follows.
\end{proof}

\begin{corollary}\label{extrdeg}
{\hspace{2mm}}

\begin{enumerate}[$($i$)$]
\item\label{extrdeg.1} For $M\in {}^{\mZ}\cO_\lambda$, the simple modules in the extremal degrees of 
$\theta_\lambda^{out}M$ are all of the form $L(y)$ with $y\in X_\lambda$.
\item\label{extrdeg.2} More generally, for $\cN^\bullet\in\cD^b({}^{\mZ}\cO_\lambda^\mu)$, the extremal 
values in the set $\sigma(\theta_\lambda^{out}\cN^\bullet)$ in Definition \ref{defglD}  only 
come from indecomposable projective objects of the form $P^\mu(y)$ with $y\in X^\mu_\lambda$.
\end{enumerate}
\end{corollary}

\begin{proof}
For any $j\in\mZ$ and $z\in W$, equation \eqref{adjtransg} implies that we have
\begin{displaymath}
\hom_{\cO_0}(P(z)\langle j\rangle,\theta_\lambda^{out}M)\cong
\hom_{\cO_\lambda}(\theta_\lambda^{on}P(z)\langle j-\len(w_0^\lambda)\rangle,M ) 
\end{displaymath}
Comparing Lemma~\ref{thetaP}\eqref{thetaP.1} and \eqref{thetaP.2}, then implies that the 
extremal values of $j$ which give non-zero morphism spaces will be reached only for $z\in X_\lambda$.
The same argument can be used in the derived category.
\end{proof}

\begin{proof}[Proof of Theorem~\ref{DCshift}]
As the Koszul duality functor \eqref{Kdfun} intertwines the parabolic inclusion 
functor and translation out of the wall, see equation~\eqref{eq5.8}, claims~\eqref{DCshift.1}
and \eqref{DCshift.2}) are equivalent by Proposition~\ref{exchange}. We focus on proving claim~\eqref{DCshift.1}.

Take $\cN^\bullet\in\cD^b(\cO_\lambda^\mu)$. By Corollary \ref{extrdeg}(ii), it suffices to use 
projective objects $P^\mu(x)$ with $x\in X_\lambda^\mu$ in Definition~\ref{defglD}. 
Equation~\eqref{adjtransg} and Lemma~\ref{thetaP}\eqref{thetaP.1} then imply
\begin{eqnarray*}
&&\sum_{j\in\mZ}\dim\hom_{\cD^b(\cO^\mu_0)}(P^\mu(x)^\bullet[-i+j]\langle j\rangle, \theta_\lambda^{out} \cN^\bullet)\\
&&=\sum_{j',k\,\in\mZ}c_k\dim\hom_{\cD^b(\cO^\mu_\lambda)}(P^\mu(x\cdot\lambda)^\bullet[-(i+k)+j']\langle j'\rangle, \cN^\bullet), \end{eqnarray*}
with $j'=j+k$. This implies indeed that $\gl \theta^{out}\cN^{\bullet}=\gl \cN^\bullet +2\len(w_0^\lambda)$.
\end{proof}

\begin{proof}[Proof of Theorem~\ref{pdshift}]
Claim~\eqref{pdshift.1} is a special case of Theorem~\ref{DCshift}\eqref{DCshift.2}, 
by Proposition~\ref{restpd}\eqref{restpd.1}. Claim~\eqref{pdshift.2} is the Koszul dual of Corollary~\ref{extrdeg}.
\end{proof}

An alternative proof for the first part of Theorem~\ref{pdshift} can be given  analogously 
to the proof of~\cite[Theorem~21(i)]{SHPO3}, using the following lemma as a replacement 
of~\cite[Lemma~23]{SHPO3}. We prove this lemma without using the results in Section~\ref{subgenpd}.

\begin{lemma}\label{pdPmu}
All projective modules $Q$ in integral $\cO^\mu$ satisfy $\pd_{\cO}Q=2\len(w_0^\mu)$. 
Furthermore, we have
\begin{displaymath}
\dim\Ext^{2\len(w_0^\mu)}_{\cO}(P^\mu(\kappa),L(\eta))=\delta_{\kappa,\eta}, 
\end{displaymath}
for arbitrary $\kappa,\eta\in \Lambda_{{\rm int}}$.
\end{lemma}

\begin{proof}
Let $\lambda\in \intdom$ and recall that 
$P^\mu(x\cdot \lambda)\langle 0\rangle\cong Z^\mu P(x\cdot\lambda)\langle 0\rangle$, 
for all $x\in X_\lambda^\mu$. Equations~\eqref{Kdfun} and~\eqref{eq5.8} and Lemma~\ref{lemKPhi} yield 
\begin{eqnarray*}
&&\Ext^j_{\cO_\lambda}(P^\mu(x\cdot\lambda),L(y\cdot\lambda))\\
&=&\oplus_{i\in\mZ}\hom_{\cD^b({}^{\mZ}\cO_\lambda)}\left(i^\mu \cL Z^\mu P(x\cdot\lambda)^\bullet,L(y\cdot\lambda)^\bullet[j]\langle i\rangle\right)\\
&=&\oplus_{i\in\mZ}\hom_{\cD^b({}^{\mZ}\cO^\lambda)}\left(P^\lambda(y^{-1}w_0)^\bullet[i-j]\langle i\rangle,\theta_{w_0^{\widehat\mu}} L(x^{-1}w_0)^\bullet\langle \len(w_0^\mu)\rangle\right)\\
&=&\hom_{\cO^\lambda}\left(P^\lambda(y^{-1}w_0)\langle j\rangle,\theta_{w_0^{\widehat\mu}} L(x^{-1}w_0)\langle \len(w_0^\mu)\rangle\right).
\end{eqnarray*}
The results hence follow from \cite[Lemma~5.2(ii)]{CM4}.
\end{proof}

\begin{remark}{\rm
The proof of Lemma~\ref{pdPmu} also shows that $P^\mu(x\cdot\lambda)$ has a linear projective resolution in~${}^{\mZ}\cO_\lambda$, as a generalisation of~\cite[Proposition~41]{SHPO1}.}
\end{remark}

\subsection{The category $\cO^{\hat{R}}$.}

The constant shift in projective dimension between parabolic and original category $\cO$ 
in Theorem~\ref{pdshift} will turn out to be useful for the calculations of projective 
dimensions in original category $\cO$, besides their obvious interest in the corresponding 
questions in $\cO^\mu$. In particular, a seemingly logical idea to generalise the statement 
in Proposition~\ref{maxpdL}\eqref{maxpdL.5} to arbitrary elements (and hence right cells) outside type 
$A$, is to investigate whether the principle of Theorem~\ref{pdshift} extends to other full 
Serre subcategories of $\cO_0$ generalising $\cO_0^\mu$, introduced in \cite[Section~4.3]{cell}. 
Unfortunately the answer is negative, as we prove in this section. One of the reasons for that, is
the fact that the global  dimension of these categories can be infinite.

For $\mathbf{R}$ a right cell in $W$, we set 
\begin{displaymath}
\hat{\mathbf{R}}:=\{w\in W\;|\; w\le_R \mathbf{R}\},
\end{displaymath}
so, in particular, $\hat{\mathbf{R}}(w_0^\mu w_0)=X^\mu$, for any $\mu\in \Lambda^+_{\mathrm{int}}$. 
For any right cell $\mathbf{R}$, let $\cO^{\hat{\mathbf{R}}}_0$ denote the Serre subcategory of 
$\cO_0$ generated by the modules $L(x)$ with $x\in \hat{\mathbf{R}}$. By the above, we have, as 
a special case, $\cO_0^\mu=\cO_0^{\hat{\mathbf{R}}(w_0^\mu w_0)}$.

\begin{proposition}
In general ,the category $\cO^{\hat{\mathbf{R}}}_0$ can have infinite global dimension. 
In particular, simple modules can have infinite projective dimension.
\end{proposition}

\begin{proof}
We prove that this is the case for the category $\cO_0^{\widehat{\mathbf{R}}}$ in 
\cite[Example~5.3]{cell}. This example considers the case $\fg=\mathfrak{sl}(4)$ and 
$\mathbf{R}=\mathbf{R}(s_2)$. We denote $s=s_1$, $t=s_2$ and $r=s_3$. Then we have
$\mathbf{R}(t)=\{e,t,ts,tr\}$ and the graded filtrations of projective modules in 
$\cO^{\widehat{\mathbf{R}}}_0$ look as follows:
\begin{displaymath}
\begin{array}{|c||c|c|c|c|}
\hline
w & e & t & ts & tr\\
\hline\hline
P^{\hat{\mathbf{R}}}(w)&
\begin{array}{c}e\\t\\\text{\hspace{2mm}}\end{array}&
\begin{array}{ccc}&t&\\ts&e&tr\\&t&\end{array}&
\begin{array}{c}ts\\t\\ts\end{array}&
\begin{array}{c}tr\\t\\tr\end{array}
\\
\hline
\end{array}
\end{displaymath}
From this description of projective modules in $\cO_0^{\widehat{\mathbf{R}}}$, we find that 
the projective resolution of the injective envelope of $L(e)$ in $\cO_0^{\widehat{\mathbf{R}}}$ is given by
\begin{displaymath}
0\to P^{\hat{\mathbf{R}}}(e)\oplus P^{\hat{\mathbf{R}}}(t)\to 
P^{\hat{\mathbf{R}}}(ts)\oplus P^{\hat{\mathbf{R}}}(tr)\to 
P^{\hat{\mathbf{R}}}(t)\to I^{\hat{\mathbf{R}}}(e)\to 0. 
\end{displaymath}
The other three indecomposable injective modules are, clearly, self-dual and projective. 
Hence all injective modules have finite global dimension and the finitistic dimension of 
the category is therefore equal to the maximum of those projective dimensions, see e.g. 
the proof of \cite[Theorem~3]{Rome}. Hence we find
\begin{displaymath}
\fd\cO_0^{\widehat{\mathbf{R}}}\;=\;3. 
\end{displaymath}
It thus suffices to prove that there exists a module with projective dimension strictly 
greater than $3$. For this we consider the module $M$ of length two with top $L(s_2s_1)$ 
and socle $L(s_2)$. For this module we, clearly, have
\begin{displaymath}
0\to L(ts)\to P^{\hat{\mathbf{R}}}(ts)\to M\to 0. 
\end{displaymath}
Therefore $\pd M =\pd L(ts)+1$. Constructing the minimal projective resolution shows 
that the projective dimension of $L(ts)$ must be greater than two, so that of $M$ must be 
greater than $3$ and therefore infinite. This means that also the projective dimension 
of $L(ts)$ must be infinite.
\end{proof}

\begin{remark}
{\rm
As the global dimension of $\cO_0^{\hat{\mathbf{R}}}$ might be infinite, 
the category $\cO_0^{\hat{\mathbf{R}}}$, in general, fails to admit  
any structure of a highest weight category due to \cite[Theorem~4.3]{PS}. Moreover, the 
above calculation even shows that the finitistic dimension may be odd. This suggests  
that $\cO_0^{\hat{\mathbf{R}}}$, in general, is not equivalent to the module
category of a properly stratified algebra due to  \cite[Theorem~1]{MO}.
}
\end{remark}

\section{Blocks and their global dimension}\label{section5}
\subsection{Indecomposability}\label{blockdecomp}

The categories $\cO_0^\mu$ and $\cO_\lambda$ are indecomposable, where one is the Koszul 
dual of the other claim. However, in general, $\cO_\lambda^\mu$ may decompose, see \cite{ES} or \cite[Section 8.2.1]{BN}. At he same time, Brundan
proved in \cite[Section~1]{BrundanHecke} that all blocks $\cO^\mu_\lambda$ remain 
indecomposable for $\mathfrak{g}=\mathfrak{sl}(n)$ (whenever they are non-zero). 
We give an independent proof of this statement.

\begin{proposition}[J.~Brundan]
For $\mathfrak{g}=\mathfrak{sl}(n)$ and any $\lambda,\mu\in\intdom$, 
the subcategory $\cO_\lambda^\mu$ is an indecomposable block of $\cO^\mu$, whenever it is non-zero.
\end{proposition}

\begin{proof}
Assume that the category $\cO_\lambda^\mu$ decomposes into two subcategories. The restriction 
of the translation functor $\theta^{on}_\lambda$ to $\cO^\mu$ decomposes accordingly. 
That both parts are not trivial follows from equation \eqref{normon}. It then follows 
also that $\theta_{w_0^\lambda}=\theta_{\lambda}^{out}\theta_\lambda^{on}$ decomposes, 
as $\theta_{\lambda}^{out}$ is faithful. However, $\theta_{w_0^\lambda}$ is indecomposable 
by \cite[Theorem~1(i)]{Kilde}. We thus obtain a contradiction.
\end{proof}

Even though, strictly speaking, it is only justified for $\mathfrak{g}=\mathfrak{sl}(n)$, 
we will refer to the category $\cO^\mu_\lambda$ as a {\em block}.

\subsection{Homological dimension of blocks}\label{section5.2}

\begin{theorem}\label{thmgd}
The global dimension of each integral non-zero block in parabolic category~$\cO$ is given by
\begin{displaymath}
\gd \cO_\lambda^\mu\,=\, 2\aaa(w_0w_0^\lambda)-2\aaa(w_0^\mu). 
\end{displaymath}
\end{theorem}

\begin{proof}
In case $\mu=0$, this is precisely \cite[Theorem~25(ii)]{SHPO3}. 
The combination of that result and Theorem~\ref{pdshift} then implies the inequality 
$\gd \cO_\lambda^\mu\,\le\, 2\aaa(w_0w_0^\lambda)-2\aaa(w_0^\mu).$

To prove the statement, it hence suffices to prove that there is a simple module 
in~$\cO^\mu_\lambda$ with projective dimension equal to $2\aaa(w_0w_0^\lambda)-2\aaa(w_0^\mu)$. 
By Proposition~\ref{exchange} and Lemma~\ref{lemKPhi}, this is equivalent to the claim 
that there is a projective module in~$\cO^\lambda_{\widehat\mu}$ with graded 
length $2\aaa(w_0w_0^\lambda)-2\aaa(w_0^\mu)$. The latter is guaranteed by the equivalence of
Lemma~\ref{lemIrving}\eqref{lemIrving.4} and \eqref{lemIrving.6}.
\end{proof}

Theorem~\ref{thmgd} implies a nice criterion for the semisimplicity of the category 
$\cO^\mu_\lambda$.  Other criteria for special cases have been obtained in \cite{BN}.

\begin{corollary}
A non-zero block $\cO^\mu_\lambda$ is semisimple if and only if
\begin{displaymath}
\aaa(w_0^\mu)\;= \; \aaa(w_0w_0^\lambda). 
\end{displaymath}
\end{corollary}

\begin{remark}
{\rm  
From the above results it follows that the inequality $\aaa(w_0^\mu)>\aaa(w_0w_0^\lambda)$ implies
that the block $\cO^\mu_\lambda$ is zero. However, there are zero blocks $\cO^\mu_\lambda$ for which
$\aaa(w_0^\mu)\leq \aaa(w_0w_0^\lambda)$. For example, in the case
$\mathfrak{g}=\mathfrak{sl}(4)$, $w_0^\lambda =s_1$ and $w_0^\mu = s_1 s_2 s_1$,
we have $\aaa(w_0^\mu)=3=\aaa(w_0w_0^\lambda)$, while $\cO^\mu_\lambda=0$.
} 
\end{remark}

%%%%%%%%%%%%%%%%%%%%%%%%%%%%%%%%%%%%%%%%%%%%%%%%%%%%%%%%%%%%%%%%%%%%%%%%%%%%%%%%%%%%%%%%%%%%%%%%%%%%%%%%%%%%%%%%%%

\section{Connections between the projective dimensions and graded lengths}\label{secconnec}

\subsection{Preliminaries}\label{secconnec.1}

In this section we establish some connections between the projective dimensions and the graded 
lengths of the structural modules in blocks of the parabolic version of category~$\cO$. In 
Subsection~\ref{secconnec.2} this is achieved by applying Koszul and Koszul-Ringel duality.
In Subsection~\ref{secconnec.3}, by making use of the graded lifts of translation functors 
and in Subsection~\ref{secconnec.4}, by applying the derived Zuckerman functor.

We start with proving an analogue of Lemma~\ref{lemIrving} for tilting modules.

\begin{lemma}\label{corTproj}
{\hspace{2mm}}

\begin{enumerate}[$($i$)$]
\item\label{corTproj.1}
For $\lambda,\mu\in\intdom$ and $x\in X^\mu_\lambda$,  the following properties are equivalent:
\begin{enumerate}[$($a$)$]
\item\label{corTproj.11} $T^\mu(x\cdot\lambda)$ is projective.
\item\label{corTproj.12} $w_0^\mu xw_0^\lambda\in\mathbf{R}( w_0^\mu)$.
\item\label{corTproj.13} $\gl\, T^\mu(x\cdot\lambda)=2\aaa(w_0^\mu w_0)-2\aaa(w_0^\lambda)$.
\end{enumerate}
\item\label{corTproj.2}
The graded length of any indecomposable tilting module which is not 
projective is strictly smaller than $2\aaa(w_0^\mu w_0)-2\aaa(w_0^\lambda)$.
\end{enumerate}
\end{lemma}

\begin{proof}
The implication \eqref{corTproj.11}$\Rightarrow$\eqref{corTproj.13} follows from Lemma~\ref{lemIrving}. 

The combination of~\cite[Proposition 1]{SHPO2} and equation~\eqref{tiltingtheta} imply that 
\begin{displaymath}
\begin{cases}
\gl\, T^\mu(x)=2\aaa(w_0^\mu w_0),&\mbox{if }w_0^\mu x\in\mathbf{R}(w_0^\mu);\\ 
\gl\, T^\mu(x) <2\aaa(w_0^\mu w_0),&\mbox{otherwise}.
\end{cases} 
\end{displaymath}
Using equation~\eqref{tiltingout} and Corollary~\ref{glidentity}, we then obtain
\begin{displaymath}
\begin{cases}\gl\, T^\mu(x\cdot\lambda)=2\aaa(w_0^\mu w_0)-2\aaa(w_0^\lambda),&
\mbox{if }w_0^\mu xw_0^\lambda\in\mathbf{R}( w_0^\mu);\\ 
\gl\, T^\mu(x\cdot\lambda) <2\aaa(w_0^\mu w_0)-2\aaa(w_0^\lambda),&\mbox{otherwise}.
\end{cases}
\end{displaymath}
This implies claim~\eqref{corTproj.2} and shows that \eqref{corTproj.12}$\Leftrightarrow$\eqref{corTproj.13}. 

Next we prove the implications \eqref{corTproj.12}$\Rightarrow$\eqref{corTproj.11} for the case 
$\lambda=0$. As by the above we already have \eqref{corTproj.12}$\Leftarrow$\eqref{corTproj.11}, it suffices to prove that the number of non-isomorphic indecomposable projective tilting modules in~$\cO_0^\mu$ is equal to the cardinality of the set $w_0^\mu \mathbf{R}( w_0^\mu)\cap X^\mu$. By equation~\eqref{altXmu}, the latter set is just $\mathbf{R}(w_0^\mu)$. The claim thus follows from Lemma~\ref{lemIrving}\eqref{lemIrving.5} and equation~\eqref{LRorder}.

Finally we prove \eqref{corTproj.13}$\Rightarrow$\eqref{corTproj.11}  for general $\lambda\in\intdom$, 
relying on the result for $\lambda=0$. If $\gl\, T^\mu(x\cdot\lambda)=2\aaa(w_0^\mu w_0)-2\aaa(w_0^\lambda)$, 
then equation \eqref{tiltingout} and Corollary \ref{glidentity} in combination with the case 
$\lambda=0$ imply that $T^\mu(xw_0^\lambda)$ is projective in~$\cO^\mu_0$. This implies 
that $\theta_\lambda^{on}T^\mu(xw_0^\lambda)\cong T^\mu(x)^{\oplus |W_\lambda|}$ is 
projective in~$\cO_\lambda^\mu$.
\end{proof}

\subsection{Applying duality functors}\label{secconnec.2}

\begin{proposition}\label{proplinks}
For $\lambda,\mu\in\intdom$ and $x\in X^\mu_\lambda$, we have the 
following links between graded lengths  and projective dimensions:
\begin{enumerate}[$($i$)$]
\item\label{proplinks.1} 
$\pd_{\cO^\mu_\lambda}\, L(x\cdot\lambda)=\gl\,  P^\lambda(x^{-1} w_0\cdot\widehat\mu) $.
\item\label{proplinks.2} 
$\pd_{\cO^\mu_\lambda}\, \Delta^\mu(x\cdot\lambda)=\gl\, \Delta^\lambda(x^{-1}w_0 \cdot\widehat\mu)$.
\item\label{proplinks.3} 
$\pd_{\cO^\mu_\lambda}\,  L(x\cdot\lambda)=
\frac{1}{2}\gl\,  T^\lambda(w_0^\lambda x^{-1}w_0^\mu\cdot\mu)+\aaa(w_0w_0^\lambda)-\aaa(w_0^\mu)$.
\item\label{proplinks.4} 
$\pd_{\cO^\mu_\lambda}\, \nabla^\mu(x\cdot\lambda)=
\gl\,  \Delta^\lambda(w_0^\lambda x^{-1}w_0^\mu\cdot\mu)+\aaa(w_0w_0^\lambda)-\aaa(w_0^\mu)$.
\end{enumerate}
\end{proposition}

Before proving this we note the following immediate corollary.

\begin{corollary}\label{corglpd}
For all $x\in X^\mu_\lambda$, we have
\begin{enumerate}[$($i$)$]
\item\label{corglpd.1} $\gl\, P^\mu(x\cdot\lambda)=\frac{1}{2}\gl\, T^\mu(w_0^\mu x w_0^\lambda w_0 \cdot\widehat\lambda)+\aaa(w_0w_0^\mu)-\aaa(w_0^\lambda)$,
\item\label{corglpd.2} $\pd_{\cO^\mu_\lambda}\,\nabla^\mu(x\cdot\lambda)=\pd_{\cO^{\widehat\mu}_\lambda}\, \Delta^{\widehat\mu}(w_0 w_0^\mu xw_0^\lambda\cdot\lambda)+\aaa(w_0w_0^\lambda)-\aaa(w_0^\mu)$.
\end{enumerate}
\end{corollary}

We also have the following bounds for projective dimensions.

\begin{proposition}\label{DeltaLnew} 
Consider arbitrary $\lambda,\mu\in\intdom$.
\begin{enumerate}[$($i$)$]
\item\label{DeltaLnew.1} For arbitrary simple and standard modules $L$ and $\Delta$ in 
$\cO_\lambda^\mu$, we have
\begin{displaymath}
\pd_{\cO^\mu_\lambda}\,\Delta\;\,\le\;\, \aaa(w_0w_0^\lambda)-\aaa(w_0^\mu)\;\,\le \;\,\pd_{\cO^\mu_\lambda}\, L. 
\end{displaymath}
\item\label{DeltaLnew.2} The equality $\pd_{\cO^\mu_\lambda}\, L=\aaa(w_0w_0^\lambda)-\aaa(w_0^\mu)$ 
holds if and only if the simple module $L$ is a standard module.
\item\label{DeltaLnew.3} The equality $\pd_{\cO^\mu_\lambda}\, \Delta^\mu(x\cdot\lambda)=\aaa(w_0w_0^\lambda)-\aaa(w_0^\mu)$, 
for $x\in X^\mu_\lambda$, holds if and only if $x\in w_0^\mu \mathbf{L}(w_0 w_0^\lambda)w_0^\lambda$.
\end{enumerate}
\end{proposition}

Now we prove all these statements.

\begin{proof}[Proof of Proposition~\ref{proplinks}]
Claims~\eqref{proplinks.1} and \eqref{proplinks.2} follow immediately from 
the combination of Lemma~\ref{lemKPhi} and Proposition~\ref{exchange}.

As $A_\lambda^\mu$ is Koszul, for any $x\in X_\lambda^\mu$ 
and any simple module $L$ in~$\cO_\lambda^\mu$, we have: 
\begin{displaymath}
\Ext^j_{\cO^\mu_\lambda}(L(x\cdot\lambda),L)\cong\ext^j_{\cO^\mu_\lambda}(L(x\cdot\lambda),L\langle j\rangle). 
\end{displaymath}
Going to the derived category and using equation~\eqref{KRdfun} and Lemma~\ref{lemKPhi}, yields
\begin{displaymath}
\hom_{\cD^b(\cO^\mu_\lambda)}\left(L(x\cdot\lambda)^\bullet,L^\bullet[j]
\langle j\rangle\right)=\hom_{\cD^b(\cO^\lambda_{\mu})}
\left(T^\bullet\langle j\rangle, T^\lambda(w_0^\lambda x^{-1}w_0^\mu\cdot\mu)^\bullet\right), 
\end{displaymath}
for some tilting module $T$ in~$\cO^\lambda_\mu$. Now, set $p(x)$ to equal the extremal non-zero degree 
of $T^\lambda(w_0^\lambda x^{-1}w_0^\mu\cdot\mu)$. As tilting modules are self-dual with 
respect to~$\dd$, we have $\gl T^\lambda(w_0^\lambda x^{-1}w_0^\mu\cdot\mu)=2p(x)$. 
Lemma~\ref{corTproj} then implies that
\begin{displaymath}
\pd_{\cO^\mu_\lambda}L(x\cdot\lambda)\le p(x)+\aaa(w_0^\lambda w_0)-\aaa(w_0^\mu). 
\end{displaymath}
Now consider a simple subquotient $L'$ in extremal degree of $T^\lambda(w_0^\lambda x^{-1}w_0^\mu\cdot\mu)$. 
This must be in the socle, so, in particular, in the socle of a standard module in a standard filtration
of $T^\lambda(w_0^\lambda x^{-1}w_0^\mu\cdot\mu)$. Lemma~\ref{lemIrving} then implies that the 
indecomposable projective cover of $L'$ is a tilting module with graded length given by 
$2\aaa(w_0^\lambda w_0)-2\aaa(w_0^\mu)$. We can set $T$ equal to this tilting module showing 
that the above inequality is, in fact, an equality. This proves claim~\eqref{proplinks.3}.

Now we consider a linear complex $\cT^\bullet_\nabla$ of tilting modules for $\nabla^\mu(x\cdot\lambda)$ 
and a linear complex $\cT^\bullet_L$ of tilting modules for some arbitrary simple module $L$. 
Both exist, see e.g. \cite[Corollary~9.10]{CM4}. Then we have
\begin{displaymath}
\Ext^j_{\cO^\mu_\lambda}(\nabla^\mu(x\cdot\lambda),L)\cong 
\Hom_{\cD^b(\cO_\lambda^\mu)}(\cT^\bullet_\nabla,\cT^\bullet_L[j]). 
\end{displaymath}
The homomorphisms between the two complexes in the right-hand side can be computed in 
the homotopy category $K^b(\cO^\mu_\lambda)$ by \cite[Lemma~III.2.1]{Happel}. 
From \cite[Corollary 9.10 and Lemma~9.11]{CM4}, we therefore find that
\begin{displaymath}
\pd_{\cO_\lambda^\mu}\,\nabla^\mu(x\cdot\lambda)\le \gl \,
\Delta^\lambda(w_0^\lambda x^{-1} w_0^\mu\cdot\mu)+\frac{1}{2}\gl T^\lambda_\mu, 
\end{displaymath}
where $T^\lambda_\mu$ is the characteristic tilting module in~$\cO^\lambda_\mu$. For the latter module,
we have $\gl\, T^\lambda_\mu\le 2\aaa(w_0^\lambda w_0)-2\aaa(w_0^\mu)$ by Lemma~\ref{corTproj}. 
On the other hand, we can apply equation~\eqref{KRdfun} to obtain
\begin{displaymath}
\Ext^j_{\cO^\mu_\lambda}(\nabla^\mu(x\cdot\lambda),L)\cong \bigoplus_{i\in\mZ}\hom_{\cD^b(\cO^\lambda_\mu)}\left(T^\bullet[i-j]\langle i\rangle,\Delta^\lambda(w_0^\lambda x^{-1}w_0^\mu\cdot\mu)\right),
\end{displaymath}
with $T$ the tilting module $\Phi^\mu_\lambda(L)$.
We claim that the summand for $i=j$ on the right-hand side of the above is non-zero for 
$j=\gl\Delta^\lambda(w_0^\lambda x^{-1} w_0^\mu\cdot\mu)+\aaa(w_0^\lambda w_0)-\aaa(w_0^\mu)$, 
which proves claim~\eqref{proplinks.4}. Indeed, as in the proof of claim~\eqref{proplinks.3},
we can take a simple module in the socle of $\Delta^\lambda(w_0^\lambda x^{-1} w_0^\mu\cdot\mu)$ 
and use its projective cover as $T$, by Lemma~\ref{lemIrving}.
\end{proof}

\begin{lemma}\label{lempdLalt}
For any $x\in X^\mu_\lambda$, the quantity
\begin{displaymath}
\pd_{\cO_\lambda^\mu}\, L(x\cdot\lambda)-\aaa(w_0w_0^\lambda)+\aaa(w_0^\mu) 
\end{displaymath}
is given by the maximum, over $y\in X^\mu_\lambda$, of the values
\begin{displaymath}
\pd_{\cO^\mu_{\widehat\lambda}} \, \Delta^\mu(w_0^\mu yw_0^\lambda w_0\cdot\widehat\lambda)-\min\{j\in\mN\,|\, \Ext^j_{\cO_\lambda^\mu}(\Delta^\mu(y\cdot\lambda),L(x\cdot\lambda))\not=0)\}. 
\end{displaymath}
\end{lemma}

\begin{proof}
We freely use the standard properties of (graded) quasi-hereditary algebras
\begin{displaymath}
\hom_{\cO^\lambda_\mu}(T^\lambda(x\cdot\mu),\nabla^\lambda(y\cdot\mu)\langle- j\rangle)=
(T^\lambda(x\cdot\mu):\Delta^\lambda(y\cdot\mu)\langle -j\rangle)
\end{displaymath}
and
\begin{displaymath}
\Ext^k_{\cO^\lambda_\mu}(T^\lambda(x\cdot\mu),\nabla^\lambda(y\cdot\mu))=0,\qquad \mbox{ for }\; k>0. 
\end{displaymath}

By \cite[Lemma~2.4]{MaTilt}, we have
\begin{equation}\label{glTDelta}
  \begin{aligned}
&\frac{1}{2}\gl \, T^\lambda(x\cdot\mu)=\\
&\max_{y\in X^\lambda_\mu}\left(\gl \, \Delta^\lambda(y\cdot\mu)-\min_{j\in\mN}\left\{\hom_{\cO_\mu^\lambda}(T^\lambda(x\cdot\mu),\nabla^\lambda(y\cdot\mu)\langle- j\rangle)\not=0\right\}\right).
  \end{aligned}
\end{equation}
By the above vanishing of extensions, the homomorphism in the second line of
  can be calculated in the derived category. 
Then we apply equation~\eqref{KRdfun} and Lemma~\ref{lemKPhi} to obtain
\begin{eqnarray*}
&&\hom_{\cD^b(\cO^\lambda_\mu)}(T^\lambda(x\cdot\mu)^\bullet,\nabla^\lambda(y\cdot\mu)^\bullet\langle -j\rangle)\cong \\
&&\hom_{\cD^b(\cO^\mu_\lambda)}(\Delta^\mu(w_0^\mu y^{-1}w_0^\lambda\cdot\lambda)^\bullet, 
L(w_0^\mu x^{-1} w_0^\lambda \cdot\lambda)^\bullet [j]\langle j\rangle).
\end{eqnarray*}
Applying this and Proposition~\ref{proplinks}\eqref{proplinks.3} to equation \eqref{glTDelta} yields
\begin{eqnarray*}
&&\pd_{\cO_\lambda^\mu}L(x\cdot\lambda)-\aaa(w_0w_0^\lambda)+\aaa(w_0^\mu)=\\
&&\max_{y\in X^\lambda_\mu}\left(\gl \, \Delta^\lambda(y\cdot\mu)-\min_{j\in\mN}\left\{
\Ext^j_{\cO^\mu_\lambda}(\Delta^\mu(w_0^\mu y^{-1}w_0^\lambda\cdot\lambda), 
L(x\cdot\lambda))\not=0\right\}\right),
\end{eqnarray*}
where we also used the fact that $\cO_\lambda^\mu$ is standard Koszul.

Application of Proposition~\ref{proplinks}\eqref{proplinks.2} then concludes the proof.
\end{proof}

\begin{lemma}\label{Set} 
Set $S^\mu_\lambda:=X_\lambda^\mu \;\cap\; w_0^\mu \mathbf{R}(w_0^\mu)w_0^\lambda$.
\begin{enumerate}[$($i$)$]
\item\label{Set.1} Every simple module in $\cO_\lambda^\mu$ is a subquotient of the module
\begin{displaymath}
\bigoplus_{x\in S^\mu_\lambda}\,\Delta^\mu(x\cdot\lambda). 
\end{displaymath}
\item\label{Set.2} For any $x\in X_\lambda^\mu$, we have 
\begin{displaymath}
\begin{cases}
\gl\, \Delta^\mu(x\cdot\lambda)=\aaa(w_0^\mu w_0)-\aaa(w_0^\lambda),&\mbox{if }x\in S^\mu_\lambda;\\ 
\gl\, \Delta^\mu(x\cdot\lambda) <\aaa(w_0^\mu w_0)-\aaa(w_0^\lambda),&\mbox{otherwise}.
\end{cases} 
\end{displaymath}
\end{enumerate}
\end{lemma}

Note that $S_\lambda=\{w_0^\lambda\}$ and $S_0^\mu=w_0^\mu\mathbf{R}(w_0^\mu)$.

\begin{proof}
Consider an arbitrary simple object $L$ in $\cO_\lambda^\mu$ and the standard module 
$\Delta$ which has simple top $L$. Take a smallest quotient $\widetilde\Delta$ 
of $\Delta$ which still contains a simple subquotient which has an injective module 
as projective cover. This quotient $\widetilde\Delta$ exists by Lemma~\ref{lemIrving} 
and has simple socle which we denote by $L'$. By construction, $L'$ has, as injective 
envelope, a projective-injective module $I'$. By Lemma~\ref{lemIrving}, $I'$ is a a tiling module. 
We denote by $\Delta'$ the unique standard module which injects into $I'$ such that the 
quotient has a standard flag. 

Now we have two submodules $\widetilde\Delta$ and $\Delta'$ of $I'$. 
We claim that $\widetilde\Delta$ is a submodule of $\Delta'$. 
Indeed, the module $Q$ defined by the short exact sequence
\begin{displaymath}
0\to \Delta'\to I'\to Q\to 0,
\end{displaymath}
has a standard filtration. In particular, the socle of $Q$ consists of simple modules
whose projective cover is injective by Lemma~\ref{lemIrving}\eqref{lemIrving.4}. 
By construction, $L'$ does not appear in the socle of $(\Delta'+ \widetilde \Delta)/\Delta'\cong
\widetilde \Delta/(\Delta'\cap \widetilde \Delta)$.
Since $L'$ was the only simple subquotient of $\widetilde \Delta$ whose projective cover is injective,
we get $\mathrm{Hom}_{\mathcal{O}}((\Delta'+ \widetilde \Delta)/\Delta',Q)=0$. This means exactly that
$\widetilde \Delta\subset \Delta'$.

The inclusion $\widetilde\Delta\subset\Delta'$ implies that $L$, being the top of $\widetilde\Delta$, 
must be a subquotient of $\Delta'$. Now, by Lemma~\ref{corTproj} and the construction of $I'$ and 
$\Delta'$, we have $\Delta'\cong\Delta^\mu(x\cdot\lambda)$ for some $x\in S_\lambda^\mu$. As $L$ was chosen arbitrarily, this concludes 
the proof of claim~\eqref{Set.1}.

Next we prove claim~\eqref{Set.2}. Equation \eqref{glTDelta} implies 
\begin{displaymath}
\gl\, \Delta^\mu(x\cdot\lambda)\;\le \;\frac{1}{2}\gl\, T^\mu(x\cdot\lambda). 
\end{displaymath}
So, by Lemma~\ref{corTproj}, the graded length of standard modules in $\cO^\mu_\lambda$ 
is bounded by $\aaa(w_0w_0^\mu)-\aaa(w_0^\lambda)$ and this value can only be reached 
for $x\in S^\mu_\lambda$. On the other hand, equation \eqref{glTDelta} shows that, for 
$T^\mu(x\cdot\lambda)$ to have the maximal graded length amongst tilting modules, the 
corresponding standard module must also have maximal graded length. This completes the proof. 
\end{proof}

\begin{remark}\label{remarkwfunction}
{\rm
Let $\mu\in\Lambda_{\mathrm{int}}^+$. Then projective-injective modules in $\mathcal{O}_0^{\mu}$
are indexed by elements in $\mathbf{R}(w_0^{\mu}w_0)$, see Lemma \ref{lemIrving}. Each indecomposable projective-injective
module $P^{\mu}(x)$ is also tilting and hence isomorphic to some $T^{\mu}(\psi^\mu(x))$. The set of
$\psi^\mu(x)$ which appear in this way is exactly $w_0^{\mu}\mathbf{R}(w_0^{\mu}w_0)w_0=S_0^{\mu}$,
{\it cf.} Lemma~\ref{Set}. However, it is not true, in general, that $\psi^\mu(x)=w_0^{\mu}xw_0$.
For example, in case $\mathfrak{g}=\mathfrak{sl}_3$ and $w_0^{\mu}=s_1$, we have 
$\mathbf{R}(w_0^{\mu}w_0)=\{s_2,s_2s_1\}$, $S_0^{\mu}=\{s_2,e\}$ and 
$\psi^\mu(s_2)=e\neq s_1\cdot s_2\cdot s_1 s_2s_1$. It would be interesting to find an explicit formula for the bijection $\psi^\mu:\mathbf{R}(w_0^{\mu}w_0)\to w_0^{\mu}\mathbf{R}(w_0^{\mu})$. Note that we have the following alternative description: or any $x\in \mathbf{R}(w_0^\mu w_0)$, $L(x)$ is the simple socle of $\Delta^\mu(\psi^\mu(x))$. In particular, $\psi^\mu(d_\mu)=e$, with $d_\mu$ the Duflo involution in $\mathbf{R}(w_0^{\mu}w_0)$.
}
\end{remark}

\begin{proof}[Proof of Proposition~\ref{DeltaLnew}]
First we prove claim~\eqref{DeltaLnew.1}. The inequality for simple modules follows from 
Proposition~\ref{proplinks}\eqref{proplinks.3}. The inequality for standard modules follows from 
Lemma~\ref{Set}\eqref{Set.2}.

The Koszul dual of claim~\eqref{DeltaLnew.2} is, according to Lemma~\ref{lemKPhi} and 
Proposition~\ref{exchange}, the statement that the graded length of an indecomposable projective 
object in $\cO^\lambda_{\widehat\mu}$ is given by $\aaa(w_0w_0^\lambda)-\aaa(w_0^\mu)$ if and only 
if it is a standard module. The combination of Lemma~\ref{Set} and the BGG reciprocity in 
\cite[Theorem~9.8]{Humphreys} implies that every projective module in $\cO^\lambda_{\widehat\mu}$ 
must contain a standard module with graded length $\aaa(w_0w_0^\lambda)-\aaa(w_0^\mu)$ as a subquotient 
in its standard filtration. By positivity of the grading, the fact that the graded length of the projective 
module is exactly $\aaa(w_0w_0^\lambda)-\aaa(w_0^\mu)$ hence implies that it is isomorphic to 
such a standard module. This proves claim~\eqref{DeltaLnew.2}.

Claim~\eqref{DeltaLnew.3} follows from Lemma~\ref{Set}\eqref{Set.2}.
\end{proof}

The arguments in this subsection lead to the following observation.

\begin{lemma}\label{onlyKLVcell}
Take $M$ to be a simple, standard or costandard module in $\cO^\mu_\lambda$ and denote its projective dimension by $p=\pd_{\cO^\mu_\lambda}\, M$. Then, for any $y\in X_\lambda^\mu$, we have
\begin{displaymath}
\Ext^p_{\cO_\lambda^\mu}(M,L(y\cdot\lambda))\not=0\qquad\Rightarrow\qquad y\in \LL(w_0^\lambda).
\end{displaymath}
\end{lemma}

\begin{proof}
First, take $M$ to be a costandard module. By the proof of Proposition \ref{proplinks}\eqref{proplinks.4}, 
in order to have an extension with a simple module $L$ in the maximal possible degree, $L$ 
needs to be such that $\Phi^\mu_\lambda(L)$ is a projective tilting module. 
Lemmata~\ref{lemKPhi} and \ref{corTproj} therefore show that $L=L(y\cdot\lambda)$ 
with $y^{-1}\in \RR(w_0^\lambda)$.

Now set $M=\Delta^\mu(x\cdot\lambda)$. Assume $y\in X_\lambda^\mu$ is such that 
\begin{displaymath}
\Ext^p_{\cO_\lambda^\mu}(\Delta^\mu(x\cdot\lambda),L(y\cdot\lambda))\neq 0.
\end{displaymath}
Koszul duality, see
e.g. \cite[Proposition~1.3.1]{BGS}, then implies that $L(y^{-1}w_0\cdot\widehat{\mu})$ appears in 
the maximal degree of $\Delta^\lambda(x^{-1}w_0\cdot\widehat{\mu})$. Hence 
$L(y^{-1}w_0\cdot\widehat{\mu})$ appears in the 
socle of a standard module in~$\cO_{\widehat{\mu}}^\lambda$. Lemma~\ref{lemIrving}\eqref{lemIrving.4} 
and \eqref{lemIrving.5} thus imply that $y^{-1}w_0\sim_R w_0^\lambda w_0$. 
The result hence follows from equation~\eqref{LRorder}.

Finally, the claim for a simple module follows from the case of a standard module and 
\cite[Corollary~3.9]{CPS2}.
\end{proof}

In particular, this means that equation \eqref{pdDKLV} can be simplified to
\begin{equation}\label{onlyKLVcelleq}
\llll(x)\;=\;\max_{y\in\mathbf{L}(w_0^\lambda)}\deg p_\lambda(x,y). 
\end{equation}

\subsection{Applying translation functors}\label{secconnec.3}

\begin{proposition}\label{proptrans}
For any $x\in X^\mu_\lambda$, we have
\begin{enumerate}[$($i$)$]
\item\label{proptrans.1} $\gl\,  \Delta^\mu(x\cdot\lambda)=\gl\,  \Delta^\mu(xw_0^\lambda)-\len(w_0^\lambda)$;
\item\label{proptrans.2} $\gl\,  T^\mu(x\cdot\lambda)=\gl\,  T^\mu(xw_0^\lambda)-2\len(w_0^\lambda)$;
\item\label{proptrans.3} $\pd_{\cO_\lambda^\mu}\,  T^\mu(x\cdot \lambda)=\pd_{\cO^\mu_0}\,  T^\mu(xw_0^\lambda)$;
\item\label{proptrans.4} $\pd_{\cO_\lambda^\mu}\, I^\mu(x\cdot\lambda)=\pd_{\cO^\mu_0}\,  I^\mu(x)$.
\end{enumerate}
\end{proposition}

\begin{proof}
Corollary~\ref{glidentity} and \cite[Theorem~5.5]{CM4} imply
\begin{equation}\label{Anna}
\gl\,  \Delta^\mu(x\cdot\lambda)=\max_{u\in W_\lambda}\{\gl\,  \Delta^\mu(xu)+l(u)\}-2\len(w_0^\lambda).
\end{equation}
This proves $\gl\,  \Delta^\mu(x\cdot\lambda)\le \gl \, \Delta^\mu(xw_0^\lambda)-\len(w_0^\lambda)$. On the other hand, \cite[Theorem~5.4]{CM4} yields
\begin{displaymath}
\theta^{on}_\lambda \Delta^\mu(xw_0^\lambda)\langle 0\rangle=\Delta^\mu(x\cdot\lambda)\langle 0\rangle.
\end{displaymath}
By equation~\eqref{normon}, this means that $L(x)$, the simple subquotient of lowest degree in $\Delta^\mu(xw_0^\lambda)$ which 
does not get canceled by $\theta^{on}_\lambda$, sits in degree $\len(w_0^\lambda)$. This implies the inequality in the other direction and concludes the proof of 
claim~\eqref{proptrans.1}.

Claim~\eqref{proptrans.2} follows immediately from equation~\eqref{tiltingout} and Corollary~\ref{glidentity}.

Equations \eqref{projout} and \eqref{tiltingout}, in combination with the identity $\theta_\lambda^{on}\theta_\lambda^{out}\cong \Id^{\oplus |W_\lambda|}$, show that 
translating out of and onto the wall exchange multiples of, on the one hand, 
$I^\mu(x\cdot\lambda)$ and $I^\mu(x)$ and, on the other hand, $T^\mu(x\cdot\lambda)$ 
and $T^\mu(xw_0^\lambda)$. As translation functors are exact and preserve the 
categories of projective modules, this implies claim~\eqref{proptrans.3} and \eqref{proptrans.4}.
\end{proof}

\begin{corollary}\label{corpdDelta}
{\hspace{2mm}}

\begin{enumerate}[$($i$)$]
\item\label{corpdDelta.1} For any $x\in X^\mu_\lambda$, we have
\begin{displaymath}
\pd_{\cO^\mu_\lambda}\,\Delta^\mu(x\cdot\lambda)=
\pd_{\cO_\lambda}\, \Delta(w_0^\mu x\cdot\lambda)-\len(w_0^\mu). 
\end{displaymath}
\item\label{corpdDelta.2} For any simple module $L$ in~$\cO^\mu_\lambda$ and $j\in\mN$, we have 
\begin{displaymath}
\Ext^{j}_{\cO^\mu_\lambda}(\Delta^\mu(x\cdot\lambda),L)\cong 
\Ext^{j+\len(w_0^\mu)}_{\cO_\lambda}(\Delta(w_0^\mu x\cdot\lambda),L).
\end{displaymath}
\end{enumerate}
\end{corollary}

\begin{proof}
Claim~\eqref{corpdDelta.1} follows from the combination of Proposition~\ref{proptrans}\eqref{proptrans.1}
with Proposition~\ref{proplinks}\eqref{proplinks.2}.

The isomorphism
\begin{displaymath}
\Ext^{j}_{\cO^\mu_\lambda}(\Delta^\mu(x\cdot\lambda),L)\cong \Ext^{j}_{\cO_\lambda}(\Delta( x\cdot\lambda),L). 
\end{displaymath}
follows from \cite[Theorem 5.15]{CM4} and the adjunction between the derived Zuckerman functor and 
the parabolic inclusion functor. Now, as $L$ is $s$-finite for any simple reflection $s\in W_\mu$, 
we can use the computation in the proof of \cite[Proposition~3]{SHPO1}, which can be applied 
to singular blocks by the results in \cite[Section~5]{CM1}, to obtain 
\begin{displaymath}
\Ext^{j}_{\cO_\lambda}(\Delta( x\cdot\lambda),L)\cong 
\Ext^{j+\len(w_0^\lambda)}_{\cO_\lambda}(\Delta( w_0^\lambda x\cdot\lambda),L). 
\end{displaymath}
This proves claim~\eqref{corpdDelta.2}.
\end{proof}

\begin{lemma}\label{cor0}
For $\lambda,\mu\in\intdom$, we have:
\begin{enumerate}[$($i$)$]
\item\label{cor0.1} $\pd_{\cO_\lambda}\, \Delta(x\cdot\lambda)=0$, for $x\in X_\lambda$, if and only if $x=w_0^\lambda$;
\item\label{cor0.2} $\Delta^\mu(x)=L(x)$, for $x\in X^\mu$, if and only if $x=w_0^\mu w_0$;
\item\label{cor0.3} $T^\mu(x)=L(x)$, for $x\in X^\mu$, if and only if $x=w_0^\mu w_0$;
\item\label{cor0.4} $\pd_{\cO_\lambda}\, L(x\cdot\lambda)=\aaa(w_0w_0^\lambda)$, for $x\in X_\lambda$,
if and only if $x=w_0$.
\end{enumerate}
\end{lemma}

\begin{proof}
The Verma module $\Delta(w_0^\lambda\cdot\lambda)$ is projective. Now, assume that 
$\Delta(x\cdot\lambda)$ is projective for some $x\in X_\lambda$. Then 
$\theta_{\lambda}^{out}\Delta(x\cdot\lambda)$ must be projective. 
As, for any projective module $P(y)$ in~$\cO_0$, by the BGG reciprocity, we have
\begin{displaymath}
(P(y):\Delta(e))=[\Delta(e):L(y)]\not=0,
\end{displaymath}
the module $\Delta(e)$ must appear as a subquotient of a 
standard filtration of the module $\theta_{\lambda}^{out}\Delta(x\cdot\lambda)$. 
Claim~\eqref{cor0.1} therefore follows from \cite[Theorem~5.5]{CM4}.

Claim~\eqref{cor0.2} follows from claim~\eqref{cor0.1} by 
Proposition~\ref{proplinks}\eqref{proplinks.1}. 
Claim~\eqref{cor0.3} follows immediately from claim~\eqref{cor0.2}. 
Claim~\eqref{cor0.4} follows from claim~\eqref{cor0.3} by Proposition~\ref{proplinks}\eqref{proplinks.3}.
\end{proof}

\begin{lemma}\label{lemsgl}
For a simple reflection $s$ and $x\in X^\mu$ such that $xs>x$ and $xs\in X^\mu$, we have
\begin{displaymath}
\gl\, \Delta^\mu(xs)\le \gl\, \Delta^\mu(x)\le \gl\,\Delta^\mu(xs)+1. 
\end{displaymath}
\end{lemma}

\begin{proof}
From \cite[Theorems~5.4 and~5.5]{CM4}, we find a short exact sequence
\begin{displaymath}
0\to \Delta^\mu(x)\langle 1\rangle\to \theta_s \Delta^\mu(x)\to \Delta^\mu(xs)\to 0,
\end{displaymath}
where  $\theta_s \Delta^\mu(x)\cong \theta_s \Delta^\mu(xs)$.
Set $d=\gl\, \Delta^\mu(x)$. Note that $\theta_s L(x)=0$ by our assumptions.
Then $\theta_s \Delta^\mu(x)$ is concentrated 
between degrees $0$ and $d+1$ and hence has graded length at most $d+1$. Furthermore, 
every simple module appearing in the maximal degree $d+1$ 
of $\theta_s \Delta^\mu(x)$ must appear in the maximal degree $d$ of $\Delta^\mu(x)$
with at least the same multiplicity. 
This implies that the natural injection from $\Delta^\mu(x)\langle 1\rangle_{d+1}$  to 
$(\theta_s \Delta^\mu(x))_{d+1}$ is, in fact, a bijection. Consequently,
$\gl\, \Delta^\mu(xs)\le d$.

Now, the centre $\cZ(\fg)$ acts diagonalisably on both $\Delta^\mu(xs)$ and $\Delta^\mu(x)$, 
but not on a non-zero module $\theta_s L$, for $L$ a simple object in~$\cO_0$. Indeed, the unique (up to scalar) 
non-zero map from the top to the socle of $\theta_s L$ can be viewed as the evaluation 
at $L$ of the  endomorphism of the functor $\theta_s$ given by composition of the adjunction morphisms
$\theta_s\to\theta_e\to\theta_s$.  By \cite[Theorem~3.5]{BG}, this corresponds to the
nilpotent endomorphism of $P(s)$ which is given, due to \cite[Theorem~7.1]{Stroppel2}, by the action of 
$\cZ(\fg)$. 

The fact that $\cZ(\fg)$ does not act diagonalisably on $\theta_s L$,
implies, for $L\in \Delta^\mu(x)_d$, that there must be something in degree $d-1$ 
of $\theta_s \Delta^\mu(x)$ which survives the projection onto $\Delta^\mu(xs)$. 
This gives $\gl\,\Delta^\mu(xs)\ge d-1$.
\end{proof}

\begin{lemma}
For any $x\in X^\mu$, we have
\begin{displaymath}
\gl\, P^\mu(x)\;=\;\frac{1}{2}\gl\, \theta_x L(d_{\mu})+\aaa(w_0^\mu w_0),
\end{displaymath}
where $d_\mu$ is the Duflo involution in~$\mathbf{R}(w_0^\mu w_0)$.
\end{lemma}

\begin{proof}
We have $P^\mu(x)\langle 0\rangle\cong\theta_x \Delta^\mu(e)\langle 0\rangle$. 
Furthermore, for $y\in\mathbf{R}(w_0^\mu w_0)$, the only subquotient of
$\Delta^\mu(e)$ of the form $L(y)$ 
is, by definition, $L(d_{\mu})$ appearing in degree $\aaa(w_0^\mu w_0)$. 
By Lemma~\ref{lemIrving}, the graded length of $P^\mu(x)$ is given by the 
highest degree in which some $L(y)$ with $y\in\mathbf{R}(w_0^\mu w_0)$ appears. 
Now, $\theta_x$ acting on a arbitrary simple module $L(z)$ gives a module in 
which all appearing submodules $L(w)$ satisfy $w\le_R z$. Hence the only simple 
subquotients in~$\theta_x \Delta^\mu(e)$ of the form $L(y)$, where $y\in \mathbf{R}(w_0^\mu w_0)$, 
must come from $\theta_x L(d_\mu)\langle \aaa(w_0^\mu w_0)\rangle$. As 
$\theta_x L(d_\mu)$ is a self-dual module, the claim follows.
\end{proof}

\subsection{Applying Zuckerman functors}\label{secconnec.4}

\begin{proposition}\label{propparinjtilt}
For any $x\in X^\mu$, we have
\begin{enumerate}[$($i$)$]
\item\label{propparinjtilt.1} $\pd_{\cO_0^\mu} I^\mu(x)\;= \; \pd_{\cO_0}I(x)\,-\,2\len(w_0^\mu)$;
\item\label{propparinjtilt.2} $\pd_{\cO^\mu_0}T^\mu (x)\;=\; \pd_{\cO_0}T(w_0^\mu x)-\len(w_0^\mu)$.
\end{enumerate}
\end{proposition}

Before proving this proposition, we need two preparatory lemmata.

\begin{lemma}\label{lemZuck}
For any $x\in X^\mu$, we have
\begin{enumerate}[$($i$)$]
\item\label{lemZuck.1} $\cL Z^\mu (I(x)^\bullet)\;\cong\; I^\mu(x)^\bullet[2\len(w_0^\mu)]$;
\item\label{lemZuck.2} $\cL Z^\mu (T(w_0^\mu x)^\bullet) \cong T^\mu(x)^\bullet[\len(w_0^\mu)]$.
\end{enumerate}
\end{lemma}

\begin{proof}
To prove claim~\eqref{lemZuck.1}, it suffices to consider the case $x=e$, 
as $I^\mu(x)=\theta_x I^\mu(e)$ and $I(x)=\theta_x I(e)$, for any $x\in X^\mu$, 
and Zuckerman functors commute with projective functors. 
From \cite[Propositions~4.1 and~4.2]{EW}, we find that
\begin{displaymath}
\cL_k Z^\mu\, \nabla(e)\cong
\begin{cases}
\dd \,\cL_{2\len(w_0^\mu)-k}Z^\mu\,\Delta(e),&\mbox{if }k\le 2\len(w_0^\mu);\\ 
0,&\mbox{if }k> 2\len(w_0^\mu).
\end{cases} 
\end{displaymath}
It is well-known that $\cL_j Z^\mu \Delta(e)\cong \delta_{j0}\Delta^\mu(e)$, 
see for instance \cite[Theorem~5.15]{CM4}. The result hence follows 
by observing that $I(e)=\nabla(e)$ and $I^\mu=\nabla^\mu(e)$.

Equation~\eqref{tiltingtheta} and \cite[Theorem 5.15]{CM4} imply that
\begin{displaymath}
\cL Z^\mu T(w_0^\mu x)^\bullet\cong 
\theta_{w_0 w_0^\mu x}L(w_0^\mu w_0)^\bullet [\len(w_0^\mu)]\cong T^\mu(x)^\bullet[\len(w_0^\mu)]. 
\end{displaymath}
This proves claim~\eqref{lemZuck.2}.
\end{proof}

For the next lemma we note that $w_0^\mu X^\mu=\{x\in W\,|\, w_0^\mu\le_R x\}$ 
is a collection of right cells, which follows from equation~\eqref{altXmu}.

\begin{lemma}\label{lemconsRight}
For all $x,y\in X^\mu$, we have
\begin{enumerate}[$($i$)$]
\item\label{lemconsRight.1} $x\le_R y\;\Rightarrow\; \pd_{\cO_0^\mu}\, I^\mu(x)\ge \pd_{\cO^\mu_0}\, I^\mu(y)$;
\item\label{lemconsRight.2} $w_0^\mu x\le_R w_0^\mu y\;\Rightarrow\; \pd_{\cO_0^\mu}\,T^\mu(x)\le 
\pd_{\cO^\mu_0}\, T^\mu(y)$.
\end{enumerate}
Consequently, both the function $x\mapsto \pd_{\cO_0^\mu}\,I^\mu(x)$, where $x\in X^\mu$, and the function 
$y\mapsto \pd_{\cO^\mu_0}\,T^\mu(w_0^\mu y)$, where $y\in w_0^\mu X^\mu$, are constant on right cells.
\end{lemma}

\begin{proof}
Consider $x,y\in W$ and $M\in \cO_0$. We claim that $x\le_R y$ implies
\begin{displaymath}
\pd_{\cO_0}\, \theta_x M\ge\pd_{\cO_0}\,\theta_y M.
\end{displaymath}
This is a standard consequence of the connection between the composition of projective 
functors and the right KL order, see e.g. \cite[Equation~(1)]{SHPO1}. This connection
means the following: if $x\le_R y$, then there is some projective functor $\theta$ on $\cO_0$ 
such that $\theta_y$ is a direct summand of $\theta\circ \theta_x$. Consequently, 
$\theta_y M$ is a direct summand of $\theta \theta_x M$. As $\theta$ is an exact 
functor preserving projectivity of modules, the bound on the projective dimensions follows.

The statements in the lemma are direct consequences of the above paragraph, 
by equations~\eqref{deftheta} and~\eqref{tiltingtheta}.
\end{proof}

\begin{proof}[Proof of Proposition~\ref{propparinjtilt}]
For $j\in\mN$ and $M\in \cO_0^\mu$, we consider the extension group
\begin{displaymath}
\Ext^j_{\cO^\mu_0}(I^\mu(x),M)\cong \Hom_{\cD^b(\cO^\mu_0)}(I^\mu(x)^\bullet,M^\bullet[j]). 
\end{displaymath}
By Lemma~\ref{lemZuck}\eqref{lemZuck.1} and adjunction, the latter space can be computed as follows:
\begin{displaymath}
\Hom_{\cD^b(\cO^\mu_0)}(\cL Z^\mu I(x)^\bullet,M^\bullet[j+2\len(w_0^\mu)])\cong 
\Hom_{\cD^b(\cO_0)}( I(x)^\bullet,M^\bullet[j+2\len(w_0^\mu)]). 
\end{displaymath}
We therefore find an isomorphism
\begin{equation}\label{eqconnecext}
\Ext^j_{\cO^\mu_0}(I^\mu(x),M)\;\cong\; 
\Ext^{j+2\len(w_0^\mu)}_{\cO_0}( I(x),M),\qquad \text{ for all }\quad M\in \cO^\mu_0.
\end{equation}

Equation \eqref{eqconnecext} implies immediately that
\begin{displaymath}
\pd_{\cO^\mu_0}I^\mu(x)\;\le \; \pd_{\cO_0} I(x)-2\len(w_0^\mu).
\end{displaymath}
To prove that this is an equality, it suffices to consider some fixed element $x$ for every right cell 
in~$X^\mu$, by Lemma~\ref{lemconsRight}. Hence in each such right cell we can choose $x$ to be the
corresponding Duflo involution. In this case, the proof of \cite[Lemma~23]{SHPO2} 
implies that the extension groups in equation \eqref{eqconnecext} are non-zero for 
\begin{displaymath}
j+2\len(w_0)=\pd_{\cO_0}I(x)=2\aaa(w_0 x) 
\end{displaymath}
and $M=\theta_x\theta_x L(x)\in \cO_0^\mu$. This concludes the proof of claim~\eqref{propparinjtilt.1}.

As in the proof of claim~\eqref{propparinjtilt.1}, Lemma~\ref{lemZuck}\eqref{lemZuck.2} implies
\begin{equation}\label{eqconnecext2}
\Ext^j_{\cO^\mu_0}(T^\mu(x),M)\;\cong\; 
\Ext^{j+\len(w_0^\mu)}_{\cO_0}( T(w_0^\mu x),M),\qquad \text{ for all }\quad  M\in \cO^\mu_0.
\end{equation}
This yields the inequality $\pd_{\cO^\mu_0}T^\mu (x)\le \pd_{\cO_0}T(w_0^\mu x)-\len(w_0^\mu)$. 
To prove that this is actually an equality, by Lemma~\ref{lemconsRight} it suffices to 
prove this for the case where $w_0^\mu x$ is a Duflo involution. We set $w:= w_0^\mu x$ 
and define $y\in W$ to be the unique element in~$\mathbf{R}(w)$ such that $y^{-1}w_0$ 
is a Duflo involution. The proof of \cite[Lemma~19]{SHPO2} then implies that 
\begin{displaymath}
\Ext^{p}_{\cO_0}(T(w_0^\mu x), L(w_0 y^{-1}))\not=0,\qquad \mbox{where}\;\; p=\pd_{\cO_0}\,T(w_0^\mu x). 
\end{displaymath}
Now, by the definition of $y$ and the properties of the KL orders in Subsection~\ref{prelKL}, we have
\begin{displaymath}
w_0y^{-1}=y w_0\sim_R w_0^\mu x w_0,\quad \mbox{where}\quad w_0^\mu x w_0\in X^\mu.
\end{displaymath}
This means that $w_0y^{-1}\in X^\mu$, {\it i.e.} $L(w_0 y^{-1})\in \cO^\mu_0$, so we can 
use $M=L(w_0 y^{-1})$ and $j+\len(w_0^\mu)=\pd_{\cO_0}\,T(w_0^\mu x)$ in 
equation~\eqref{eqconnecext2}. This completes the proof.
\end{proof}

\subsection{Applying twisting and shuffling functors}\label{secconnec.5}

For the principal block, we have the following well-known formula: 
\begin{equation}\label{eqCarlin}
\dim\Ext_{\cO_0}^{\len(x)-\len(y)}(\Delta(x),L(y))=1,\quad\mbox{for all }x,y\in W \mbox{ with }x\ge y,
\end{equation}
see \cite{Carlin}. The corresponding statement is false for singular blocks. 
The origin of this lies in the following statement.

\begin{proposition}\label{propextremalext}
Consider $x\in X_\lambda$ such that $x=sx'$, for a simple reflection $s$ and $x'\in X_\lambda$ 
with $x'<x$. For $y\in X_\lambda$ with $y\le x$, set $n=\len(x)-\len(y)$. Then we have
\begin{displaymath}
\Ext^{n}_{\cO_\lambda}(\Delta(x\cdot\lambda),L(y\cdot\lambda))
=\begin{cases}0,& \mbox{if } y>sy \mbox{ and }sy\not\in X_\lambda;\\
\Ext^{n}_{\cO_\lambda}(\Delta(x'\cdot\lambda),L(y'\cdot\lambda)),&\mbox{if } y>sy=y' \in X_\lambda;\\
\Ext^{n-1}_{\cO_\lambda}(\Delta(x'\cdot\lambda),L(y\cdot\lambda)),&\mbox{if } y<sy.
\end{cases}
\end{displaymath}
\end{proposition}

This leads, by induction on $\len(x)$, to the following analogue of equation~\eqref{eqCarlin}:

\begin{corollary}\label{cordim}
For any $x,y\in X_\lambda$ with $x\ge y$, we have
\begin{displaymath}
\dim\Ext_{\cO_\lambda}^{\len(x)-\len(y)}(\Delta(x\cdot\lambda),L(y\cdot\lambda))\le 1. 
\end{displaymath}
\end{corollary}

\begin{remark}
{\em
Contrary to the principal block $\cO_0$, to determine  the projective dimension of the module 
$\Delta(x\cdot\lambda)$ in general, it is not sufficient to consider extensions of the form
\begin{displaymath}
\Ext^{\len(x)-\len(y)}_{\cO_\lambda}(\Delta(x\cdot\lambda),L(y\cdot\lambda)). 
\end{displaymath}
By equation~\eqref{onlyKLVcelleq} and Proposition~\ref{DeltaLnew}\eqref{DeltaLnew.1}, 
a counterexample is found as soon as
\begin{displaymath}
\aaa(w_0w_0^\lambda)\; <\; \len(w_0)-\max\{\len(x)\,|\, x\in \mathbf{L}(w_0^\lambda)\}. 
\end{displaymath}
This is the case for the examples in Subsections \ref{s2} and \ref{s1s3}.
}
\end{remark}

In the following proof we use the twisting functor $T_s$ and its adjoint $G_s$ as defined in e.g. \cite{AS},
see also \cite{KhMa}.

\begin{proof}[Proof of Proposition~\ref{propextremalext}]
By \cite[Lemma~5.4, Corollary~5.6 and Proposition~5.11]{CM1}, we have
\begin{displaymath}
\Ext^{n}_{\cO_\lambda}(\Delta(x\cdot\lambda),L(y\cdot\lambda))\cong 
\Hom_{\cD^b(\cO_\lambda)}(\Delta(x'\cdot\lambda), \cR G_s L(y\cdot\lambda)[n]),
\end{displaymath}
where, by \cite[Theorem~4.1]{AS},
\begin{displaymath}
\cR G_s L(y\cdot\lambda)= \dd \theta^{on}_{\lambda} \cL T_s L(y). 
\end{displaymath}
Assume that $sy > y$. Then, by \cite[Theorem~5.12(i)]{CM1}, we have
\begin{displaymath}
\cR G_s L(y\cdot\lambda)=L(y\cdot\lambda)[-1]. 
\end{displaymath}

Assume that $sy<y$. Set $y=sy'$ with $\len(y)=1+\len(y')$. Then $\cL T_s L(y)= T_s L(y)$. 
The module $T_s L(y)$ has simple top $L(y)$ and semisimple radical $R$ which is of the form
\begin{displaymath}
R\cong L(y') \,\oplus \bigoplus_i L(z_i), 
\end{displaymath}
where all $z_i$ satisfy $sz_i > z_i$ and $z_i>y'$, see \cite[Theorem~6.3(3) and Section~7]{AS}
(we note that, from \cite[Corollary~5.2.4]{Irving3}, we even have $z_i>y$). 
This means that we have
\begin{displaymath}
\Ext^{n}_{\cO_\lambda}(\Delta(x\cdot\lambda),L(y\cdot\lambda))\cong 
\Ext^n_{\cO_\lambda}(\Delta(x'\cdot\lambda), M), 
\end{displaymath}
where the module $M:=\dd\theta_\lambda^{on}T_s L(y)$ fits into a short exact sequence
\begin{displaymath}
0\to L(y\cdot\lambda)\to M\to \theta_\lambda^{on }R\to 0. 
\end{displaymath}
Applying the functor $\Hom_{\cO_\lambda}(\Delta(x'\cdot\lambda),-)$ to this short exact 
sequence yields a long exact sequence containing
\begin{multline*}
\Ext^n_{\cO_\lambda}(\Delta(x'\cdot\lambda),L(y\cdot\lambda))\to \Ext^n_{\cO_\lambda}(\Delta(x'\cdot\lambda), M) \to\\
\to  \Ext^n_{\cO_\lambda}(\Delta(x'\cdot\lambda), \theta_\lambda^{on}R)\to \Ext^{n+1}_{\cO_\lambda}(\Delta(x'\cdot\lambda), L(y\cdot\lambda)) . 
\end{multline*}
First note that
\begin{displaymath}
\Ext^n_{\cO_\lambda}(\Delta(x'\cdot\lambda),L(y\cdot\lambda))= 
\Ext^{n+1}_{\cO_\lambda}(\Delta(x'\cdot\lambda),L(y\cdot\lambda))=0 
\end{displaymath}
by Lemma~\ref{vanish}, so
\begin{displaymath}
\Ext^{n}_{\cO_\lambda}(\Delta(x\cdot\lambda),L(y\cdot\lambda))\cong 
\Ext^n_{\cO_\lambda}(\Delta(x'\cdot\lambda), \theta_\lambda^{on}R). 
\end{displaymath}
In the above, the contributions of $\theta_\lambda^{on}L(z_i)$ must 
always vanish by Lemma~\ref{vanish}, since $z_i> y'$, yielding
\begin{displaymath}
\Ext^{n}_{\cO_\lambda}(\Delta(x\cdot\lambda),L(y\cdot\lambda))\cong 
\Ext^n_{\cO_\lambda}(\Delta(x'\cdot\lambda), \theta_\lambda^{on}L(y'\cdot\lambda)). 
\end{displaymath}
This concludes the proof.
\end{proof}

We note that the last case in Proposition~\ref{propextremalext} does not depend 
on the fact that $n=\len(x)-\len(y)$, it is also possible to give an analogue using 
the results on shuffling functors in \cite[Section~6 and~7]{CM4}.

\begin{lemma}\label{lllnew}
Consider $x,y\in X_\lambda$ and a simple reflection $s\in W$.
\begin{enumerate}[$($i$)$]
\item\label{lllnew.1} If $x=sx'$ with $x'<x$, $x'\in X_\lambda$ and $sy>y$, then
\begin{displaymath}
\Ext^{j+1}_{\cO_\lambda}(\Delta(x\cdot\lambda),L(y\cdot\lambda))\;\cong\; 
\Ext^{j}_{\cO_\lambda}(\Delta(x'\cdot\lambda),L(y\cdot\lambda)),\quad\text{ for all}\quad j\in\mN. 
\end{displaymath}
\item\label{lllnew.2} If $x=x's$ with $x'<x$, $x'\in X_\lambda$ and $ys>y$, where $s$ 
is orthogonal to all simple reflections in $W_\lambda$, then
\begin{displaymath}
\Ext^{j+1}_{\cO_\lambda}(\Delta(x\cdot\lambda),L(y\cdot\lambda))\;\cong\; 
\Ext^{j}_{\cO_\lambda}(\Delta(x'\cdot\lambda),L(y\cdot\lambda)),\quad\text{ for all}\quad j\in\mN. 
\end{displaymath}
\end{enumerate}
\end{lemma}

%%%%%%%%%%%%%%%%%%%%%%%%%%%%%%%%%%%%%%%%%%%%%%%%%%%%%%%%%%%%%%%%%%%%%%%%%%%%%%%%%%%%%%%%%%%%%%%%%%%%%%%%%%%%%%%%%%%%%%%%%%%%%%%

\section{Projective dimensions of structural modules}\label{section7}

\begin{definition}
We define the maps $\llll:X_\lambda\to \mN$ and $\dddd:X_\lambda\to \mN$  as follows:
\begin{displaymath}
\llll(x)=\pd_{\cO_\lambda}\,L(x\cdot\lambda)\qquad\mbox{and}\qquad\dddd(x)=\pd_{\cO_\lambda}\,\Delta(x\cdot\lambda). 
\end{displaymath}
\end{definition}

\subsection{Projective dimensions}

The results in Section \ref{secconnec} and in \cite{SHPO1, SHPO2} allow one to write all 
projective dimensions and graded lengths of the structural modules in some arbitrary block 
$\cO^\mu_\lambda$ in terms of $\llll$ and $\dddd$.

\begin{theorem}[Simple and (co)standard modules] \label{thmsimstan}
For $\lambda,\mu\in\intdom$, we have:
\begin{enumerate}[$($i$)$]
\item\label{thmsimstan.1} 
$\begin{array}{ccll}
\pd_{\cO^\mu_\lambda}\, L(x\cdot\lambda)&=&
\llll(x)-2\len(w_0^\mu),&\quad\text{ for }\quad x\in X_\lambda^\mu;\\ 
\pd_{\cO^\mu_0}\, L(x)&=&2\len(w_0w_0^\mu)-\len(x),&\quad\text{ for }\quad  x\in X^\mu.
\end{array}$
\item\label{thmsimstan.2} 
$\begin{array}{ccll}
\pd_{\cO^\mu_\lambda}\, \Delta^\mu(x\cdot\lambda)&=&
\dddd(w_0^\mu x)-\len(w_0^\mu),&\quad\text{ for }\quad x\in X_\lambda^\mu;\\ 
\pd_{\cO^\mu_0}\, \Delta^\mu(x)&=&\len(x),&\quad\text{ for }\quad x\in X^\mu.
\end{array}$
\item\label{thmsimstan.3} 
$\begin{array}{ccll}
\pd_{\cO^\mu_\lambda}\, \nabla^\mu(x\cdot\lambda)&=&
\mathtt{d}_{\lambda}(w_0xw_0^\lambda)+\aaa(w_0w_0^\lambda)-2\aaa(w_0^\mu),&
\quad\text{ for }\quad x\in X^\mu_\lambda;\\ 
\pd_{\cO^\mu_0}\, \nabla^\mu(x)&=&2\len(w_0 w_0^\mu)-\len(x),&\quad\text{ for }\quad x\in X^\mu.
\end{array}$
\item\label{thmsimstan.4} 
$\begin{array}{ccll}
\gl\,  \Delta^\mu(x\cdot\lambda)=
\gl\,  \nabla^\mu(x\cdot\lambda)&=&
\mathtt{d}_\mu( w_0w_0^\lambda x^{-1})-\len(w_0^\lambda),&\quad\text{ for }\quad x\in X_\lambda^\mu;\\ 
\gl\,  \Delta(x\cdot\lambda)=\gl \, \nabla(x\cdot\lambda)&=&\len(w_0)-\len(x),&
\quad\text{ for }\quad x\in X_\lambda.
\end{array}$
\end{enumerate}
\end{theorem}

\begin{proof}
Claim~\eqref{thmsimstan.1} follows from Theorem~\ref{pdshift} and \cite[Proposition~6]{SHPO1}. 
Claim~\eqref{thmsimstan.2} follows from Corollary \ref{corpdDelta}(i) and \cite[Proposition~3]{SHPO1}. 
Claim~\eqref{thmsimstan.3} follows from claim~\eqref{thmsimstan.2} and Corollary \ref{corglpd}\eqref{corglpd.2}. 
Claim~\eqref{thmsimstan.4} follows from claim~\eqref{thmsimstan.2} and Proposition~\ref{proplinks}\eqref{proplinks.2}.
\end{proof}

\begin{theorem}[Tilting and injective modules] \label{thminjtilt}
For $\lambda,\mu\in\intdom$, we have:
\begin{enumerate}[$($i$)$]
\item\label{thminjtilt.1} 
$\pd_{\cO^\mu_\lambda} T^\mu(x\cdot\lambda)=\aaa(w_0^\mu x w_0^\lambda)-\aaa(w_0^\mu)$,
for $x\in X^\mu_\lambda.$
\item\label{thminjtilt.2} 
$\pd_{\cO^\mu_\lambda} I^\mu(x\cdot\lambda)=2\aaa(w_0 x )-2\aaa(w_0^\mu)$, for  $x\in X_\lambda^\mu.$
\item\label{thminjtilt.3} 
$\begin{array}{ccll}
\gl\, T^\mu(x\cdot\lambda)&=&
2\left(\mathtt{s}_\mu(w_0^\lambda x^{-1}w_0^\mu)-\aaa(w_0^\lambda)-\aaa(w_0 w_0^\mu)\right),&
\quad\text{ for }\quad x\in X^\mu_\lambda;\\ 
\gl\, T(x\cdot\lambda)&=&2\len(w_0)-2\len(x), &\quad\text{ for }\quad x\in X_\lambda.
\end{array}$
\item\label{thminjtilt.4} 
$\begin{array}{ccll}
\gl\, P^\mu(x\cdot\lambda)=
\gl \,I^\mu(x\cdot\lambda)&=&
\mathtt{s}_\mu(w_0x^{-1})-2\len(w_0^\lambda),&\quad\text{ for }\quad x\in X^\mu_\lambda;\\
\gl\, P(x\cdot\lambda)=\gl\, I(x\cdot\lambda)&=&
\len(w_0)+\len(x)-2\len(w_0^\lambda),&\quad\text{ for }\quad x\in X_\lambda.
\end{array}$
\end{enumerate}
\end{theorem}

\begin{proof}
Claims~\eqref{thminjtilt.1} and \eqref{thminjtilt.2} follow from 
Proposition~\ref{proptrans}\eqref{proptrans.3} and \eqref{proptrans.4} 
and Proposition~\ref{propparinjtilt} in combination with \cite[Theorems~17 and~20]{SHPO2}.
Claims~\eqref{thminjtilt.3} and \eqref{thminjtilt.4} follow from 
Proposition~\ref{proplinks}\eqref{proplinks.1} and \eqref{proplinks.3} in combination 
with Theorem~\ref{thmsimstan}\eqref{thmsimstan.1}.
\end{proof}

\begin{remark}\label{remRpd}{\rm
As determined in \cite[Section 9.1]{CM4}, the Ringel dual of $\cO_\lambda^\mu$ is 
$\cO_\lambda^{\widehat\mu}$. The Ringel duality functor 
$\underline\cR_\lambda^\mu: \cO^\mu_\lambda\to \cO^{\widehat\mu}_\lambda$ satisfies
\begin{displaymath}
\underline\cR_\lambda^\mu T^\mu(x\cdot\lambda)\;\cong\; 
I^{\widehat\mu}(w_0 w_0^\mu x w_0^\lambda\cdot\lambda),\qquad\text{ for all }\quad\; x\in X^\mu_\lambda, 
\end{displaymath}
see \cite[Theorem~9.1(ii)]{CM4} and \cite[Proposition~2.2]{prinjective}. 
Hence, Theorem~\ref{thminjtilt}\eqref{thminjtilt.1} and \eqref{thminjtilt.2} imply that, 
for any tilting module $T$ in~$\cO^\mu_\lambda$, we have
\begin{displaymath}
\pd_{\cO^{\widehat\mu}_\lambda}\,\underline\cR_\lambda^\mu T = 2\,\pd_{\cO_\lambda^\mu}\,T. 
\end{displaymath}
So far, we do not have a direct argument why this property should hold.
}
\end{remark}

\subsection{On the functions $\llll$ and $\dddd$}\label{subsecsd}
We fix a $\lambda\in\intdom$.
In the following three statements we determine the extremal values of the functions 
$\llll$ and $\dddd$, for which elements $X_\lambda$ these values are attained and some further estimates. We also prove an inequality connecting the two functions $\llll$ and $\dddd$. 
We will investigate in Section \ref{section9} for which blocks this inequality is, actually, an equality.

\begin{proposition}[simple modules]\label{maxpdL}
For any $x\in X_\lambda$, we have:
\begin{enumerate}[$($i$)$]
\item\label{maxpdL.1} $\aaa(w_0w_0^\lambda)\le \llll(x)\le 2\aaa(w_0w_0^\lambda)$,
\item\label{maxpdL.2} $\llll(x)=2\aaa(w_0w_0^\lambda)$ if and only if $x\in\mathbf{L}(w_0^\lambda)$,
\item\label{maxpdL.3} $\llll(x)=\aaa(w_0w_0^\lambda)$ if and only if $ x=w_0$,
\item\label{maxpdL.4} $\llll(x)\;\le\; \len(w_0 x)+ \aaa(w_0w_0^\lambda)$.
\end{enumerate}
Moreover, in case $\mathbf{R}(x)$ contains an element $w_0^\mu w_0$ for some $\mu\in\intdom$ 
or in case $\fg$ is of type $A$, we have:
\begin{enumerate}[$($i$)$]
\setcounter{enumi}{4}
\item\label{maxpdL.5} $\llll(x)\;\ge\; \aaa(w_0 x)+\aaa(w_0w_0^\lambda).$
\end{enumerate}
\end{proposition}

\begin{proposition}\label{ineq}
For any $x\in X_\lambda$, we have
\begin{displaymath}
\llll(x)\;\ge\; \dddd(w_0 x w_0^\lambda)+\aaa(w_0w_0^\lambda). 
\end{displaymath}
\end{proposition}

\begin{proposition}[standard modules]\label{maxpdD} 
For any $x\in X_\lambda$, we have:
\begin{enumerate}[$($i$)$]
\item\label{maxpdD.1} $0\le \dddd(x)\le \aaa(w_0w_0^\lambda)$,
\item\label{maxpdD.2} $\dddd(x)=\aaa(w_0w_0^\lambda)$ if and only if $ xw_0^\lambda\in\mathbf{L} (w_0w_0^\lambda)$,
\item\label{maxpdD.3} $\dddd(x)=0$ if and only if $ x=w_0^\lambda$,
\item\label{maxpdD.4} $\dddd(x)\le \len(xw_0^\lambda)$.
\end{enumerate}
\end{proposition}

Before proving these three propositions, we need to prove the following lemma.

\begin{lemma}\label{categorification}
Assume that $\fg$ is of type $A$. Then, for any $y,z\in W$, we have:
\begin{displaymath}
\gl\, M(y,z)\ge 2\aaa(y),\qquad\mbox{ for }\quad y\le_R z^{-1}. 
\end{displaymath}
\end{lemma}

\begin{proof}
Our proof of this statement uses techniques and results from the abstract 
$2$-representation theory developed in \cite{MM1,MM2}. We refer the reader 
to these two papers and references therein for more details.

Let $\fancyS$ be the fiat $2$-category of projective functors  on $\mathcal{O}_0$ 
(or, equivalently, Soergel bimodules over the coinvariant algebra) associated 
to $\fg$, as in \cite[Subsection~7.1]{MM1}. 
Then indecomposable $1$-morphisms in $\fancyS$ are exactly
$\theta_w$, where $w\in W$, up to isomorphism. Let $\fancyS^{(y)}$ denote the 
$2$-full fiat $2$-subcategory of $\fancyS$ where indecomposable $1$-morphisms
are all $1$-morphisms of $\fancyS$ which are isomorphic to $\theta_e$
or $\theta_w$, where $w\geq_J y$. Apart from the two-sided cell corresponding
to the identity $1$-morphism $\theta_e$, all other two-sided cells in $\fancyS^{(y)}$
are, by construction, greater than or equal to the two-sided cell containing
$\theta_y$ with respect to the two-sided order.

Let $\mathcal{X}_{y,z}$ be the full subcategory $\add(X)$ of $\mathcal{O}_0$, where
\begin{displaymath}
X\,=\,L(z)\oplus\bigoplus_{w\ge_J y}\theta_w L(z). 
\end{displaymath}
By construction, the action of 
$\fancyS^{(y)}$ on $\mathcal{O}_0$ restricts to $\mathcal{X}_{y,z}$ 
and this gives a finitary $2$-rep\-re\-sen\-ta\-ti\-on of $\fancyS^{(y)}$.

Consider the weak Jordan-H{\"o}lder series of this $2$-rep\-re\-sen\-ta\-ti\-on in the
sense of \cite[Subsection~4.3]{MM2}.
Subquotients of this series are simple transitive
$2$-rep\-re\-sen\-ta\-ti\-ons of $\fancyS^{(y)}$. The $2$-category $\fancyS$, and hence
also the $2$-category $\fancyS^{(y)}$, satisfy all assumptions of \cite[Theorem~18]{MM2},
see \cite[Subsection~7.1]{MM1}. Therefore any simple transitive
$2$-rep\-re\-sen\-ta\-ti\-on of $\fancyS^{(y)}$ is equivalent to a cell $2$-rep\-re\-sen\-ta\-ti\-on in the sense of \cite{MM1}. 

In the following we will use the term Loewy length of an object in a finitary category 
for its Loewy length in the abelianisation of the category. Take $N'$ to be an 
indecomposable direct summand of $M(y,z)$. It's Loewy length is smaller than or equal 
to the graded length of $M(y,z)$. Let $\mathbf{R}$ be the right cell which corresponds
to the cell $2$-representation of $\fancyS^{(y)}$ which has, as an indecomposable direct summand,
the image $N$ of $N'$. Note that the Loewy length of $N$ is not greater than that of~$N'$.

As mentioned, this $2$-representation must be equivalent to the cell $2$-representation 
constructed on a subcategory of $\cO_0$ in \cite[Section~7.1]{MM1}. In particular the 
relevant subquotient category of $\cX_{y,z}$ is equivalent to the category of \cite[Section~7.1]{MM1}. 
This means that the Loewy length of $N$ is equal to the Loewy length of a module of the form 
$\theta_w L(d)$, where $d$ is the Duflo involution in $\RR$. All these modules have simple 
top by \cite[Theorem~6]{SHPO2}, so their Loewy length is given by $2\aaa(\mathbf{R})$ by 
\cite[Proposition~1(c)]{SHPO2}. Putting all inequalities together implies $\gl\, M(y,z)$ 
is at least~$2\aaa(\mathbf{R})$. 

Since $\aaa$ is weakly monotone with respect to KL-orders,  
it remains to observe that the combination of $M(y,z)\neq 0$
(which is equivalent to $y\le_R z^{-1}$)
and the above construction implies $\mathbf{R}\neq \{e\}$, so $\RR\ge_J y $ and $\aaa(\RR)\ge \aaa(y)$. 
The claim of the lemma  follows.
\end{proof}

\begin{proof}[Proof of Proposition~\ref{maxpdL}]
The lower bound of claim~\eqref{maxpdL.1} follows from Proposition~\ref{DeltaLnew}\eqref{DeltaLnew.1}. 
The upper bound follows from the global dimension in Theorem~\ref{thmgd}. 

Claim~\eqref{maxpdL.2} follows from Lemma~\ref{lemIrving}\eqref{lemIrving.5} and \eqref{lemIrving.6}
and Proposition~\ref{proplinks}\eqref{proplinks.1}.

Claim~\eqref{maxpdL.3} is just Lemma~\ref{cor0}\eqref{cor0.4}.

Proposition~\ref{proplinks}\eqref{proplinks.3} implies that claim~\eqref{maxpdL.4} is equivalent to 
the claim 
\begin{displaymath}
\gl\, T^\lambda(w_0^\lambda x^{-1})\le 2\len(w_0 x). 
\end{displaymath}
The latter is known to be true. 
Indeed, by equation~\eqref{tiltingtheta}, it is a special case of the property
\begin{displaymath}
\gl\, \theta_{y}L\;\le \; 2\len(y),\qquad\mbox{ for }\quad y\in W, 
\end{displaymath}
for any simple module $L$. This inequality follows by induction on the length of $y$ 
using \cite[Equation~(1)]{SHPO1} and the fact that the action of $\theta_s$ for 
simple reflection $s$ can only increase the graded length of a module by $2$.

For claim~\eqref{maxpdL.5}, we first assume that there is some $\mu\in\intdom$ 
such that $x\sim_R w_0^\mu w_0$. In particular, $x\in X^\mu$, so 
Proposition~\ref{DeltaLnew}\eqref{DeltaLnew.1} and Theorem \ref{pdshift}\eqref{pdshift.1} imply
\begin{displaymath}
\pd_{\cO_\lambda}L(x\cdot\lambda)\ge \aaa(w_0w_0^\lambda)+\aaa(w_0^\mu). 
\end{displaymath}
By $x\sim_Rw_0^\mu w_0$ and equation \eqref{LRorder}, we find $\aaa(w_0^\mu)=\aaa(xw_0)$. 

In type $A$, Lemma~\ref{categorification} and Equation~\eqref{tiltingtheta} imply that
\begin{displaymath}
\gl\,  T^\mu(y)\ge 2\aaa(w_0^\mu yw_0),\qquad\text{ for all }\quad y\in X^\mu. 
\end{displaymath}
Claim~\eqref{maxpdL.5} for type $A$ hence follows from Proposition~\ref{proplinks}\eqref{proplinks.2}.
\end{proof}

\begin{proof}[Proof of Proposition~\ref{ineq}]
By Proposition~\ref{proplinks}\eqref{proplinks.2} and \eqref{proplinks.3}, the statement is equivalent the condition
\begin{displaymath}
\frac{1}{2}\gl T^\lambda(w_0^\lambda x^{-1})+
\aaa(w_0w_0^\lambda)\,\ge\,\gl\Delta^\lambda(w_0^\lambda x^{-1})+\aaa(w_0w_0^\lambda).
\end{displaymath}
The latter is an immediate consequence of Equation~\eqref{glTDelta}.
\end{proof}

\begin{proof}[Proof of Proposition~\ref{maxpdD}]
The upper bound in claim~\eqref{maxpdD.1}  follows from Proposition~\ref{DeltaLnew}\eqref{DeltaLnew.1}. 

Claim~\eqref{maxpdD.2} is Proposition~\ref{DeltaLnew}\eqref{DeltaLnew.3} and claim~\eqref{maxpdD.3} is just Lemma~\ref{cor0}\eqref{cor0.1}.

Claim~\eqref{maxpdD.4} is a consequence of the combination of inequalities in Propositions~\ref{maxpdL}\eqref{maxpdL.4} and~\ref{ineq}.
\end{proof}

We end this subsection with some consequences of the main results.
Propositions \ref{maxpdL} and \ref{maxpdD} are sufficient 
to determine $\llll$ and $\dddd$ (and hence the projective dimensions 
of all structural modules in all parabolic versions of) all blocks 
$\cO_\lambda$ where the global dimension is not greater than $4$. 
Note that, by Theorem~\ref{thmgd}, this correspond to the cases 
where $\aaa(w_0^\lambda w_0)\le 2$.

\begin{proposition}\label{asmall}
Let $\lambda\in \intdom$ be such that $\aaa(w_0w_0^\lambda)\le 2$.
Then, for all $x\in X_\lambda$, we have
\begin{eqnarray*}
\llll(x)&=& \aaa(w_0 x)+\aaa(w_0 w_0^\lambda),\\
\dddd(x)&=& \aaa(x w_0^\lambda).
\end{eqnarray*}
In particular, the inequalities in both Proposition~\ref{maxpdL}\eqref{maxpdL.4} 
and Proposition~\ref{ineq} are always equalities in such blocks.
\end{proposition}

\begin{proof}
First consider the case $\aaa(w_0w_0^\lambda)=1$. In this case 
Proposition~\ref{maxpdL} implies that
\begin{displaymath}
X_\lambda\;=\; \mathbf{L}(w_0^\lambda)\cup\{w_0\}.
\end{displaymath}
The statement is then just a reformulation of Lemmata \ref{maxpdL} and \ref{maxpdD}.

If $\aaa(w_0w_0^\lambda)=2$, Proposition~\ref{maxpdL} implies that
\begin{displaymath}
X_\lambda \;=\; \mathbf{L}(w_0^\lambda)\cup\mathbf{C}\cup\{w_0\}, 
\end{displaymath}
for some collection $\mathbf{C}$ of left cells  such that $\aaa(x w_0^\lambda)=1$ 
for all $x\in\mathbf{C}$.
The result hence follows again from Lemmata \ref{maxpdL} and \ref{maxpdD}.
\end{proof}

We can also determine the projective dimension of a certain type of 
simple modules by the following proposition.

\begin{proposition}\label{propanti}
Consider a fixed $x\in X_\lambda$. Assume that there is some $\mu\in\intdom$ for which $x\in X^\mu$.
\begin{enumerate}[$($i$)$]
\item\label{propanti.1} If $L(x\cdot\lambda)$ is a standard module in $\cO^\mu_\lambda$, then
%\footnote{If Proposition~\ref{maxpdL}(5) is true in general, this proposition 
%implies that $L(x\cdot\lambda)$ can only be a standard module in 
%$\cO^\mu_\lambda$ if $x\sim_R w_0^\mu w_0$!}
\begin{eqnarray*}\llll(x)&=&\aaa(w_0w_0^\lambda)+\aaa(w_0^\mu),\\
\dddd(w_0^\mu x)&=&\aaa(w_0w_0^\lambda).
\end{eqnarray*}
\item\label{propanti.2} If $L(x\cdot\lambda)$ is not a standard module 
in $\cO^\mu_\lambda$, then
\begin{displaymath}
\llll(x)>\aaa(w_0w_0^\lambda)+\aaa(w_0^\mu). 
\end{displaymath}
\end{enumerate}
\end{proposition}
 
\begin{proof}
Consider the condition for claim~\eqref{propanti.1}. 
Proposition~\ref{DeltaLnew}\eqref{DeltaLnew.2} implies that in this case
\begin{displaymath}
\pd_{\cO_\lambda^\mu}\,L(x\cdot\lambda)=\aaa(w_0w_0^\lambda)-\aaa(w_0^\mu). 
\end{displaymath}
The first result thus follows from Theorem \ref{pdshift}. As $\Delta^\mu(x\cdot\lambda)=L(x\cdot\lambda)$, the second formula follows 
from Corollary~\ref{corpdDelta}\eqref{corpdDelta.1}.

Claim~\eqref{propanti.2} follows from Proposition~\ref{DeltaLnew}\eqref{DeltaLnew.1} and
\eqref{DeltaLnew.2} and Theorem \ref{pdshift}\eqref{pdshift.1}.
\end{proof}

%%%%%%%%%%%%%%%%%%%%%%%%%%%%%%%%%%%%%%%%%%%%%%%%%%%%%%%%%%%%%%%%%%%%%%%%%%%%%%%%%%%%%%%%%%%%%%%%%%%%%%%%%%%%%%%%

\section{Monotonicity  for quasi-hereditary algebras}\label{section9}

\subsection{General principles}\label{secmono}

In this subsection we will consider indices which are empty or equal to $0$ or $1$. 
We denote the corresponding set of indices by $\{\ast,0,1\}$ and set $c_\ast=-1$, $c_0=0$ and $c_1=1$.

We consider the following possible monotonicity properties of projective dimensions 
for modules over a quasi-hereditary algebra $(B,\le)$, where we have $\gamma\in\{\ast,0,1\}$
and $\cC=B$-mod:
\begin{itemize}
\item $\textswab{S}_\gamma(B)$: For all $\alpha,\beta\in\Lambda$ with $\alpha< \beta$, 
$\pd_{\cC}\, L(\alpha)\le \pd_{\cC}\, L(\beta)+c_\gamma$;
\item $\textswab{C}_\gamma(B)$: For all $\alpha,\beta\in\Lambda$ with $\alpha< \beta$, 
$\pd_{\cC}\, \nabla(\alpha)\le \pd_{\cC}\, \nabla(\beta)+c_\gamma$;
\item $\textswab{D}_\gamma(B)$: For all $\alpha,\beta\in\Lambda$ with $\alpha< \beta$, 
$\pd_{\cC}\, \Delta(\alpha)\ge \pd_{\cC}\, \Delta(\beta)-c_\gamma$.
\end{itemize}
Obviously, we have
\begin{equation}\label{trivialeq}
\begin{array}{l}
\textswab{S}(B)\Rightarrow \textswab{S}_0(B)\Rightarrow \textswab{S}_{1}(B),\\
\textswab{S}(B)\Rightarrow \textswab{C}_0(B)\Rightarrow \textswab{C}_{1}(B),\\
\textswab{D}(B)\Rightarrow \textswab{D}_0(B)\Rightarrow \textswab{D}_{1}(B).
\end{array}
\end{equation}

We also define the following two possible properties
\begin{itemize}
\item $\textswab{P}(B)$: For all $\alpha\in \Lambda$, $\pd_{\cC}\, L(\alpha)=\pd_{\cC}\, \nabla(\alpha)$;
\item $\textswab{Q}(B)$: For all $\alpha\in\Lambda$, $\gl\,  T(\alpha)=\gl\,  \Delta(\alpha)+\gl\,  \nabla(\alpha)$.
\end{itemize}
Here, for the second property, we assume that the algebra~$B$ is graded. 

There are some immediate links between these properties, as we summarise in the following two propositions.

\begin{proposition}\label{costmono}
For any quasi-hereditary algebra $B$, we have
\begin{displaymath}
\textswab{S}_0(B)\Rightarrow \textswab{P}(B)\qquad\mbox{and}\qquad\textswab{C}(B)\Rightarrow \textswab{P}(B). 
\end{displaymath}
Consequently, we have
\begin{displaymath}
\textswab{S}_0(B)\Rightarrow\textswab{C}_0(B)\qquad\mbox{and}\qquad \textswab{S}(B)\Leftrightarrow\textswab{C}(B). 
\end{displaymath}
\end{proposition}

\begin{proposition}\label{linkTQ}
Consider a standard Koszul quasi-hereditary algebra $B$ with a simple preserving duality. 
If $\textswab{D}_{1}(B)$ is true and the grading on $R(E(B))$ induced from the Koszul 
grading on on $E(B)$ is positive, then $\textswab{Q}(E(B))$ is true.
\end{proposition}

The monotonicity properties for quasi-hereditary algebras are also closely related to 
the question whether the corresponding module categories are Guichardet.

\begin{lemma}\label{lemGui}
Consider a quasi-hereditary algebra $B$ such that every covering $\alpha \le \beta$ in 
the poset $\Lambda$ implies 
\begin{displaymath}
\Ext^1_{\cC}(L(\alpha),L(\beta))\not=0\quad\mbox{ and }\quad\;\pd\, L(\alpha)-\pd\, L(\beta)\le 1.
\end{displaymath}
\begin{enumerate}[$($i$)$]
\item\label{lemGui.1} If $\textswab{S}_0(B)$ is true, then $B$-mod is weakly Guichardet.
\item\label{lemGui.2} If $\textswab{S}(B)$ is true, then $B$-mod is strongly Guichardet.
\end{enumerate}
\end{lemma}

The remainder of this subsection is devoted to the proofs of these statements.

\begin{proof}[Proof of Proposition~\ref{costmono}]
We consider the short exact sequence 
\begin{equation}\label{defmodQ}
0\to L(\alpha)\to \nabla(\alpha)\to Q\to0,
\end{equation}
which defines the module $Q$. 

Assume that $\textswab{S}_0(B)$ is true. Then $\pd_{\cC}\, Q\le \pd_{\cC}\, L(\alpha)$ 
and $\pd_{\cC}\, \nabla(\alpha)\le \pd_{\cC}\, L(\alpha)$. For any object $M\in\cC$, 
the contravariant left exact functor $\Hom_{\cC}(-,M)$ applied to \eqref{defmodQ}
yields a long exact sequence. For $p=\pd_{\cC}\, L(\alpha)$, a  part of this long
exact sequence is given by
\begin{equation}\label{eqcostmono}
\Ext^{p}_{\cC}(Q,M)\to\Ext^{p}_{\cC}(\nabla(\alpha),M)\to\Ext^{p}_{\cC}(L(\alpha),M)\to 0.
\end{equation}
By the definition of $p$, the last extension group is not always trivial, implying 
$\pd_{\cC}\,  \nabla(\alpha)\ge \pd_{\cC}\, L(\alpha)$ and hence $\pd_{\cC}\,  \nabla(\alpha)= \pd_{\cC}\,  L(\alpha)$. 

Now assume that $\textswab{C}(B)$ is true. We prove that 
$\pd_{\cC}\, L(\lambda)=\pd_{\cC}\, \nabla(\lambda)$ by induction along the partial 
order on $\Lambda$. Consider a minimal element $\alpha\in\Lambda$, then 
$\nabla(\alpha)\cong L(\alpha)$. Then consider an $\alpha\in\Lambda$ such 
that $\pd_{\cC}\, L(\lambda)=\pd_{\cC}\, \nabla(\lambda)$ for all $\lambda<\alpha$. 
In particular, $\pd_{\cC}\, L(\lambda)< \pd_{\cC}\, \nabla(\alpha)$ for all 
$\lambda<\alpha$ which yields $\pd_{\cC}\,  Q< \pd_{\cC}\,  \nabla(\alpha)$.
By the same reasoning as in the above paragraph, we hence obtain the exact 
sequence~\eqref{eqcostmono} with $p:=\pd_{\cC}\, \nabla(\alpha)$, where now also 
the first term vanishes. This implies $\pd_{\cC}\, L(\alpha)\ge \pd_{\cC}\, \nabla(\alpha)$. 
The inequality $\pd_{\cC}\, L(\alpha)\le \pd_{\cC}\, \nabla(\alpha)$ follows from 
Equation~\eqref{eqcostmono} for $p>\pd_{\cC}\,  \nabla(\alpha)$, where the first term still vanishes.

The statements $\textswab{S}_0(B)\Rightarrow\textswab{C}_0(B)$ and 
$\textswab{S}(B)\Leftrightarrow\textswab{C}(B)$ follow from the above properties 
and the observation that, when $\textswab{P}(B)$ is true, $\textswab{S}_\gamma(B)$ is 
equivalent to $\textswab{C}_\gamma(B)$, for any $\gamma\in\{\ast,0,1\}$.
\end{proof}

\begin{corollary}
Assume that $\textswab{S}(B)$ or $\textswab{C}(B)$ is true. Then we have
\begin{displaymath}
\Ext^p_{\cC}(L(\alpha),M)\cong \Ext_{\cC}^p(\nabla(\alpha),M), 
\end{displaymath}
where $p=\pd_{\cC}\, L(\alpha)=\pd_{\cC}\, \nabla(\alpha)$, $M\in\cC$ and $\alpha\in\Lambda$.
\end{corollary}

\begin{proof}
Under our assumptions, the first term in Equation~\eqref{eqcostmono} is zero 
implying the isomorphism of extension groups.
\end{proof}

\begin{lemma}\label{monQ}
Consider a positively graded quasi-hereditary algebra $B$ which satisfies
\begin{displaymath}
\gl\,  \nabla(\alpha)\le \gl\, \nabla(\beta)+1\mbox{ and }\;
\gl\, \Delta(\alpha)\le \gl\, \Delta(\beta)+1,
\qquad\mbox{for all}\quad \alpha\le \beta.
\end{displaymath}
If also the induced grading on $R(B)$ is also positive, then $\textswab{Q}(B)$ is true.
\end{lemma}

\begin{proof}
By definition it follows that positivity of the grading on $R(B)$ is equivalent 
to the fact that any subquotient of a standard filtration of any tilting module 
$T^B(\alpha)\langle 0\rangle$ in the graded lift of $\cC_B$ is of the form 
$\Delta^B(\beta)\langle - j\rangle$ where $j=0$ if $\beta=\alpha$ and $j>0$ 
otherwise, see e.g. \cite[Section~2.3]{MaTilt}. Similarly, any subquotient of 
a costandard filtration of $T^B(\alpha)\langle 0\rangle$ is either of the form 
$\nabla^B(\alpha)\langle 0\rangle$ or $\nabla^B(\beta)\langle j\rangle$ with $j>0$ and $\beta <\alpha$. 

Using the standard filtration then implies that, for $j>0$, we have 
$T^B(\alpha)_j\not=0$ if $\Delta^B(\alpha)_j\not=0$ and, by the assumptions, 
$T^B(\alpha)_j=0$ if $\Delta^B(\alpha)_j=0$. The costandard filtration then 
similarly yields the maximal $j>0$ for which $T^B(\alpha)_{-j}$ is non-zero, concluding the proof.
\end{proof}

\begin{proof}[Proof of Proposition~\ref{linkTQ}]
Consider $D=E(B)$. By the standard Koszulity of $B$ and $\textswab{D}_{1}(B)$, 
it follows that $\gl\,  \Delta^D(\alpha)\le \gl\,  \Delta^D(\beta)+1$, if $\alpha\le\beta$,
for $\alpha,\beta\in\Lambda_D$. Hence the result follows from applying Lemma~\ref{monQ} to $D$. 
\end{proof}

\begin{proof}[Proof of Lemma~\ref{lemGui}]
If $\textswab{S}(B)$ is true, it follows that every initial segment is generated by 
the simple modules corresponding to an ideal in the poset $\Lambda$. If 
$\textswab{S}_0(B)$ is true, it still follows that every saturated initial segment 
is generated by the simple modules corresponding to an ideal in the poset $\Lambda$. 
The result therefore follows from \cite[Theorem~3.9(i)]{CPS}.
\end{proof}

\subsection{General results for $A_\lambda^\mu$}

We set $\textswab{P}(\mu,\lambda):=\textswab{P}(A^\mu_\lambda)$ etc. 
For the quasi-hereditary algebras corresponding to parabolic category~$\cO$ we can improve substantially on the relations between the different monotonicity properties in Propositions~\ref{costmono} and~\ref{linkTQ}. This leads to the following theorem. 
\begin{theorem}\label{monA}
Consider fixed $\lambda,\mu\in\intdom$. We have the following links between the monotonicity properties
\begin{displaymath}
\xymatrix{
\textswab{S}(\mu,\lambda)\ar@{=>}[r]\ar@{=>}[d]& \textswab{S}_0(\mu,\lambda)
\ar@{=>}[rr]\ar@{=>}[d]&&\textswab{P}(\mu,\lambda)\ar@{=>}[d]\\
\textswab{D}(\widehat\mu,\lambda)\ar@{=>}[r]\ar@{=>}[u]\ar@{=>}[d]& 
\textswab{D}_0(\widehat\mu,\lambda)\ar@{=>}[r]\ar@{=>}[u]\ar@{=>}[d]&
\textswab{D}_{1}(\widehat\mu,\lambda)\ar@{=>}[r]\ar@{=>}[d]&\textswab{Q}(\lambda,\mu)\ar@{=>}[u]\\
\textswab{C}(\mu,\lambda)\ar@{=>}[r]\ar@{=>}[u]& \textswab{C}_0(\mu,\lambda)
\ar@{=>}[u]\ar@{=>}[r]&\textswab{C}_{1}(\mu,\lambda)\ar@{=>}[u] &   
}
\end{displaymath}
Furthermore, we have 
\begin{displaymath}
\textswab{S}_\gamma(0,\lambda)\Rightarrow\textswab{S}_\gamma(\mu,\lambda),\quad \textswab{C}_\gamma(0,\lambda)\Rightarrow\textswab{C}_\gamma(\mu,\lambda)\quad\text{ and }\quad\textswab{D}_\gamma(0,\lambda)\Rightarrow\textswab{D}_\gamma(\mu,\lambda),
\end{displaymath}
for all $\gamma\in \{\ast,0,1\}$, as well as $\textswab{P}(0,\lambda)\Rightarrow \textswab{P}(\mu,\lambda)$ 
and $\textswab{Q}(\mu,0)\Rightarrow \textswab{Q}(\mu,\lambda)$.
\end{theorem}

Note that $\textswab{P}(0,0)$, $\textswab{Q}(0,0)$, $\textswab{S}(0,0)$, $\textswab{C}(0,0)$ 
and $\textswab{D}(0,0)$ are all true by~\cite{SHPO1}. We will prove in Theorem~\ref{sln1} that 
analogous properties do not hold for arbitrary blocks. Before proving Theorem~\ref{monA}, we 
introduce the following definition, motivated by the result.

\begin{definition}\label{defmonot}
For $\lambda\in\intdom$, we say that the block $\cO_\lambda$ is
\begin{itemize}
\item {\em strictly monotone} if $\textswab{D}(0,\lambda)$ is true,
\item {\em weakly monotone} if $\textswab{D}_0(0,\lambda)$ is true,
\item {\em almost monotone} if $\textswab{D}_1(0,\lambda)$ is true.
\end{itemize}
\end{definition}

\begin{corollary}\label{almostmonotoneblock}
Let $\lambda\in\intdom$.
\begin{enumerate}[$($i$)$]
\item\label{almostmonotoneblock.1} 
If $\cO_\lambda$ is almost monotone, then 
\begin{displaymath}
\llll(x)=\dddd(w_0 x w_0^\lambda)+\aaa(w_0w_0^\lambda),
\qquad\quad\text{ for all } \quad x\in X_\lambda. 
\end{displaymath}
\item\label{almostmonotoneblock.2} 
If $\cO_\lambda$ is weakly monotone, it is weakly Guichardet. 
If $\cO_\lambda$ is strictly monotone, it is strongly Guichardet.
\end{enumerate}
\end{corollary}

Now we prove Theorem~\ref{monA} and Corollary~\ref{almostmonotoneblock}.

\begin{proof}[Proof of Theorem~\ref{monA}]
The implication $\textswab{S}_\gamma(0,\lambda)\Rightarrow\textswab{S}_\gamma(\mu,\lambda)$ 
follows from Theorem~\ref{pdshift}\eqref{pdshift.1}. The implication 
$\textswab{D}_\gamma(0,\lambda)\Rightarrow\textswab{D}_\gamma(\mu,\lambda)$ follows from 
Corollary~\ref{corpdDelta}\eqref{corpdDelta.1}. The combination of Corollaries \ref{corpdDelta}\eqref{corpdDelta.1} and \ref{corglpd}\eqref{corglpd.2} implies that
\begin{equation}\label{pdshiftco}
\pd_{\cO^\mu_\lambda}\nabla^\mu(x\cdot\lambda)\;=\;
\pd_{\cO_\lambda}\nabla(x\cdot\lambda)-2\len(w_0^\mu),\quad\text{ for all }\quad x\in X_\lambda^\mu,
\end{equation}
which yields  the implication $\textswab{C}_\gamma(0,\lambda)\Rightarrow\textswab{C}_\gamma(\mu,\lambda)$. 
The combination of Equation~\eqref{pdshiftco} and Theorem~\ref{pdshift}\eqref{pdshift.1} gives the implication  
$\textswab{P}(0,\lambda)\Rightarrow\textswab{P}(\mu,\lambda)$. Further, the implication 
$\textswab{Q}(0,\lambda)\Rightarrow\textswab{Q}(\mu,\lambda)$ follows from 
Proposition~\ref{proptrans}\eqref{proptrans.1} and \eqref{proptrans.2}. 

Now we prove implications in the diagram. The implications of the form 
$\textswab{S}(\mu,\lambda)\Rightarrow\textswab{S}_0(\mu,\lambda)$ are trivial, see Equation~\eqref{trivialeq}. 
The implications $\textswab{D}_{1}(\widehat\mu,\lambda)\Rightarrow\textswab{Q}(\lambda,\mu)$ follows from Proposition~\ref{linkTQ}.

Assume that $\textswab{D}_0(\widehat\mu,\lambda)$ is true, then by the above $\textswab{Q}(\lambda,\mu)$ is also true. 
It follows, moreover, that the graded length of tilting modules in~$\cO_\mu^\lambda$ is weakly 
monotone along the Bruhat order, as this property is inherited from the corresponding property of 
standard modules. Proposition~\ref{proplinks}\eqref{proplinks.3} then shows that 
$\textswab{S}_0(\mu,\lambda)$ follows, proving the implication 
$\textswab{D}_0(\widehat\mu,\lambda)\Rightarrow\textswab{S}_0(\mu,\lambda)$. The
implication $\textswab{D}(\widehat\mu,\lambda)\Rightarrow\textswab{S}(\mu,\lambda)$ follows 
similarly.

The implication $\textswab{P}(\mu,\lambda)\Leftrightarrow \textswab{Q}(\lambda,\mu)$ follows from Proposition~\ref{proplinks}\eqref{proplinks.3} and \eqref{proplinks.4}.
The implication $\textswab{C}_\gamma(\mu,\lambda)\Leftrightarrow \textswab{D}_\gamma(\widehat\mu,\lambda)$ 
follows from Corollary~\ref{corglpd}\eqref{corglpd.2}.
The implication $\textswab{S}_\gamma(\mu,\lambda)\Rightarrow \textswab{D}_\gamma(\mu,\lambda)$ 
follows from the combination of the above implications and Proposition~\ref{costmono}.

Finally, the implication $\textswab{S}_0(\mu,\lambda)\Rightarrow\textswab{P}(\mu,\lambda)$ 
follows from the combination of the other implications.
\end{proof}

\begin{proof}[Proof of Corollary \ref{almostmonotoneblock}]
Claim~\eqref{almostmonotoneblock.1} is the combination of Corollary~\ref{corglpd}\eqref{corglpd.2} 
and the statement $\textswab{D}_1(0,\lambda)\Rightarrow \textswab{P}(0,\lambda)$ in Theorem~\ref{monA}.

Claim~\eqref{almostmonotoneblock.2} follows from Lemma~\ref{lemGui}, the statement 
$\textswab{D}_\gamma(0,\lambda)\Rightarrow \textswab{S}_\gamma(0,\lambda)$ for 
$\gamma\in\{\ast,0\}$ in Theorem~\ref{monA} and the Kazhdan-Lusztig conjecture.
\end{proof}

\subsection{Not all blocks in category~$\cO$ are almost monotone}
\label{extrmon}

In this subsection we consider an example of a singular block for type $A$ which 
shows that the non-monotonicity in projective dimensions of standard modules can be arbitrarily high. 
Consequently we show that equation \eqref{eqintro} is not valid in this block, which is equivalent to
saying that $\textswab{P}(0,\lambda)$ does not hold.

According to equation~\eqref{onlyKLVcelleq}, the projective dimension of $\Delta(x\cdot\lambda)$  
is determined by its extensions with simple modules $L(y\cdot\lambda)$ with 
$y\in\mathbf{L}(w_0^\lambda)$. The maximal degree in which such an extension can 
appear is bounded by $\len(y)-\len(x)$, see e.g. Lemma~\ref{vanish}. The variation in 
length between the elements in~$\mathbf{L}(w_0^\lambda)$ therefore gives an natural 
rough indication of the level in which monotonicity in the projective dimension of 
standard modules might be broken. Indeed, for the examples in Section~\ref{blocks4}
we find that, when the maximal difference in length between elements in~$\mathbf{L}(w_0^\lambda)$ 
is $1$, the block is weakly monotone and when this difference is $2$, it is almost monotone. 
In the block we will consider in this subsection we will take $w_0^\lambda$ such that 
this maximal variation in length becomes arbitrarily high.

\begin{proposition}\label{verynonmon}
Consider $\fg=\mathfrak{sl}(n+1)$ and $\lambda\in\intdom$ with $w_0^\lambda=s_n$. The block $\cO_\lambda$ is
\begin{displaymath}
\begin{cases}
\mbox{weakly monotone, but not strictly monotone, }&\mbox{ if }\, n=2\\
\mbox{almost monotone, but not weakly monotone,}&\mbox{ if } \,n=3\\
\mbox{not almost monotone,}&\mbox{ if }\, n\ge 4.
\end{cases} 
\end{displaymath}
\end{proposition}

\begin{theorem}\label{sln1}
Integral category~$\cO$ for $\mathfrak{sl}(n+1)$ contains blocks $\cO_\lambda$ 
such that $\textswab{P}(0,\lambda)$ is not true, if and only if $n>3$.
\end{theorem}

\begin{proof}[Proof of Proposition~\ref{verynonmon}]
The case $n=2$ is dealt with in \cite[Section 6.2]{SHPO3}. 
The case $n=3$ will be considered in Subsection \ref{weirdsl4}. 
So, we consider the case $n\ge 4$.

We take $x,y\in X_\lambda$ defined as 
\begin{equation}\label{defXY}
x=s_2s_3\cdots s_ns_1s_2\cdots s_n\;\mbox{ and }\;\;y=s_2s_3\cdots s_{n-1}s_1s_2\cdots s_n.
\end{equation} 
Then we  have $x\cdot\lambda\le y\cdot\lambda$ and  $\len(x)-\len(y)=1$. However, we claim that
\begin{equation}\label{estimpd}
\begin{aligned}
&\pd_{\cO_\lambda}\Delta(x\cdot\lambda)\le n-1,\\
&\pd_{\cO_\lambda}\Delta(y\cdot\lambda)\ge 2n-3,
\end{aligned}
\end{equation}
which implies the proposition as, for $n\ge 4$, we have $(2n-3)-(n-1)>1$.

First we prove the second of the inequalities in \eqref{estimpd}. As the module 
$L(s_n\cdot\lambda)$ is $s_i$-finite for all $1\le i < n$, we can use the procedure 
in the proof of \cite[Proposition~3]{SHPO1} (see also \cite[Lemma~3.6(ii)]{CS}) 
iteratively. It follows immediately that 
\begin{displaymath}
\Ext^{2n-3}_{\cO_\lambda}(\Delta(y\cdot\lambda),L(s_n\cdot\lambda))\cong  \Hom_{\cO_\lambda}(\Delta(s_n\cdot\lambda),L(s_n\cdot\lambda)).
\end{displaymath}

To prove the other estimate in \eqref{estimpd} we employ equation~\eqref{onlyKLVcelleq}, 
which implies that we only need to consider extensions with $L(z_i\cdot\lambda)$ 
where 
\begin{displaymath}
z_i=s_i s_{i+1}\cdots s_n,\qquad\mbox{for}\qquad 1\le i\le n. 
\end{displaymath}
As $s_1x\cdot\lambda= x\cdot\lambda$  while $s_1 z_i\cdot\lambda < z_i\cdot\lambda$ 
unless $i=1$, application of \cite[Lemma~3.6(i)]{CS} gives
\begin{displaymath}
\Ext_{\cO_\lambda}^\bullet(\Delta(x\cdot\lambda),L(z_i\cdot\lambda))=0 
\end{displaymath}
unless $i=1$. 

An upper bound on the projective dimension of $\Delta(x\cdot\lambda)$ is hence 
given by Lemma~\ref{vanish}, as the difference between $\len(x)=2n-1$ and $\len(z_1)=n$.
\end{proof}

\begin{proof}[Proof of Theorem~\ref{sln1}]
That all the properties are always satisfied provided $n\le 3$ follows from the fact that 
all blocks are almost monotone, see Section~\ref{blocks4}, \cite[Section~6.2]{SHPO3} and 
Theorem~\ref{monA}.

To deal with the case $n>3$, we consider the block introduced in Proposition~\ref{verynonmon}. 
We will prove that we have
\begin{equation}\label{eqstrictin}
\pd_{\cO_\lambda}\, L(w\cdot\lambda)-\aaa(w_0w_0^\lambda)\,>\,
\pd_{\cO_\lambda}\, \Delta(w_0 w w_0^\lambda\cdot\lambda),
\end{equation}
for $w=w_0 xw_0^\lambda$ and $x$ in Equation~\eqref{defXY}. This shows that 
$\textswab{P}(0,\lambda)$ is not true because of Corollary~\ref{corglpd}\eqref{corglpd.2}.

Lemma~\ref{lempdLalt} implies that $\pd_{\cO_\lambda} L(w\cdot\lambda)-\aaa(w_0w_0^\lambda)$ 
is greater than or equal to
\begin{displaymath}
\pd_{\cO_\lambda}\,  \Delta(w_0 z w_0^\lambda\cdot\lambda)-
\min\{j\in\mN\,|\, \Ext^j_{\cO_\lambda}(\Delta(z\cdot\lambda),L(w\cdot\lambda))\not=0\} ,
\end{displaymath}
for an arbitrary $z\in X_\lambda$.

Therefore we introduce $z=w_0 y w_0^\lambda$ with $y$ as in Equation~\eqref{defXY}. 
As $z\not=w$, in order to prove~\eqref{eqstrictin} it thus suffices to prove that
\begin{itemize}
\item $\pd_{\cO_\lambda}\,\Delta(w_0 z w_0^\lambda\cdot\lambda)>\pd_{\cO_\lambda} \,
\Delta(w_0 w w_0^\lambda\cdot\lambda)+1$ and
\item $\Ext^\bullet_{\cO_\lambda}(\Delta(z\cdot\lambda),L(w\cdot\lambda) )\not=0$.
\end{itemize}
The first property follows immediately from Equation~\eqref{estimpd}. 
We have  $w\le z$ and  $\len(w)=\len(z)-1$. Therefore it is possible to derive
\begin{displaymath}
\Ext^1_{\cO_\lambda}(\Delta(z\cdot\lambda),L(w\cdot\lambda) )\not=0 
\end{displaymath}
by applying Lemma~\ref{propextremalext}. This concludes the proof.
\end{proof}

%%%%%%%%%%%%%%%%%%%%%%%%%%%%%%%%%%%%%%%%%%%%%%%%%%%%%%%%

\section{Hermitian symmetric pairs}\label{secHS}

In this section we calculate $\llll$ and $\dddd$ when $\lambda\in\intdom$ 
is such that it `corresponds to a hermitian symmetric pair'. By this we mean that for 
the reductive Lie algebra $\fl$, generated by the Cartan subalgebra of $\fg$ and all 
root vectors corresponding to $B_\lambda$ and $-B_\lambda$, the pair $(\fg,\fl)$ is a 
{\em hermitian symmetric pair}. In particular, this implies that $W_\lambda$ is a maximal 
Coxeter subgroup of $W$.

\begin{theorem}\label{ThmHS}
Consider a reductive Lie algebra $\fg$ and $\lambda\in\intdom$ which corresponds 
to a hermitian symmetric pair. Then, for all $x\in X_\lambda$, we have
\begin{enumerate}[$($i$)$]
\item\label{ThmHS.1} $\llll(x)=\aaa(w_0 x)+\aaa(w_0w_0^\lambda)$,
\item\label{ThmHS.2} $\dddd(x)=\aaa(x w_0^\lambda)$.
\end{enumerate} 
Furthermore, the block $\cO_\lambda$ is weakly monotone.
\end{theorem}

We start the proof of this theorem, by linking the main results of \cite{CIS} to Lusztig's $\aaa$-function. \begin{proposition}\label{propCIS}
Consider $\mu\in\intdom$, such that the Levi subalgebra $\fl$ of the parabolic subalgebra $\fq_\mu$ forms a hermitian symmetric pair $(\fg,\fl)$. 
\begin{enumerate}[$($i$)$]
\item\label{propCIS.1} For any $x\in X^\mu$, the set $\Sigma_x$ from \cite[Definition~2.2]{CIS} 
satisfies
\begin{displaymath}
{\rm card}\, \Sigma_x\;=\;\aaa(x). 
\end{displaymath}
\item\label{propCIS.2} The auxiliary integral constant $p$ attached to any hermitian 
symmetric pair in \cite[Table~2.1]{CIS} satisfies
\begin{displaymath}
p\;=\;\aaa(w_0^\mu w_0)\;=\;{\rm card}\,  \{\mbox{distinct right cells in $X^\mu$}\}-1. 
\end{displaymath}
\item\label{propCIS.3} For any two $x,y\in X^\mu$, we have
\begin{displaymath}
x\le_Ry\quad\mbox{or}\quad y\le_Rx.
\end{displaymath}
\end{enumerate}
\end{proposition}

We will use freely that the Bruhat order on $X^\mu$ is generated by right multiplication 
with simple reflections, see e.g.~\cite[Corollary~3.12]{EHP}.

\begin{proof}
First we prove claim~\eqref{propCIS.3}. Assume we have two $x,y\in X^\mu$ which are not right comparable. 
By \cite[Theorem~1.4]{CIS}, $\Sigma_x$ and $\Sigma_y$ cannot have the same cardinality, 
so without loss of generality we assume that $\card\, \Sigma_x>\card\,  \Sigma_y$. 
There must be some simple reflection $s$ such that $x'=xs < x$ and $xs\in X^\mu$. 
By \cite[Lemma~5.9]{CIS}, $x'$ satisfies $\card\, \Sigma_{x'}=\card\, \Sigma_x-1$. 
We can repeat this construction until we obtain some $x''$ which, by construction, 
satisfies $x''\le_R x$, and, at the same time, satisfies $\card\, \Sigma_{x''}=\card\, \Sigma_y$. 
Applying \cite[Theorem~1.4]{CIS} once more, yields $x''\sim_R y$ and thus $x\le_R y$, a contradiction.

Now we prove claim~\eqref{propCIS.1}. By claim~\eqref{propCIS.3}, there is some number 
$q$ such that we can decompose $X^\mu$ into right cells as
\begin{displaymath}
X^\mu\;=\; \RR_0\cup\RR_1\cup\cdots\cup \RR_q, 
\end{displaymath}
where we have $\RR_i\le_R\RR_j$ if and only if $i\le j$. By construction, we must 
have $\RR_0=\RR(e)=\{e\}$ and $\RR_q=\RR(w_0^\mu w_0)$.

We define two sequences of numbers,
\begin{displaymath}
\sigma(i):=\card\,\Sigma_x\qquad\mbox{ and }\qquad\aaa(i):=\aaa(x),\qquad \mbox{for an arbitrary }x\in \RR_i. 
\end{displaymath}
By \cite[Definition~2.2]{CIS}, we have $\card\,\Sigma_e=0=\aaa(e)$. Then, by claim~\eqref{propCIS.3} 
and \cite[Lemma~5.9]{CIS}, we have $\sigma(i)=i$. The combination of 
Lemma~\ref{Set}\eqref{Set.2} and the remark below \cite[Theorem~1.4]{CIS}, 
then implies that we have 
\begin{displaymath}
\sigma(q)=q=\aaa(w_0^\mu w_0)=\aaa(q). 
\end{displaymath}
Now, as the sequence of number $\aaa(i)$ must be strictly monotone, we also find
\begin{displaymath}
\aaa(\RR_i)=\aaa(i)=i=\sigma(i), 
\end{displaymath}
proving claim~\eqref{propCIS.1}.

By the above, to prove claim~\eqref{propCIS.2}, it suffices to show that $p$ corresponds 
to the maximal graded length (in our convention) of a standard module in $\cO^\mu_0$. 
For HS.6 and HS.7 in \cite[Table~2.1]{CIS}, this follows immediately from comparing 
to \cite[Tables~7.1 and~7.2]{CIS}. For HS.2 and HS.4, this follows by immediate computation, 
see e.g. the displayed equation on \cite[page~73]{CIS}. For HS.1, HS3 and HS.5, this
follows from the proof by induction on $p$ in \cite[Section~5]{CIS}.
\end{proof}

We also have the following lemma.

\begin{lemma}\label{lemmonHS}
Consider $\mu\in\intdom$ as in Proposition~\ref{propCIS}. Then, for any $x,y\in X^\mu$, 
the condition $x\le y$ implies
\begin{displaymath}
\gl\, \Delta^\mu(x)\;\ge\;\gl\, \Delta^\mu(y). 
\end{displaymath}
\end{lemma}

\begin{proof}
This follows immediately from the fact that the Bruhat order is generated by simple 
reflections and Lemma~\ref{lemsgl}.
\end{proof}

\begin{proof}[Proof of Theorem~\ref{ThmHS}]
First we prove claim~\eqref{ThmHS.2}. By Lemma~\ref{lemKPhi}, 
the statement is equivalent to the claim that 
\begin{displaymath}
\gl\,  \Delta^\mu(x)\;=\;\aaa(w_0^\mu x w_0),
\end{displaymath}
for $\mu$ as in Proposition~\ref{propCIS} and any $x\in X^\mu$. As by~\cite{CIS}, we have the equality
$\gl\,  \Delta^\mu(x)=\card\, \Sigma_{w_0^\mu x w_0}$, the claim follows from 
Proposition~\ref{propCIS}\eqref{propCIS.1}.

Now Lemma~\ref{lemmonHS} implies that $\cO_\lambda$ is weakly monotone in the sense of 
Definition~\ref{defmonot}.  Claim~\eqref{ThmHS.1} therefore follows from claim~\eqref{ThmHS.2} 
and Corollary~\ref{almostmonotoneblock}\eqref{almostmonotoneblock.1}.
\end{proof}

%%%%%%%%%%%%%%%%%%%%%%%%%%%%%%%%%%%%%%%%%%%%%%%%%%%%%%%%%%%

\section{A family of non-maximal singularities}\label{secFam}\label{examsec2}

In this subsection we completely determine projective dimensions of simple modules 
in the block $\cO_\lambda$, for $\g=\mathfrak{sl}(n)$ with $\lambda\in\intdom$ 
satisfying 
\begin{displaymath}
W_\lambda=S_1\times S_1\times S_{n-2}\;\subset\; S_n. 
\end{displaymath}
We assume $n>3$, as otherwise this is a regular block. We define 
$x_1,x'_1\in X_\lambda$ as $x_1=s_{n-1}s_{n-2}\cdots s_2 w_0^\lambda$ and 
$x'_1=s_{n-1}s_{n-2}\cdots s_1 w_0^\lambda$. Making use of the Robinson-Schensted correspondence allows to conclude that $X_\lambda$ is the union of the following left cells:
\begin{displaymath}
\LL_0:=\{w_0\},\; \LL_1:=\LL(x_1),\; \LL'_1:=\LL(x'_1),\; 
\LL_2:=\LL(s_1 w_0^\lambda)\mbox{ and }\;\LL_3:=\LL(w_0^\lambda),
\end{displaymath}
with values
\begin{displaymath}
\aaa(w_0\LL_0)=0,\;\aaa(w_0\LL_1)=\aaa(w_0\LL'_1)=1,\;\aaa(w_0\LL_2)=2\mbox{ and }\, \aaa(w_0\LL_3)=3. 
\end{displaymath}
We will write out these cells explicitly, for $n=4$, in Subsection~\ref{APsl4.1}.

The cells $\LL_1$ and $\LL'_1$ belong to the same two-sided cell and contain $n-1$ elements each. 
The cell $\LL_1$ consists of the elements $x_j$ defined as
\begin{displaymath}
x_j=s_{j-1}s_{j-2}\cdots s_1 x_1,\qquad \mbox{for}\quad 1\le j\le n-1. 
\end{displaymath}
The cell $\LL_1'$ consists of the elements $x'_j$ defined as
\begin{displaymath}
x'_j=
\begin{cases}
s_{j-1}s_{j-2}\cdots s_1 x'_1,&\mbox{for }1\le j\le n-2;\\
s_{n-3}s_{n-2}\cdots s_1(s_n x'_1),&\mbox{for }j=n-1;
\end{cases}
\end{displaymath}
where we note that $s_nx'_1 < x'_1$.
In particular, we have
\begin{displaymath}
\len(x_j)-\len(w_0^\lambda)=j+n-3,\qquad \mbox{for}\quad 1\le j\le n-1,
\end{displaymath}
and
\begin{displaymath}
\len(x'_j)-\len(w_0^\lambda)=
\begin{cases}
j+n-2,&\mbox{for }1\le j\le n-2;\\
2n-5,&\mbox{for }j=n-2.
\end{cases}
\end{displaymath}
Now we can state the result.

\begin{proposition}\label{Ex11n}
Consider $\g=\mathfrak{sl}(n)$ with $\lambda\in\intdom$ such that 
\begin{displaymath}
W_\lambda=S_1\times S_1\times S_{n-2}\;\subset\; S_n
\end{displaymath}
and $n>3$. We have
\begin{displaymath}
\llll(x)\;=\;\aaa(w_0 x)+\aaa(w_0w_0^\lambda)\qquad \mbox{if }\; x\not\in \LL_1\cup \LL_1', 
\end{displaymath}
so $\llll(\mathbf{L}_i)=3+i$, for $i\in \{0,2,3\}$.

For $x\in \; \LL_1\cup \LL_1'$, we have
\begin{displaymath}
\llll(x)\;=\;\begin{cases}
4=\aaa(w_0 x)+\aaa(w_0w_0^\lambda),&\mbox{for }x\in\; \LL_1'\backslash\{x'_{n-1}\}\;\cup\;\{x_{n-1}\};\\
5=\aaa(w_0 x)+\aaa(w_0w_0^\lambda)+1,&\mbox{for }x\in\; \LL_1\backslash\{x_{n-1}\}\;\cup\;\{x'_{n-1}\}.
\end{cases} 
\end{displaymath}
\end{proposition}

\begin{proof}
For $x\in \LL_0$ and $x\in \LL_3$, this follows from Proposition~\ref{maxpdL}\eqref{maxpdL.2} and 
\eqref{maxpdL.3}. The projective dimensions are $3$, respectively $6$. For $x\in\LL_2$, this follows 
from Proposition~\ref{maxpdL}\eqref{maxpdL.5} and \eqref{maxpdL.2}, the projective dimension is $5$. 
For $x\in \LL_1\cup \LL_1'$, Proposition~\ref{maxpdL} allows to conclude that the projective dimensions 
are either~$4$ or~$5$.

For $1\le j\le n-1$, we choose $\mu_j\in\intdom$ such that $w_0^{\mu_j}=s_{j}$. We have 
\begin{displaymath}
X^\mu\cap \LL_1\;=\; \{x_j\}\quad\mbox{and}\quad X^\mu\cap\LL_1'=\{x_j'\}. 
\end{displaymath}
In particular, this implies that $\cO_\lambda^\mu$ is not zero, so there must be a simple 
standard module in $\cO^\mu_\lambda$. Proposition~\ref{propanti}\eqref{propanti.1}, 
with $\aaa(w_0^{\mu_j})=1$, implies that such standard modules must have projective 
dimension $4$ in $\cO_\lambda$. The obtained projective dimensions for simple 
modules corresponding to $\LL_0\cup\LL_2\cup \LL_3$ imply that this simple standard 
module must be either $L(x_j\cdot\lambda)$ or $L(x_j'\cdot\lambda)$. As we have 
\begin{displaymath}
x'_j > x_j\mbox{ with }\len(x'_j)=\len(x_j)+1,\quad\mbox{for }1\le j\le n-2,
\end{displaymath}
we find that for those cases $L(x'_j\cdot\lambda)$ is a simple standard module in 
$\cO_\lambda^{\mu_j}$, while $L(x_j\cdot\lambda)$ is not. Their projective dimensions 
hence follow from Proposition~\ref{propanti}. As we also have
\begin{displaymath}
x'_{n-1} < x_{n-1}\mbox{ with }\len(x_{n-1})=\len(x'_{n-1})+1,
\end{displaymath}
Proposition~\ref{propanti} determines also the remaining projective dimensions.
\end{proof}

\begin{remark}{\rm
For $\fg=\mathfrak{sl}(n)$ and arbitrary $\lambda\in\Lambda^+_{{\rm int}}$, Proposition~\ref{maxpdL} allows to conclude that, 
for any $x\in X_\lambda$, we have
\begin{displaymath}
\aaa(w_0 x)+\aaa(w_0w_0^\lambda)\;\le\; \llll(x)\;\le\; \len(w_0 x)+\aaa(w_0w_0^\lambda). 
\end{displaymath}
The upper bound is known to be an equality when $W_\lambda=\{e\}$, whereas the lower bound 
is an equality when $W_\lambda$ is a maximal Coxeter subgroup of $W$, by Theorem~\ref{ThmHS}. 
The example of Proposition~\ref{Ex11n} deals with the case where $W_\lambda$ is very 
large but not maximal. We clearly see how $s_\lambda$ starts moving away from the 
lower bound towards the upper bound in the following way. The set $X_\lambda$ with 
pre-order $\le_L$ is no longer totally ordered (contrary to the case of maximal 
singularity by Proposition~\ref{propCIS}). For those two cells which are incomparable 
with respect to  the left order, the values of $\llll(x)$ can be higher than the lower bound, where 
precisely the length function and the right order come into play. This is made precise 
in the following corollary. This gives a unifying formula for the cases of maximal 
singularity and the singularity considered in this section. Note that it clearly does 
not hold for the regular case and hence is only a small step towards a general description.
}
\end{remark}

\begin{corollary}\label{examsec2-112}
Consider $\g=\mathfrak{sl}(n)$ with $\lambda\in\intdom$ either as in Proposition \ref{Ex11n} 
or such that $W_\lambda$ is a maximal Coxeter subgroup of $W$. For any $x\in X_\lambda$, we have
\begin{displaymath}
\llll(x)=\begin{cases}
\aaa(w_0 x)+\aaa(w_0w_0^\lambda),&\mbox{if }\;\len(x)=
\min\{\len(y)\,|  \,y\in \RR(x)\cap X_\lambda\};\\
\aaa(w_0 x)+\aaa(w_0w_0^\lambda)+1,&\mbox{otherwise}.
\end{cases} 
\end{displaymath}
\end{corollary}

\begin{proof}
When $W_\lambda$ is a maximal Coxeter subgroup, $\lambda$ corresponds to a hermitian 
symmetric pair, so this follows from the considerations in Section~\ref{secHS}.
 
So it suffices to consider the case in Proposition~\ref{Ex11n}. From the Robinson-Schensted 
correspondence  it follows that the only right KL relations on $\LL_1\cup\LL_1'$ are given 
by $x_j\sim_R x_j'$, for $1\le j\le n-1$. Note that this is consistent with the construction 
of parabolic subcategories in the proof of Proposition~\ref{propCIS}. With this observation, the 
result follows immediately from Proposition~\ref{propCIS}.
\end{proof}

\section{The blocks of category~$\cO$ for $\mathfrak{sl}(4)$}\label{blocks4}

\subsection{General description}\label{blocks4.001}

In this section we calculate the projective dimensions of  structural modules 
in all blocks of the parabolic category~$\cO$ for $\fg=\mathfrak{sl}(4)$. Theorems~\ref{thmsimstan} 
and~\ref{thminjtilt} imply that it suffices to consider standard and simple modules 
in category $\cO$. Note that, from \cite[Theorem 11]{SoergelD}, it follows that every 
non-integral block is equivalent to a block in category~$\cO$ for $\mathfrak{sl}(3)$ or 
$\mathfrak{sl}(2)$. These are already well-understood, see e.g. \cite[Section~6.2]{SHPO3}.
Also the regular blocks are understood by \cite{SHPO1}, so we are left with $\cO_\lambda$ for singular $\lambda$.

Up to equivalence, see \cite[Theorem 11]{SoergelD}, there are possibilities for $w_0^\lambda$, {\em viz.}:
$$s_3,\quad s_2,\quad s_1s_3\;\mbox{ and}\quad s_1s_2s_1.$$
We calculate the projective dimensions of standard and simple objects purely relying on KL combinatorics, by applying equation \eqref{pdDKLV}. Note that the third and fourth choice for $w_0^\lambda$ correspond to special cases of both Proposition \ref{asmall} and Theorem~\ref{ThmHS}, whereas the second choice is a special case of Proposition~\ref{Ex11n} (which however only determined $\llll$). Our alternative derivation in this section confirms those theoretical statements. The full knowledge of the KLV polynomials and Lusztig's canonical basis will also be essential to prove our result about Guichardet categories as explained below.

We will find that all blocks are almost monotone. However, $\cO_\lambda$ for 
$w_0^\lambda=s_3$ is not weakly monotone. Therefore
Corollary~\ref{almostmonotoneblock}\eqref{almostmonotoneblock.2}  does 
not guarantee that this block is weakly Guichardet. We prove 
explicitly that the block is not weakly Guichardet, from which we obtain the following conclusion.

\begin{theorem}\label{counterex}
Category~$\cO$ for $\fg=\mathfrak{sl}(4)$ contains an integral block which is not weakly Guichardet.
\end{theorem}

In this section we identify $\Lambda_{{\rm int}}$ with $\mZ^4$ by mapping $\kappa$ to $(\langle \kappa+\rho,\epsilon_i\rangle)_{1\le i\le 4}$. As usual, the generators of the Weyl group are denoted by $s_i,$ $i\in\{1,2,3\}$ with $s_i$ the reflection corresponding to $\epsilon_i-\epsilon_{i+1}$.

\subsection{The case $w_0^\lambda=s_3$}

\label{weirdsl4}\label{APsl4.1}
Note that in this case we have
\begin{displaymath}
\mathbf{L}(w_0^\lambda)=\{s_3,s_2s_3,s_1s_2s_3\}. 
\end{displaymath}

We calculate algorithmically Lusztig's canonical basis. For this, 
we follow the conventions and notations of \cite[Section~3]{Brundan}. 
The canonical basis is given by:
\begin{eqnarray*}
\dot{b}_{2100}&=&\dot{v}_{2100}\\
\dot{b}_{1200}&=&\dot{v}_{1200}+q\dot{v}_{2100}\\
\dot{b}_{2010}&=&\dot{v}_{2010}+q\dot{v}_{2100}\\
\dot{b}_{1020}&=&\dv_{1020}+q(\dv_{1200}+\dv_{2010})+q^2\dv_{2100}\\
\db_{0210}&=&\dv_{0210}+q(\dv_{2010}+\dv_{1200})+q^2\dv_{2100}\\
\db_{2001}&=&\dv_{2001}+q\dv_{2010}+q^2\dv_{2100}\\
\db_{1002}&=&\dv_{1002}+q(\dv_{1020}+\dv_{2001})+q^2(\dv_{2010}+\dv_{1200})+q^3\dv_{2100}\\
\db_{0120}&=&\dv_{0120}+q(\dv_{1020}+\dv_{0210})+q^2(\dv_{2010}+\dv_{1200})+q^3\dv_{2100}\\
\db_{0201}&=&\dv_{0201}+q(\dv_{2001}+\dv_{0210})+q^2(\dv_{2010}+\dv_{1200})+q^3\dv_{2100}\\
\db_{0102}&=&\dv_{0102}+q(\dv_{1002}+\dv_{0120}+\dv_{0201}+\dv_{1200})+q^2(\dv_{1020}+\dv_{0210}\\
&&+\dv_{2001}+\dv_{2100})+q^3(\dv_{2010}+\dv_{1200})+q^4\dv_{2100}\\
\db_{0021}&=&\dv_{0021}+q(\dv_{0120}+\dv_{0201})+q^2(\dv_{1020}+\dv_{0210}+\dv_{2001})\\
&&+q^3(\dv_{1200}+\dv_{2010})+q^4\dv_{2100}\\
\db_{0012}&=&\dv_{0012}+q(\dv_{0102}+\dv_{0021})+q^2(\dv_{1002}+\dv_{0120}+\dv_{0201})+q^3(\dv_{1020}\\
&&+\dv_{0210}+\dv_{2001})+q^4(\dv_{1200}+\dv_{2010})+q^5\dv_{2100})
\end{eqnarray*}

Consequently, the KLV polynomials are described by the inversion of the above triangular 
transformation matrix.
\begin{eqnarray*}
\dot{v}_{2100}&=&\dot{b}_{2100}\\
\dot{v}_{1200}&=&\dot{b}_{1200}-q\dot{b}_{2100}\\
\dot{v}_{2010}&=&\dot{b}_{2010}-q\dot{b}_{2100}\\
\dot{v}_{1020}&=&\db_{1020}-q(\db_{1200}+\db_{2010})+q^2\db_{2100}\\
\dv_{0210}&=&\db_{0210}-q(\db_{2010}+\db_{1200})+q^2\db_{2100}\\
\dv_{2001}&=&\db_{2001}-q\db_{2010}\\
\dv_{1002}&=&\db_{1002}-q(\db_{1020}+\db_{2001})+q^2\db_{2010}\\
\dv_{0120}&=&\db_{0120}-q(\db_{1020}+\db_{0210})+q^2(\db_{2010}+\db_{1200})-q^3\db_{2100}\\
\dv_{0201}&=&\db_{0201}-q(\db_{2001}+\db_{0210})+q^2\db_{2010}\\
\dv_{0102}&=&\db_{0102}-q(\db_{1002}+\db_{0120}+\db_{0201}+\db_{1200})+q^2(\db_{1020}+\db_{0210}+\db_{2001})-q^3\db_{2010}\\
\dv_{0021}&=&\db_{0021}-q(\db_{0120}+\db_{0201})+q^2\db_{0210}\\
\dv_{0012}&=&\db_{0012}-q(\db_{0102}+\db_{0021})+q^2(\db_{1200}+\db_{0120}+\db_{0201})-q^3\db_{0210}
\end{eqnarray*}

This gives the projective dimensions of simple and standard modules 
using Equation~\eqref{pdDKLV}. These are given in 
the following table, where we also denote the corresponding elements of $X_\lambda$.
\begin{eqnarray*}
\pd\, \Delta(2100)=0&&\pd\, L(2100)=6\qquad s_3\\
\pd\, \Delta(1200)=1&&\pd\, L(1200)=5\qquad s_1s_3\\
\pd\, \Delta(2010)=1&&\pd\, L(2010)=6\qquad s_2s_3\\
\pd\, \Delta(1020)=2&&\pd\, L(1020)=5\qquad s_2s_1s_3\\
\pd\, \Delta(0210)=2&&\pd\, L(0210)=6\qquad s_1s_2s_3\\
\pd\, \Delta(2001)=1&&\pd\, L(2001)=5\qquad s_3s_2s_3\\
\pd\, \Delta(1002)=2&&\pd\, L(1002)=4\qquad s_3s_2s_1s_3\\
\pd\, \Delta(0120)=3&&\pd\, L(0120)=5\qquad s_1s_2s_1s_3\\
\pd\, \Delta(0201)=2&&\pd\, L(0201)=5\qquad s_3s_1s_2s_3\\
\pd\, \Delta(0102)=3&&\pd\, L(0102)=4\qquad s_3s_1s_2s_1s_3\\
\pd\, \Delta(0021)=2&&\pd\, L(0021)=4\qquad s_2s_3s_1s_2s_3\\
\pd\, \Delta(0012)=3&&\pd\, L(0012)=3\qquad s_3s_2s_3s_1s_2s_3
\end{eqnarray*}
It follows that this block is almost monotone, but not weakly monotone.

To write out the left cells, we use the notation of Subsection \ref{examsec2}. This gives:
\begin{eqnarray*}
\mathbf{L}_1'= \{(1002),(0102),(0120)\},&&\mathbf{L}_2=\{(1200),(1020)\},\\
\mathbf{L}_1=\{(2001),(0201),(0021)\},&&\mathbf{L}_3=\{(2100),(2010),(0210)\}.
\end{eqnarray*}

The symmetrised $\Ext^1$-quiver hence has the following form (each unoriented arrow 
in this quiver corresponds to two arrows in the usual  $\Ext^1$-quiver going in opposite directions), 
where we also mark, on the side, the projective  dimension of each simple:
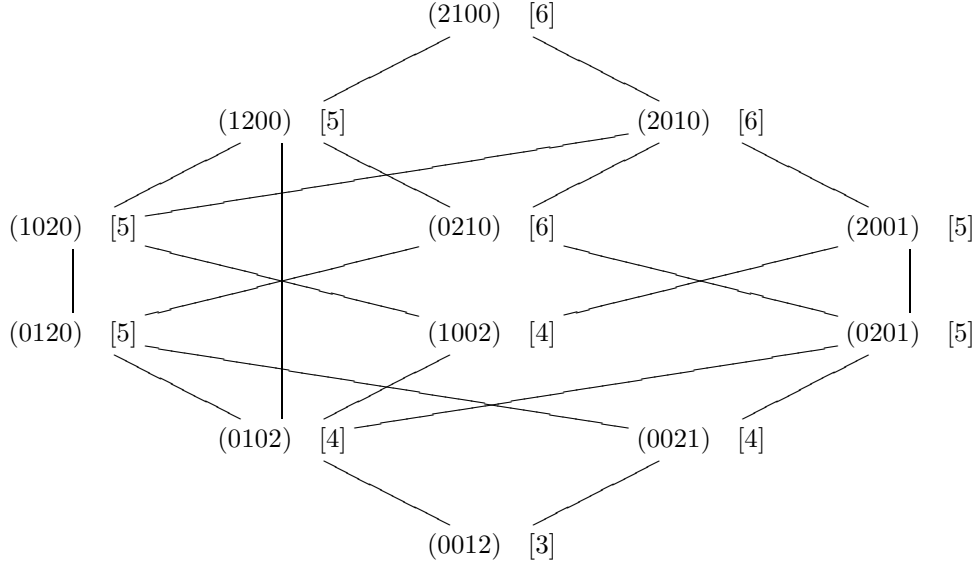
\begin{displaymath}
    \xymatrix{
&&(2100)\quad [6]&&&\\
&(1200)\quad [5]\ar@{-}[ur]&&(2010)\quad [6]\ar@{-}[ul]&\\
(1020)\quad [5]\ar@{-}[ur]\ar@{-}[urrr]&&(0210)\quad [6]\ar@{-}[ur]\ar@{-}[ul]&&(2001)\quad [5]\ar@{-}[ul]\\
(0120)\quad [5]\ar@{-}[u]\ar@{-}[urr]&&(1002)\quad [4]\ar@{-}[ull]\ar@{-}[urr]&&(0201)\quad [5]\ar@{-}[u]\ar@{-}[ull]\\
&(0102)\quad [4]\ar@{-}[ul]\ar@{-}[ur]\ar@{-}[uuu]\ar@{-}[urrr]&&(0021)\quad [4]\ar@{-}[ulll]\ar@{-}[ur]&\\
&&(0012)\quad [3]\ar@{-}[ul]\ar@{-}[ur]&&&\\
   }
\end{displaymath}
Note that this $\Ext^1$-quiver can be embedded in the one in \cite[Appendix~A]{Stroppel2}, 
as described in \cite[Proposition~3.1]{CM4}. In particular, $(2100)$ gets mapped to 
$2$ and $(0012)$ to $24$.

\begin{proof}[Proof of Theorem~\ref{counterex}]
Consider $\kappa=s_1s_3\cdot\lambda$, represented by $(1200)$. The Serre subcategory of 
$\cO_\lambda$ generated by $L(\nu)$ for $\nu\le \kappa$ is extension full by~\cite{CPS}. 
We denote this subcategory by $\cA$ and also use $L=L(1200)$ and $\Delta=\Delta(1200)$. 
By using the results on the projective dimensions and the $\Ext^1$-quiver, it follows 
that the Serre subcategory generated by the simple modules in~$\cA$ that are not 
isomorphic to $L(0210)$ or $L(0201)$, is a saturated initial segment in~$\cO_\lambda$, 
which we denote by $\cI$. It suffices to prove that $\cI$ is not extension full in~$\cA$.

It follows immediately that $\Delta$ is the projective cover of $L$ in~$\cA$.  Take $K$ 
equal to the smallest submodule of $\Delta$ which contains all occurrences of $L(0210)$ 
and $L(0201)$ (they appear once each by the Lusztig's canonical basis and the BGG reciprocity). It follows 
from standard homological arguments that the module $P_{\cI}$, defined by the short exact sequence
\begin{displaymath}
0\to K\to \Delta\to P_{\cI}\to 0,
\end{displaymath}
is an indecomposable projective cover of $L$ in~$\cI$. From Lusztig's canonical basis it follows furthermore that $L(0210)$ appears in the top of $K$, so we can define $M$ by the short exact sequence
$M\hookrightarrow K\tto L(0210)$. As $[\Delta:L]=1$, we find that the first term in the 
exact sequence
\begin{displaymath}
\Hom_{\cA}(M,L)\to \Ext^1_{\cA}(L(0210),L)\to \Ext^1_{\cA}(K,L) 
\end{displaymath}
is zero. However, $\Ext^1_{\cA}(L(0210),L)\cong \Ext^1_{\cO_\lambda}(L(0210),L)$ is 
non-zero by the $\Ext^1$-quiver, so $\Ext^1_{\cA}(K,L)\not=0$. As $\Delta$ is 
projective in~$\cA$, the exact sequence
\begin{displaymath}
\Ext^1_{\cA}(\Delta,L)\to \Ext^1_{\cA}(K,L)\to \Ext^2_{\cA}(P_{\cI},L)\to \Ext^2_{\cA}(\Delta,L), 
\end{displaymath}
then yields 
\begin{displaymath}
\Ext^2_{\cA}(P_{\cI},L)\not= 0 =\Ext^2_{\cI}(P_{\cI},L), 
\end{displaymath}
concluding the proof.
\end{proof}

\subsection{The case $w_0^\lambda=s_2$}\label{s2}

In this case, we have 
\begin{displaymath}
\mathbf{L}(s_2)=\{s_2,s_1s_2,s_3s_2\}. 
\end{displaymath}
As in the previous subsection, we can compute  the following KLV polynomials:
\begin{eqnarray*}
\dv_{2110}&=&\db_{2110}\\
\dv_{1210}&=&\db_{1210}-q\db_{2110}\\
\dv_{2101}&=&\db_{2101}-q\db_{2110}\\
\dv_{1201}&=&\db_{1201}-q(\db_{1210}+\db_{2101})+q^2\db_{2110}\\
\dv_{1120}&=&\db_{1120}-q\db_{1210}\\
\dv_{2011}&=&\db_{2011}-q\db_{2101}\\
\dv_{1021}&=&\db_{1021}-q(\db_{1120}+\db_{1201}+\db_{2011}+\db_{2011})+q^2(\db_{1210}+\db_{2101})\\
\dv_{1102}&=&\db_{1102}-q(\db_{1120}+\db_{1201})+q^2\db_{1210}\\
\dv_{0211}&=&\db_{0211}-q(\db_{2011}+\db_{1201})+q^2\db_{2101}\\
\dv_{0121}&=&\db_{0121}-q(\db_{0211}+\db_{1021})+q^2(\db_{1201}\db_{2011}+\db_{2011})-q^3\db_{2101}\\
\dv_{1012}&=&\db_{1012}-q(\db_{1102}+\db_{1021})+q^2(\db_{1201}+\db_{1120}+\db_{2110})-q^3\db_{1210}\\
\dv_{0112}&=&\db_{0112}-q(\db_{1012}+\db_{0121})+q^2\db_{1021})-q^3\db_{2110}.
\end{eqnarray*}

These yield the following projective dimensions.
\begin{eqnarray*}
\pd\,  \Delta(2110)=0 &&\pd\,  L(2110)=6\quad\qquad s_2\\
\pd\,  \Delta(1210)=1 &&\pd\,  L(1210)=6\quad\qquad s_1s_2\\
\pd\,  \Delta(2101)=1 &&\pd\,  L(2101)=6\quad\qquad s_3s_2\\
\pd\,  \Delta(1201)=2 &&\pd\,  L(1201)=5\quad\qquad s_1s_3s_2\\
\pd\,  \Delta(1120)=1 &&\pd\,  L(1120)=5\quad\qquad s_2s_1s_2\\
\pd\,  \Delta(2011)=1 &&\pd\,  L(2011)=5\quad\qquad s_2s_3s_2\\
\pd\,  \Delta(1021)=2 &&\pd\,  L(1021)=5\quad\qquad s_2s_1s_3s_2\\
\pd\,  \Delta(1102)=2 &&\pd\,  L(1102)=4\quad\qquad s_3s_2s_1s_2\\
\pd\,  \Delta(0211)=2 &&\pd\,  L(0211)=4\quad\qquad s_1s_2s_3s_2\\
\pd\,  \Delta(0121)=3 &&\pd\,  L(0121)=4\quad\qquad s_1s_2s_1s_3s_2\\
\pd\,  \Delta(1012)=3 &&\pd\,  L(1012)=4\quad\qquad s_3s_2s_1s_3s_2\\
\pd\,  \Delta(0112)=3 &&\pd\,  L(0112)=3\quad\qquad s_1s_3s_2s_1s_3s_2
\end{eqnarray*}
This block is thus weakly monotone.

\subsection{The case $w_0^\lambda=s_1s_3$}

In this case, we have
\begin{displaymath}
\mathbf{L}(s_1s_3)=\{s_1s_3,s_2s_1s_3\}
\end{displaymath}
and the KLV polynomials are given by:
\begin{eqnarray*}
\dv_{1100}&=&\db_{1100}\\
\dv_{1010}&=&\db_{1010}-q\db_{1100}\\
\dv_{0110}&=&\db_{0110}-q\db_{1010}\\
\dv_{1001}&=&\db_{1001}-q\db_{1010}\\
\dv_{0101}&=&\db_{0101}-q(\db_{0110}+\db_{1001}+\db_{1100})+q^2\db_{1010}\\
\db_{0011}&=&\db_{0011}-q\db_{0101}+q^2\db_{1100}.
\end{eqnarray*}

These yield the following projective dimensions.
\begin{eqnarray*}
\pd\,  \Delta(1100)=0&&\pd\,  L(1100)=4\qquad\quad s_1s_3\\
\pd\,  \Delta(1010)=1&&\pd\,  L(1010)=4\qquad\quad s_2s_1s_3\\
\pd\,  \Delta(0110)=1&&\pd\,  L(0110)=3\qquad\quad s_1s_2s_1s_3\\
\pd\,  \Delta(1001)=1&&\pd\,  L(1001)=3\qquad\quad s_3s_2s_1s_3\\
\pd\,  \Delta(0101)=2&&\pd\,  L(0101)=3\qquad\quad s_1s_3s_2s_1s_3\\
\pd\,  \Delta(0011)=2&&\pd\,  L(0011)=2\qquad\quad s_2s_1s_3s_2s_1s_3
\end{eqnarray*}
This block is thus weakly monotone.

\label{s1s3}

%\subsection{The case $\{s_1s_2s_3\}$}

\section{Lie superalgebras}\label{sectionsuper}

In this section we obtain the projective dimension of arbitrary injective modules in the BGG category for classical Lie superalgebras.

Consider a simple
classical Lie superalgebra $\fg$, see \cite{CW, Musson}, with an arbitrary choice of 
positive roots $\Delta^+$. To make a distinction between notation for the Lie 
superalgebra $\fg$ and its underlying Lie algebra $\fg_{\oa}$, we denote the BGG 
category for~$\fg$ by~$\sO$, simple modules by $\mathscr{L}(\kappa)$, for 
$\kappa\in\fh_{\oa}^\ast$, and their indecomposable injective envelope in~$\sO$ 
by $\mathscr{I}(\kappa)$, whereas we maintain the same notation for the Lie algebra 
$\fg_{\oa}$ as before. However, by $2\rho=2\rho_{\oa}-2\rho_{\ob}$ we now mean the sum of all even positive roots minus
the sum of all odd positive roots. Note that the functors $\Res$ and $\Ind$ induce exact functors 
between $\cO$ and $\sO$ preserving projective and injective modules.

First we prove a generalisation of \cite[Theorem~6.1(iii)]{CS}
\begin{proposition}\label{pdIsuper}
For any $\kappa\in\fh_{\oa}^\ast$, let $\nu\in\fh_{\oa}^\ast$ be such that 
$L(\nu)$ appears in the socle or top of $\Res\, \mathscr{L}(\kappa)$ (up to parity shift).
Then we have
\begin{displaymath}
\pd_{\sO}\mathscr{I}(\kappa)=\pd_{\cO} I(\nu). 
\end{displaymath}
\end{proposition}
Note that $\Res \mathscr{L}(\kappa)$ is self-dual, hence its top and socle are isomorphic. Moreover $L(\kappa)$ is in the top of $\Res \mathscr{L}(\kappa)$ if and only if it is a direct summand.
\begin{proof}
For simplicity, in this proof we ignore the parity shifts of all involved modules.
Assume that $L(\nu)\hookrightarrow\Res \mathscr{L}(\kappa)$. By adjunction, we have a morphism 
$\Ind L(\nu)\tto \mathscr{L}(\kappa)$. So, we have, in particular,
\begin{displaymath}
\Hom_{\fg_{\oa}}(\Res \mathscr{L}(\kappa),I(\nu))=[\Res \mathscr{L}(\kappa):L(\nu)]\not=0 
\end{displaymath}
and
\begin{displaymath}
\Hom_{\fg}(\Ind L(\nu),\mathscr{I}(\kappa))= [\Ind L(\nu): \mathscr{L}(\kappa)]\not=0. 
\end{displaymath}
Applying adjunction and the fact that $\Res$ and $\Ind$ preserve injective modules 
to these statements, yields inclusions 
\begin{displaymath}
\mathscr{I}(\kappa)\hookrightarrow\Ind I(\nu)\quad\mbox{and}\quad I(\nu)\hookrightarrow \Res \mathscr{I}(\kappa). 
\end{displaymath}
Since the exact functors $\Res$ and $\Ind$ map projective resolutions to projective 
resolutions, we find 
\begin{displaymath}
\pd_{\cO} I(\nu)\le \pd_{\sO} \mathscr{I}(\kappa)\le \pd_{\cO} I(\nu). 
\end{displaymath}
The claim follows.
\end{proof}

For $\kappa\in \Lambda_{{\rm int}}$, we denote by $\mathscr{D}(\kappa)\subset \Delta_{\oa}^+$ 
the set of positive roots $\alpha$ for which $\mathscr{L}(\kappa)$ is $\alpha$-free. 
We denote the corresponding set for the $\fg_{\oa}$-module $L(\kappa)$ by $D(\kappa)$. 
We also define $x^{\mathscr{D}}_\kappa$, respectively $x^D_{\kappa}$, 
as the unique elements of the Weyl group for which we have
\begin{displaymath}
\{\alpha\in\Delta^+_{\oa}\,|\, x^{\mathscr{D}}_\kappa(\alpha)\in \Delta^-_{\oa}\}\,=\, \mathscr{D}(\kappa), 
\quad
\{\alpha\in\Delta^+_{\oa}\,|\, x^{{D}}_\kappa(\alpha)\in \Delta^-_{\oa}\}\,=\, {D}(\kappa). 
\end{displaymath}
Note that $x^D_\kappa$ is equivalently defined as longest element of the Weyl group for which we have
$\kappa+\rho_{\oa}\,\in \,x^D_\kappa(\intdom+\rho_{\oa})$.

\begin{theorem}\label{Thmsup}
For $\lambda\in \Lambda_{{\rm int}}$,  we have
$$\pd_{\sO} \mathscr{I}(\lambda)=2\aaa(w_0x^{\mathscr{D}}_\lambda).$$
\end{theorem}

\begin{proof}
In case $\fg$ is even, {\it i.e.} a reductive Lie algebra, this is just a reformulation of 
Theorem~\ref{thminjtilt}\eqref{thminjtilt.2} for $\mu=0$. The extension of the characterisation to 
superalgebras thus follows from Proposition~\ref{pdIsuper}, as the property that $L(\nu)$ appears in the socle or top of $\Res\mathscr{L}(\lambda)$ implies that $D(\nu)=\mathscr{D}(\lambda)$ and hence $x^D_\nu=x^{\mathscr{D}}_{\lambda}$.
\end{proof}
 
For any $\lambda\in\Lambda_{{\rm int}}$, we set $[\lambda]\subset\Lambda_{{\rm int}}$ 
equal to the set of all $\mu$ of the form
\begin{displaymath}
\mu=w(\lambda+\rho+k_1\gamma_1+\cdots+k_n\gamma_n)-\rho, 
\end{displaymath}
where $k_i\in\mZ$  and $\{\gamma_i\}$ is a maximal set of mutually orthogonal, 
linearly independent isotropic roots orthogonal to $\lambda+\rho$. This number 
$n$ is known as the {\em degree of atypicality} of $\lambda$ and is, clearly, a constant for any $\mu\in[\lambda]$.
 
\begin{lemma}\label{superblocks}
The indecomposable block in $\sO$ containing $\mathscr{L}(\lambda)$ is the Serre subcategory 
of $\cO$ generated by all $\mathscr{L}(\mu)$ with $\mu\in[\lambda]$. We denote this block by $\sO_{[\lambda]}$.
\end{lemma}

\begin{proof}
According to \cite[Lemma~2.1]{Ser}, the set $[\lambda]$ is precisely the set of integral weights 
$\mu$ such that $\mathscr{L}(\mu)$ admits the same central character as 
$\mathscr{L}(\lambda)$. It hence suffices to show that any such $\mathscr{L}(\mu)$ is in the 
same indecomposable block as $\mathscr{L}(\lambda)$. This is a standard exercise, which 
can be carried out by the methods in the proof of \cite[Theorem~3.12]{CMW} and Serganova's
technique of odd reflections, see e.g.\cite{Musson} or \cite[Lemma~2.3]{CM1}.
\end{proof}

As shown in the proof of \cite[Theorem 3]{Rome}, the finitistic dimension of $\sO_\lambda$ 
is equal to the maximal projective dimension of an injective module in~$\sO_\lambda$ and is
subsequently always finite. 
Theorem \ref{Thmsup} and Lemma \ref{superblocks} thus determine implicitly these finitistic dimensions of blocks.
Obtaining a closed expression would require some further work. However, we immediately have the following 
consequence, where we use the concept of generic weights from \cite[Definition~7.1]{CM1}.
\begin{corollary}
If $[\lambda]$ contains a generic weight, then
$$\fd \sO_{[\lambda]}=2\len(w_0).$$
\end{corollary}
\begin{proof}
Theorem \ref{Thmsup} implies that $\fd \sO_{[\lambda]}\le 2\len(w_0)$ for any block. The assumption and the remark before \cite[Lemma~2.2]{CM1} imply that $[\lambda]$ contains a $\nu$ which is dominant (for both the $\rho$-shifted and for the $\rho_{\oa}$-shifted action) and which is generic as well. As $\nu$ is generic it is in particular regular, so $\pd_{\cO}I(\nu)=2\len(w_0)$. By \cite[Lemma~2.2]{CM1}, $L(\nu)$ appears in the top of $\Res\mathscr{L}(\nu)$, so Proposition~\ref{pdIsuper} concludes the proof.
\end{proof}

\section{Open questions}\label{secOpen}

The following questions naturally arise from the results in this paper:
\begin{enumerate}[(I)]
\item Is there a direct argument which explains the observation in Remark~\ref{remRpd}?
\item Does the formula 
\begin{displaymath}
\mathtt{s}_\lambda(x)\ge \aaa(w_0 x)+\aaa(w_0w_0^\lambda)
\end{displaymath}
of Proposition~\ref{ineq} hold outside type $A$ as well?
\item Is it possible to generalise Theorem~\ref{ThmHS} in the following way: does the formula
$\llll(x)=\aaa(w_0 x)+\aaa(w_0w_0^\lambda)$ hold for arbitrary $\lambda\in\intdom$ 
such that $W_\lambda$ is a maximal Coxeter subgroup of $W$? Note that the question 
whether the equality $\dddd(x)=\aaa( xw_0^\lambda$) holds for arbitrary maximal Coxeter subgroups 
has the negative answer by~\cite[Section 8]{Collin}.
\item What is the finitistic dimension of the category $\cO_0^{\hat{\mathbf{R}}}$ 
for a fixed arbitrary right cell $\mathbf{R}$? Do injective modules in $\cO_0^{\hat{\mathbf{R}}}$ 
always have finite projective dimension?
\item Is it possible to construct an explicit combinatorial formula for the bijection 
$\psi^\mu:\RR(w_0^\mu w_0)\to w_0^\mu\RR(w_0^\mu)$ in Remark~\ref{remarkwfunction}?
\item What is the subset $U_\lambda^\mu\subset X_\lambda^\mu$ of all $x$ for which 
$L(x\cdot\lambda)$ is standard in $\cO_\lambda^\mu$? Note that we have 
$U^\mu=\{w_0^\mu w_0\}$ and $U_\lambda=\{w_0\}$. In general, $U_\lambda^\mu$ will 
not consist of one element, in particular, since $\cO_\lambda^\mu$ can decompose 
into a non-trivial direct sum. Does every summand contain a unique simple standard module?
\item The diagram in Theorem~\ref{monA} can be applied to show that 
$\textswab{C}_1(\mu,\lambda)\Rightarrow \textswab{S}_1(\mu,\lambda)$. Does the implication 
in the other direction also hold?
\end{enumerate}

\noindent
{\bf Acknowledgment.}
KC is Postdoctoral Fellow of the Research Foundation - Flanders (FWO).
VM is partially supported by the Swedish Research Council.

\vspace{2mm}

\noindent
KC: School of Mathematics and Statistics, University of Sydney, NSW 2006, Australia \hspace{4mm} and \hspace{4mm}
Department of Mathematical Analysis, Ghent University, Krijgs-laan 281, 9000 Gent, Belgium;
E-mail: {\tt Coulembier@cage.ugent.be} 
\vspace{2mm}

\noindent
VM: Department of Mathematics, University of Uppsala, Box 480, SE-75106, Uppsala, Sweden;
E-mail: {\tt  mazor@math.uu.se}
\date{}

\end{document}